\documentclass[11pt]{article}
\usepackage{amsfonts}

\usepackage{graphics}
\usepackage{indentfirst}
\usepackage{cite}
\usepackage{latexsym}
\usepackage{amsmath,amsthm}
\usepackage{amssymb}
\usepackage[dvips]{epsfig}
\usepackage{amscd}
\usepackage{mathrsfs}
\newtheorem{theorem}{Theorem}[section]
\newtheorem{remark}{Remark}[section]

\newtheorem{definition}{Definition}[section]
\newtheorem{lemma}[theorem]{Lemma}

\newcommand{\n}{\rho}

\newcommand{\lm}{\lambda}

\renewcommand{\div}{ {\rm div }  }

\newcommand{\na}{\nabla }

\newcommand{\pa}{\partial}

\newcommand{\bt}{\begin{theorem}}
\newcommand{\bl}{\begin{lemma}}
\newcommand{\el}{\end{lemma}}
\newcommand{\et}{\end{theorem}}
\newcommand{\ga}{\gamma}

\newcommand{\OM}{\Omega}

\newcommand{\curl}{{\rm curl} }

\newcommand{\de}{\delta}
\newcommand{\ve}{\varepsilon}
\newcommand{\la}{\label}

\newcommand{\ol}{\overline}

\newcommand{\bn}{\begin{eqnarray}}
\newcommand{\en}{\end{eqnarray}}
\newcommand{\bnn}{\begin{eqnarray*}}
\newcommand{\enn}{\end{eqnarray*}}

\newcommand{\bnnn}{\begin{eqnarray*}}
\newcommand{\ennn}{\end{eqnarray*}}
\newcommand{\ben}{\begin{enumerate}}
\newcommand{\een}{\end{enumerate}}

\newcommand{\ba}{\begin{aligned}}
\newcommand{\ea}{\end{aligned}}
\newcommand{\be}{\begin{equation}}
\newcommand{\ee}{\end{equation}}

\def\p{\partial}
\def\norm[#1]#2{\|#2\|_{#1}}

\def\lam{\lambda}
\def\ep{\varepsilon}

\def\o{\omega}

\def\rr{\mathbb{R}^2}
\def\rrr{\mathbb{R}^3}

\makeatletter      % '@' is now a normal "letter" for TeX
\@addtoreset{equation}{section}
\makeatother       % '@' is restored as a "non-letter" character for TeX

\title{Global Well-Posedness of Strong Solutions to the Two-Dimensional Compressible Nematic Liquid Crystal Flows with Large Initial Data and Vacuum}

\date{}

\author{$\text{Qinghao L{\small EI}}^{a,b}, \text{Lu W{\small{ANG}}}^{b}\thanks{Email addresses:  leiqinghao22@mails.ucas.ac.cn (Q. H. Lei), wlu1130@163.com (L. Wang) }$\\
a. School of Mathematical Sciences,\\ University of Chinese Academy of Sciences,
Beijing 100049, P. R. China;\\
b. Institute of Applied Mathematics,\\ Academy of Mathematics and Systems Science, \\
Chinese Academy of Sciences, Beijing 100190, P. R. China}

\begin{document}
\maketitle

\begin{abstract}
We consider the two-dimensional barotropic compressible nematic liquid crystal flows with a planar director in the whole space or in the half-space.
Under the assumptions that the shear viscosity is a positive constant and that the bulk viscosity is given by $\lambda(\rho)=\rho^\beta$ with $\beta>4/3$, we establish the global existence and uniqueness of strong solutions.
It should be mentioned that our results are obtained without any restrictions on the size of the initial data and allow for the presence of vacuum.
In particular, the initial orientation field is not required to satisfy any geometric angle condition. \\
\par\textbf{Keywords:} Compressible nematic liquid crystal flows;
Global strong solutions; Large initial data; Vacuum
\par\textbf{2020 Mathematics Subject Classification:} 76A15, 76N10, 35Q35.
\end{abstract}

\section{Introduction and main results}
We study the two-dimensional compressible nematic liquid crystal flows
which read as follows:
\be\ba\la{nlckv}
\begin{cases}
  \rho_t + \div(\rho u) = 0, \\
(\n u)_t + \div(\n u\otimes u) -\mu \Delta u 
-\na ( (\mu + \lm) \div u) + \na P = - \na d \cdot \Delta d, \\
d_t + u \cdot \na d = \Delta d + |\na d|^2 d, \\
|d|=1,
\end{cases}
\ea\ee
where $t \ge 0$ is time, $x \in \OM \subset \rr$ is the spatial coordinate,
$\n=\n(x,t)$, $u(x,t)=(u^1(x,t),u^2(x,t))$ and $d=(d^1(x,t),d^2(x,t))$
represent the density, velocity, and macroscopic molecular orientation of the liquid crystal material of the compressible flow, respectively.
The pressure $P$ is given by
\be\ba\la{i1}
P=R \n^\ga,
\ea\ee
with constants $R>0$ and $\ga>1$.
The shear viscosity coefficient $\mu$ and the bulk viscosity coefficient $\lam$ satisfy the following hypothesis:
\be\la{i2}
0<\mu = \text{constant},\quad \lam(\n)=b \n^\beta,
\ee
where $b$ and $\beta$ are positive constants.
Without loss of generality, we assume that $R=b=1$.

The system is supplemented with the given initial data
\be\la{i30}
\n(x,0)=\n_0(x),\quad \n u(x,0)= \n_0 u_0(x), \quad d(x,0)=d_0(x),\quad x\in \OM.
\ee
In this paper, we mainly investigate the following two settings:

(1) Cauchy problem: $\OM=\rr$ with the far-field behavior
\be\la{bjtj2}\ba
(\n,u,d)(x,t) \to (\tilde{\n},0,\mathbf{1}) \ \text{ as } |x| \to \infty, \  t\ge 0.
\ea\ee

(2) Half-space problem: $\OM = \rr_+ = \{x \in \rr : x_2>0 \}$
subject to the Navier-slip and Neumann boundary conditions:
\be\la{bkjbjtj1}\ba
u \cdot n = 0,\quad \curl u = 0, \quad \frac{\p d}{\p n}=0 \ \text{ on } \ \p \OM,
\ea\ee
where $n=(n_1,n_2)$ is the unit outer normal vector to $\partial \Omega$ and $n^\bot\triangleq (n_2,-n_1)$ denotes the corresponding unit tangential vector.
In addition, the following far-field condition is imposed:
\be\la{bkjbjtj2}\ba
(\n,u,d)(x,t) \to (\tilde{\n},0,\mathbf{1}) \ \text{ as } |x| \to \infty, \  t\ge 0,
\ea\ee
where $\tilde{\n} \ge 0$ is a constant and $\mathbf{1}$ is a given unit vector.

The system (\ref{nlckv}) is a simplified hydrodynamic model describing liquid crystal flows, originally proposed by Ericksen \cite{EJL} and Leslie \cite{LFM} in the 1960s.
Mathematically, it is a strongly coupled system consisting of the compressible Navier-Stokes equations and the harmonic map heat flow.
In particular, when $d$ is a constant unit vector, the system (\ref{nlckv})--(\ref{i30}) reduces to the compressible Navier-Stokes system.
The strong solvability of the multidimensional compressible Navier-Stokes equations with constant viscosity coefficients has been the subject of extensive study.
Matsumura-Nishida \cite{MN1} obtained the first global classical solutions for initial data close to a non-vacuum equilibrium in the $H^s$-norm.
Later, Hoff \cite{H1,H3} studied the problem for discontinuous initial data.
For arbitrarily large initial data, Lions \cite{L2} (see also Feireisl \cite{F,FNP}) established the global existence of finite-energy weak solutions, under the assumption that the adiabatic exponent $\ga$ is suitably large.
More recently, Huang-Li-Xin \cite{HLX2} and Li-Xin \cite{LX2} proved the global existence and uniqueness of classical solutions to the three-dimensional and two-dimensional Cauchy problems for initial data with small total energy but possibly large oscillations and vacuum.
Cai-Li \cite{CL} later extended these results to general bounded domains with slip boundary conditions for the velocity field.
By contrast, relatively few results are available on the global existence of strong solutions without size restrictions on the initial data.
When $d$ is a constant unit vector, the system (\ref{nlckv})--(\ref{i30}) corresponds to the model introduced by Vaigant-Kazhikhov \cite{VK}, who proved the existence and uniqueness of global strong solutions in rectangular domains for large initial data with density away from vacuum, provided that $\beta>3$.
Later, Jiu-Wang-Xin \cite{JWX1} extended this result to initial data allowing vacuum in a periodic domain.
Using new techniques based on commutator theory and blow-up criteria, Huang-Li \cite{HL2,HL3} (see also \cite{JWX2}) relaxed the condition $\beta>3$ to $\beta>\frac{4}{3}$ in periodic domains or in the whole space, allowing the density to vanish.
More recently, Fan-Li-Li \cite{FLL} proved the global existence of strong solutions for $\beta>\frac{4}{3}$ in general bounded simply connected domains where the velocity field satisfies the Navier-slip boundary conditions.
Subsequently, Fan-Li-Wang \cite{FLW} obtained a time-uniform upper bound for the density and established the exponential decay of global strong solutions in both periodic and bounded simply connected domains, provided that $\beta>\frac{4}{3}$.

For the compressible nematic liquid crystal system (\ref{nlckv}) with constant viscosity coefficients, there is also a substantial body of work on the global existence of solutions.
In particular, global weak solutions for large initial data were obtained in \cite{JJW1,JJW2,LLW} under a geometric angle condition on the initial orientation field.
Huang-Wang-Wen \cite{HWW} established the local existence and uniqueness of strong solutions, allowing the initial density to vanish in open subsets.
For the three-dimensional Cauchy problem, Hu-Wu \cite{HW} proved the global existence and uniqueness of strong solutions in critical Besov spaces, assuming that the initial data are close to an equilibrium state $(1,0,e)$ for some constant vector $e \in \mathbb{S}^2$.
For initial data close to a non-vacuum equilibrium state in $H^s(\rrr)$ with $s \ge 3$, Gao-Tao-Yao \cite{GTY} proved the global existence and large-time behavior of classical solutions.
More recently, drawing on the ideas in \cite{HLX2,LX2}, Li-Xu-Zhang \cite{LXZ} and Wang \cite{WT} established the global existence and uniqueness of classical solutions to (\ref{nlckv}) for the three- and two-dimensional Cauchy problems, respectively.
These results allow for initial data with large oscillations and vacuum, provided that the initial energy is sufficiently small.
Subsequently, for general three-dimensional bounded domains with slip boundary conditions for the velocity field and Neumann boundary conditions for the orientation field, Liu-Zhong \cite{LZ} employed the approach of Cai-Li \cite{CL} to prove the global existence and uniqueness of classical solutions under the assumption that the initial energy is small enough.

More recently, Zhong-Zhou \cite{ZZ,ZZ2} proved the global existence and uniqueness of strong solutions to the system (\ref{nlckv}) in a simply connected bounded smooth domain and in the whole space with vacuum far-field density, under the assumptions that $\beta>\frac{4}{3}$ and that the initial orientation field satisfies the geometric condition
\be\la{jhtj}\ba
d_0^2 \ge \ep_0 \ \text{ for some positive } \ep_0>0.
\ea\ee
The purpose of this paper is to establish the global existence and uniqueness of strong solutions in both the whole space and the half-space, with either vacuum or non-vacuum far-field density, without imposing the geometric condition (\ref{jhtj}), provided that $\beta>\frac{4}{3}$.

Before stating our main results, we introduce the notations
and conventions used throughout the paper. We set
\be\ba\nonumber
\na^\bot \triangleq (\p_2,-\p_1),\quad
\int f dx \triangleq \int_{\OM} fdx.
\ea\ee
Moreover, $\na d \odot \na d$ denotes the matrix whose
$(i,k)$-th entry is given by $\p_i d \cdot \p_k d$. For any $R>0$, we define
\be\ba\nonumber
B_{R} \triangleq \{x \in \rr | \  |x|<R \}, \quad B^+_{R} \triangleq \{x \in \rr_+ | \  |x|<R \}.
\ea\ee
For any positive integer $s$ and $1\leq r\leq \infty$, we denote the standard Lebesgue and Sobolev spaces as follows:
\be\ba\nonumber
\begin{cases}
	L^r =L^r(\OM),\quad W^{s,r} =W^{s,r}(\OM),\quad H^s =W^{s,2}, \\
	D^{s,r}=D^{s,r}(\rr )=\{v\in L^1_\mathrm{loc}(\rr )| \nabla^s v\in L^r(\rr )\}, \quad D^1=D^{1,2}, \\
	D_+^{s,r}=D_+^{s,r}(\rr_+)=\{v\in L^1_\mathrm{loc}(\rr_+)| \nabla^s v\in L^r(\rr_+)\}, \quad D_+^1=D_+^{1,2}, \\
	\tilde{H}^1 =\{v \in H^1(\rr_+)\vert v \cdot n=0, \curl v = 0 \,\,\,\text{on}\,\,\, \partial \rr_+ \}, \\
	\tilde{D}_+^1 =\{v \in D_+^1 \vert v \cdot n=0, \curl v = 0 \,\,\,\text{on}\,\,\, \partial \rr_+ \}.
\end{cases}
\ea\ee
The material derivative is defined by
\be\ba\nonumber
\frac{D}{Dt}f=\dot{f} \triangleq f_t + u\cdot\na f.
\ea\ee
We define the effective viscous flux $G$ and the vorticity $\o$ as
\be\ba\la{gw}
G \triangleq (2\mu + \lam)\div u - (P-P(\tilde{\n})), \quad \o \triangleq \na^\bot \cdot u = \pa_2 u^1 - \pa_1 u^2.
\ea\ee
The potential energy density is defined by
\be\nonumber\ba
K(\n) \triangleq \n \int_{\tilde{\rho}}^{\n} \frac{P(s)-P(\tilde{\rho})}{s^2} ds.
\ea\ee

We now give the definition of strong solutions to (\ref{nlckv}).
\begin{definition}
If all derivatives involved in \eqref{nlckv} for $(\n,u,d)$ are regular distributions,
and equations \eqref{nlckv} hold almost everywhere in 
$\OM \times(0,T)$, then $(\n,u,d)$ is called a strong solution.
\end{definition}

The first result concerns the global existence and uniqueness of strong solutions for the Cauchy problem with vacuum far-field density.
\begin{theorem}\la{thcp1}
Let $\OM = \rr$ and $\tilde{\n}=0$.
Assume that 
\be\la{bg}\ba
\beta>\frac{4}{3}, \quad \ga>1,
\ea\ee
and that the initial data $(\n_0,u_0,d_0)$ satisfy, for some $q>2$ and $a \in (1,2)$,
\be\la{cpsol1}\ba
\begin{cases}
\n_{0}\geq 0, \quad \bar{x}^{a}\n_{0}\in L^{1}\cap H^{1}\cap W^{1,q}, \quad \tilde{M}_0\triangleq\|\rho_0\|_{L^1(\mathbb{R}^2)}>0, \\
\sqrt{\n_0} u_{0}\in L^{2},\quad
 u_0 \in D^{1}, \quad
\bar{x}^{\frac{a}{2}} \nabla d_{0} \in L^{2},\quad \nabla^{2} d_{0}\in L^{2},\quad |d_{0}|=1, 
\end{cases}
\ea\ee
with
\be\la{cpsol2}\ba
{\bar{x}}\triangleq (e+|x|^2)^{\frac{1}{2}} \log ^{1+\eta_0} (e+|x|^2),
\quad \eta_0 = \frac{3}{8}-\frac{1}{2\beta}>0.
\ea\ee
Then the Cauchy problem \eqref{nlckv}--\eqref{i30} with the far-field condition \eqref{bjtj2} admits a unique global strong solution $(\n,u,d)$ on $\rr \times (0,\infty)$ satisfying, for any $0<T<\infty$,
\be\la{cpsol3}\ba
\begin{cases}
\rho \in C([0,T];L^1 \cap H^1\cap W^{1,q} ), \\
{\bar{x}}^a\rho \in L^\infty ( 0,T ;L^1\cap H^1\cap W^{1,q} ), \\
\sqrt{\rho} u, \nabla u, {\bar{x}}^{-1}u, \sqrt{t} \sqrt{\rho } u_t, \sqrt{t} \na^2 u \in L^\infty (0,T;L^2), \\
\bar{x}^{\frac{a}{2}} \nabla d, \nabla^{2} d,  \sqrt{t} \nabla d_{t},  \sqrt{t} \nabla^{3}d, 
\sqrt{t} \bar{x}^{\frac{a}{2}} \na^2 d \in L^\infty (0,T;L^2), \\
\sqrt{\rho } u_t,  \na^2 u, \sqrt{t} \nabla u_t, \sqrt{t} {\bar{x}}^{-1}u_t \in L^2({\mathbb {R}^2 }\times (0,T)), \\
\nabla^{3} d, \nabla d_{t}, \bar{x}^{\frac{a}{2}} \nabla^{2} d, \sqrt{t} \na^2 d_t, \sqrt{t} \na^4 d \in L^{2}(\mathbb{R}^{2}\times (0,T)), \\
\nabla u \in L^{(q+1)/q}(0,T; W^{1,q}),\quad \sqrt{t}\nabla u\in L^2(0,T; W^{1,q} ), \\
\end{cases}
\ea\ee
and
\be\la{cpsol4}\ba
\inf_{0 \le t \le T} \int _{B_{N_1}}\rho (x,t) dx \ge \frac{1}{4} \tilde{M}_0,
\ea\ee
for some positive constant $N_1$ depending on
$T$, $\mu$, $\tilde{M}_0$, $\| \n_0 \|_{L^\infty}$, $\| \sqrt{\n_0} u_0 \|_{L^2}$, $\|\bar{x}^{a}\rho_0\|_{L^1}$, and $\| \na d_0 \|_{L^2}$.
\end{theorem}

We next consider the Cauchy problem with non-vacuum far-field density. The corresponding global existence and uniqueness result for strong solutions is stated below.
\begin{theorem}\la{thcp2}
Let $\OM = \rr$ and $\tilde{\n}>0$.
Assume that \eqref{bg} holds and that the initial data $(\n_0,u_0,d_0)$ satisfy, for some $q>2$ and $\alpha \in (0,1)$,
\be\la{cp2sol1}\ba
\begin{cases}
\n_{0}\geq 0, \quad \n_{0}-\tilde{\n} \in H^{1}\cap W^{1,q}, \quad
u_{0}\in H^{1},\quad \tilde{x}^{2 \alpha} K(\n_0) \in L^1, \\
\tilde{x}^{\alpha} \sqrt{\n_0} u_0 \in L^2, \quad \tilde{x}^{\alpha} \nabla d_{0} \in L^{2},\quad \nabla^{2} d_{0}\in L^{2},
\quad |d_{0}|=1, 
\end{cases}
\ea\ee
with
\be\la{cp2sol2}\ba
\tilde{x} \triangleq (e+|x|^2)^{\frac{1}{2}}.
\ea\ee
Then the Cauchy problem \eqref{nlckv}--\eqref{i30} with the far-field condition \eqref{bjtj2} admits a unique global strong solution $(\n,u,d)$ on $\rr \times (0,\infty)$ satisfying, for any $0<T<\infty$,
\be\la{cp2sol3}\ba
\begin{cases}
\rho-\tilde{\rho} \in C([0,T];H^1 \cap W^{1,q} ), \\
u\in L^\infty(0,T; H^1), \quad \tilde{x}^{2 \alpha} K(\n) \in L^\infty(0,T;L^1), \\
\tilde{x}^{\alpha} \sqrt{\n} u, \tilde{x}^{\alpha} \na d, \na^2 d, \sqrt{t} \sqrt{\n} u_t,
\sqrt{t} \na^2 u, \sqrt{t} \na d_t, \sqrt{t} \na^3 d \in L^\infty(0,T;L^2), \\
\sqrt{\n} u_t, \na^2 u, \na d_t, \na^3 d, \sqrt{t} \na u_t, \sqrt{t} \na^2 d_t, \sqrt{t} \na^4 d \in L^2(\rr \times (0,T)), \\
\na u \in L^{(q+1)/q}(0,T;W^{1,q}), \quad \sqrt{t} \na u \in L^2(0,T;W^{1,q}).
\end{cases}
\ea\ee
\end{theorem}

Consider the half-space problem with vacuum far-field density.
We have the following result on the global existence and uniqueness of strong solutions.
\begin{theorem}\la{thbkj1}
Let $\OM = \rr_+$ and $\tilde{\n}=0$.
Assume that \eqref{bg} holds and that the initial data $(\n_0,u_0,d_0)$ satisfy, for some $q>2$ and $a \in (1,2)$,
\be\la{bkjsol1}\ba
\begin{cases}
\n_{0}\geq 0, \quad \bar{x}^{a}\n_{0}\in L^{1}\cap H^{1}\cap W^{1,q}, \quad \tilde{M}_1\triangleq\|\rho_0\|_{L^1(\mathbb{R}^2_{+})}>0, \quad
\sqrt{\n_0} u_{0}\in L^{2}, \\ u_0 \in \tilde{D}_+^1, \quad
\bar{x}^{\frac{a}{2}} \nabla d_{0} \in L^{2},\quad \nabla^{2} d_{0}\in L^{2},
\quad |d_{0}|=1, \quad (n \cdot \na d_0)|_{\p \rr_+}=0,
\end{cases}
\ea\ee
where $\bar{x}$ is defined as in \eqref{cpsol2}.
Then the half-space problem \eqref{nlckv}--\eqref{i30} with the boundary conditions \eqref{bkjbjtj1} and the far-field condition \eqref{bkjbjtj2} admits a unique global strong solution $(\n,u,d)$ on $\rr_+ \times (0,\infty)$ satisfying, for any $0<T<\infty$,
\be\la{bkjsol3}\ba
\begin{cases}
\rho \in C([0,T];L^1 \cap H^1\cap W^{1,q} ), \\
{\bar{x}}^a\rho \in L^\infty ( 0,T ;L^1\cap H^1\cap W^{1,q} ), \\
\sqrt{\rho} u, \nabla u, {\bar{x}}^{-1}u, \sqrt{t} \sqrt{\rho } u_t, \sqrt{t} \na^2 u \in L^\infty (0,T;L^2), \\
\bar{x}^{\frac{a}{2}} \nabla d, \nabla^{2} d,  \sqrt{t} \nabla d_{t},  \sqrt{t} \nabla^{3}d, 
\sqrt{t} \bar{x}^{\frac{a}{2}} \na^2 d \in L^\infty (0,T;L^2), \\
\sqrt{\rho } u_t,  \na^2 u, \sqrt{t} \nabla u_t, \sqrt{t} {\bar{x}}^{-1}u_t \in L^2({\mathbb{R}_+^2 }\times (0,T)), \\
\nabla^{3} d, \nabla d_{t}, \bar{x}^{\frac{a}{2}} \nabla^{2} d, \sqrt{t} \na^2 d_t, \sqrt{t} \na^4 d \in L^{2}(\mathbb{R}_+^{2}\times (0,T)), \\
\nabla u \in L^{(q+1)/q}(0,T; W^{1,q}),\quad \sqrt{t}\nabla u\in L^2(0,T; W^{1,q} ),
\end{cases}
\ea\ee
and
\be\la{bkjsol4}\ba
\inf_{0 \le t \le T} \int _{B^+_{N_2}}\rho (x,t) dx \ge \frac{1}{4} \tilde{M}_1,
\ea\ee
for some positive constant $N_2$ depending on
$T$, $\mu$, $\tilde{M}_1$, $\| \n_0 \|_{L^\infty}$, $\| \sqrt{\n_0} u_0 \|_{L^2}$, $\|\bar{x}^{a}\rho_0\|_{L^1}$, and $\| \na d_0 \|_{L^2}$.
\end{theorem}

Finally, we address the half-space problem (\ref{nlckv})--(\ref{i30}), (\ref{bkjbjtj1}), (\ref{bkjbjtj2}) with non-vacuum far-field density and establish the global existence and uniqueness of strong solutions.
\begin{theorem}\la{thbkj2}
Let $\OM = \rr_+$ and $\tilde{\n}>0$.
Assume that \eqref{bg} holds and that the initial data $(\n_0,u_0,d_0)$ satisfy, for some $q>2$ and $\alpha \in (0,1)$,
\be\la{bkj2sol1}\ba
\begin{cases}
\n_{0}\geq 0, \quad \n_{0}-\tilde{\n} \in H^{1}\cap W^{1,q}, \quad
u_{0}\in \tilde{H}^1,\quad \tilde{x}^{2 \alpha} K(\n_0) \in L^1, \\
\tilde{x}^{\alpha} \sqrt{\n_0} u_0 \in L^2, \quad \tilde{x}^{\alpha} \nabla d_{0} \in L^{2},\quad \nabla^{2} d_{0}\in L^{2},
\quad |d_{0}|=1, \quad (n \cdot \na d_0)|_{\p \rr_+} = 0,
\end{cases}
\ea\ee
where $\tilde{x}$ is defined as in \eqref{cp2sol2}. Then the half-space problem \eqref{nlckv}--\eqref{i30} with the boundary conditions \eqref{bkjbjtj1} and the far-field condition \eqref{bkjbjtj2} admits a unique global strong solution $(\n,u,d)$ on $\rr_+ \times (0,\infty)$ satisfying, for any $0<T<\infty$,
\be\la{bkj2sol3}\ba
\begin{cases}
\rho-\tilde{\rho} \in C([0,T];H^1 \cap W^{1,q} ), \\
u\in L^\infty(0,T; H^1), \quad \tilde{x}^{2 \alpha} K(\n) \in L^\infty(0,T;L^1), \\
\tilde{x}^{\alpha} \sqrt{\n} u, \tilde{x}^{\alpha} \na d, \na^2 d, \sqrt{t} \sqrt{\n} u_t,
\sqrt{t} \na^2 u, \sqrt{t} \na d_t, \sqrt{t} \na^3 d \in L^\infty(0,T;L^2), \\
\sqrt{\n} u_t, \na^2 u, \na d_t, \na^3 d, \sqrt{t} \na u_t, \sqrt{t} \na^2 d_t, \sqrt{t} \na^4 d \in L^2(\rr_+ \times (0,T)), \\
\na u \in L^{(q+1)/q}(0,T;W^{1,q}), \quad \sqrt{t} \na u \in L^2(0,T;W^{1,q}).
\end{cases}
\ea\ee
\end{theorem}

A few remarks are in order.

\begin{remark}\la{lrk1}
If the initial data $(\n_0,u_0,d_0)$ further satisfy higher regularity
and the compatibility condition
\be\la{csol2}\ba
- \mu \Delta u_0 - \nabla( (\mu + \lm(\n_0) ) \div u_0 ) + \nabla P(\n_0) + \na d_0 \cdot \Delta d_0 =\n_0^{1/2}g,
\ea\ee
for some $g \in L^2$, then the strong solutions obtained in Theorems \ref{thcp1}, \ref{thcp2}, \ref{thbkj1}, and \ref{thbkj2} become classical solutions for positive time.
The detailed proofs follow from arguments analogous to those in \cite{JWX1,LZZ,HL,HLX2}.
\end{remark}

\begin{remark}\la{lrk3}
Compared with the results of Zhong-Zhou \cite{ZZ}, Theorem \ref{thcp1} establishes the global existence of strong solutions without imposing the geometric condition \eqref{jhtj}.
Thus, our results generalize and improve upon those in \cite{ZZ}.
\end{remark}

\begin{remark}\la{lrk4}
It appears that the condition $\beta>1$ is critical for the system \eqref{nlckv}--\eqref{i2}; see \cite{VK}.
Therefore, it would be interesting to investigate the case $1<\beta \le 4/3$, which is left for the future.
\end{remark}

\begin{remark}\la{lrk5}
When $\beta=0$, the system \eqref{nlckv}--\eqref{i30} reduces to the compressible nematic liquid crystal flows with constant viscosity coefficients.
For this system, Jiang-Jiang-Wang \cite{JJW1,JJW2} established the global existence of finite-energy weak solutions in two-dimensional bounded domains and in the whole space, under the assumptions that $\ga>1$ and that the initial orientation field satisfies the geometric condition \eqref{jhtj}.
The method developed in Lemma \ref{cp1l1} for deriving the basic energy estimate allows us to remove the geometric condition \eqref{jhtj} and thereby extend their results.
\end{remark}

We now make some comments on the analysis of this paper.
First, the local existence and uniqueness of strong solutions can be established by arguments similar to those in \cite{LLL}.
To extend these solutions globally in time, we need to derive global a priori estimates.
The key is to establish an upper bound for the density.

In previous works \cite{ZZ,ZZ2}, the initial orientation field was assumed to satisfy the geometric condition (\ref{jhtj}).
The main difficulty in removing this assumption lies in deriving an $L^2(\OM \times (0,T))$ estimate for $\na^2 d$.
Indeed, the standard energy estimate yields only an $L^2(\OM \times (0,T))$ bound for $\Delta d + |\na d|^2 d$.
To obtain an estimate for $\na^2 d$, the arguments in \cite{ZZ,ZZ2} make use of the geometric condition (\ref{jhtj}), together with the constraint $|d|=1$, and apply the maximum principle to $d$, yielding
\be\la{gjgj}\ba
\| \Delta d + |\na d|^2 d \|^2_{L^2} \ge \frac{\bar{\o}}{2} ( \| \Delta d \|^2_{L^2} + \| \na d \|^4_{L^4} ), \  \text{ for some } \bar{\o} \in (0,1).
\ea\ee
Combining this inequality with the standard energy estimate gives the desired $L^2(\OM \times (0,T))$ estimate for $\na^2 d$.
Without the geometric condition (\ref{jhtj}), however, the inequality (\ref{gjgj}) is no longer available.
To overcome this difficulty, we fully exploit the condition $|d|=1$ by representing $d$ in polar coordinates as $d=(\cos \theta, \sin \theta )$; see Lemma \ref{jzb}.
A direct calculation shows that
\be\nonumber\ba
|\Delta d + |\na d|^2 d|^2 = |\Delta \theta|^2, \quad  |\na d|^2 = |\na \theta|^2.
\ea\ee
These identities, together with the standard energy estimate and the Gagliardo-Nirenberg inequality, yield the required $L^2(\OM \times (0,T))$ estimate for $\na^2 d$.

Moreover, in contrast to the compressible Navier-Stokes equations, the presence of the liquid crystal director field $d$ gives rise to additional difficulties due to the strongly coupled term $u \cdot \na d$ and the nonlinear terms $\na d \cdot \Delta d$ and $|\na d|^2 d$.
To control these terms, we exploit the parabolic structure of the director equation $(\ref{nlckv})_3$.
More precisely, by differentiating the director equation, testing the resulting equation against $|\na d|^{p-2} \na d$, and applying the Gagliardo-Nirenberg inequality, we derive the key $L^\infty(0,T;L^p)$ estimate for $\na d$, where $2<p<\infty$.
This estimate plays a crucial role in controlling the coupling terms throughout the analysis.

For the Cauchy problem with vacuum far-field density, the main difficulty is due to the lack of integrability of $u$ and $\dot{u}$.
To overcome this, we introduce weighted estimates (see Lemma \ref{jqgj}) developed in \cite{HL3}.
Combining these estimates with the analytical framework in \cite{HL3}, we obtain the upper bound for the density under the assumption $\beta>4/3$.
For the case of non-vacuum far-field density, we first derive a Poincar\'e-type inequality (\ref{2cp15}) following the approach in \cite{WX}.
The main difficulty in this case is that the integrability of $(-\Delta)^{-1} \div (\n u)$ is unavailable, since $\n u$ belongs only to the critical space $L^2$.
To address this issue, inspired by \cite{WX}, we decompose $\n u$ as $\n u = \sqrt{\n} ( \sqrt{\n} - \sqrt{\tilde{\n}} )u + \sqrt{\tilde{\n}} \sqrt{\n} u$, and establish a weighted estimate for $\sqrt{\n} u$ (see (\ref{2cp03})).
Using these estimates, we derive the upper bound for the density, provided that $\beta>4/3$.

As in the Cauchy problem, the unboundedness of the domain poses difficulties in the analysis of the half-space problem.
However, the weighted estimates for $\rr$ (Lemma \ref{jqgj}) are not directly applicable because of the presence of the boundary.
Fortunately, the half-space possesses a special geometric structure with a flat boundary, which allows us to establish weighted estimates of the same form via an even extension argument (see (\ref{1bkj95})).
Moreover, using the Green's function for the half-space, we derive pointwise estimates for the effective viscous flux (see Lemma \ref{bkj1l8}).
Following the proof strategy used for the Cauchy problem, we then obtain the upper bound for the density under the condition $\beta>4/3$.
To handle the boundary terms, we employ the following equalities:
\be\la{bkjbjds}\ba
\dot{u} \cdot n =0, \quad  u \cdot \na u \cdot n = 0 \  \text{ on } \p \rr_+,
\ea\ee
and
\be\la{bkjbjds2}\ba
\na d^j \cdot \na (n \cdot \na d^j) = 0 \quad \text{ on } \p \rr_+,
\ea\ee
which hold because the boundary is flat and $u \cdot n = n \cdot \na d = 0$ on $\p \rr_+$.

The rest of this paper is organized as follows: Section 2 collects some preliminary results and inequalities required for our analysis.
In Sections 3 and 4, we derive the upper bound for the density and establish higher-order derivative estimates in the whole space and in the half-space, respectively.
Finally, Section 5 presents the proofs of the main results, Theorems \ref{thcp1}--\ref{thbkj2}.

\section{Preliminaries}
In this section, we collect several known facts and elementary inequalities that will be used frequently throughout this paper.

We begin with the following local existence result for strong solutions, which can be established by arguments similar to those in \cite{LLL}.
\begin{lemma}\la{lct}
Assume that $\beta \ge 1$, $\ga>1$ and the initial data $\left(\n_0,u_0,d_0 \right)$ satisfy \eqref{cpsol1}, \eqref{cp2sol1}, \eqref{bkjsol1}, or \eqref{bkj2sol1}.
Then there exists a small time $T>0$ such that the problem \eqref{nlckv}--\eqref{i30} admits a unique strong solution $(\n,u,d)$ in $\OM \times (0,T]$ subject to the corresponding boundary conditions.
Moreover, the solution satisfies the corresponding regularity properties listed in \eqref{cpsol3}, \eqref{cp2sol3}, \eqref{bkjsol3}, or \eqref{bkj2sol3}, respectively.
\end{lemma}

The following Gagliardo-Nirenberg inequalities (see \cite{NI,TG}) will be used frequently.
\begin{lemma}\la{gn1}
Let $\OM$ be either $\rr$ or $\rr_+$. There exists a positive constant $C$ such that for any $2<p<\infty$ and $f \in H^1(\OM)$, 
\be\ba\la{gn11}
\| f \|_{L^p} \le Cp^{1/2}\| f \|^{2/p}_{L^2} \| \na f \|^{1-2/p}_{L^2}.
\ea\ee
\end{lemma}

The following elliptic estimate plays an important role in the analysis of the half-space problem.
Its proof can be found in \cite{ADN,GT}.
\begin{lemma}\la{tygj}
For any $f\in H_{\textnormal{loc}}^{2}(\rr_+)$ with $\na f \in H^1(\rr_+)$ satisfying $\frac{\p f}{\p n}=0$ on $\p \rr_+$, there exists a positive constant $C$ such that
\be\la{tygj1}\ba
\| \na^2 f \|_{L^2} \le C \| \Delta f \|_{L^2}.
\ea\ee
\end{lemma}

We will frequently use the following div-curl estimates (see \cite{AJ,MD,WWV}) to bound $\na u$.
\begin{lemma}\la{dc}
Let $\Omega$ be either $\mathbb{R}^2$ or $\mathbb{R}_{+}^2$. Let $k \ge 0$ be an integer and let $1<p<\infty$.
Then there exists a positive constant $C$ depending only on $k$ and $p$ such that for every $u\in W^{k+1,p}(\Omega)$ satisfying $u \cdot n=0$ on $\p \Omega$ when $\Omega=\rr_{+}$,
\be\ba\la{dc1}
\| \na u \|_{W^{k,p}} \le C\left( \|\div u\|_{W^{k,p}} +\| \curl u \|_{W^{k,p}} \right).
\ea\ee
\end{lemma}

By virtue of the weighted estimate \cite[Theorem B.1]{L1}, we have the following weighted $L^p$-estimates for functions in $D^1(\rr)$, whose proofs can be found in \cite[Lemma 2.5]{HL3}.
\begin{lemma}\la{jqgj}
Let $\bar{x}$ and $\eta_0$ be as in \eqref{cpsol2}.
Assume that $\n \in L^1(\rr)\cap L^\infty(\rr)$ is a nonnegative function such that
\be\nonumber\ba
\int_{B_{N_1}} \n dx \ge M_1, \quad \int_{\rr} \n^{\ga_1} dx \le M_2, \quad \int_{\rr} \n \bar{x}^{\ga_2} dx \le M_3,
\ea\ee
for positive constants $M_i (i=1,2,3)$, $ N_1 \ge 1$, $\ga_1 >1$, and $\ga_2 \in (1,2)$.
Then there is a positive constant $C$ depending only on $ M_i (i=1,2,3)$,
$N_1$, $\ga_1$, $\ga_2$, and $\eta_0$ such that for any $v\in D^1(\rr)$ and $r\in (1,\infty)$,
\be\la{jqgj1}\ba
\|\n v\|_{L^r(\rr)} \le C r^{\eta_0 + \frac{1}{2}}
\left( 1 + \|\n\|_{L^\infty(\rr)} \right)
\left( \| \sqrt{\n} v \|_{L^2(B_{N_1})} + \|\na  v\|_{L^2(\rr)} \right).
\ea\ee

Furthermore, for any $\ep>0$ and $0< \eta \le 1$, there is a positive constant $C$ depending only on
$ M_i (i=1,2,3)$, $N_1$, $\ga_1$, $\ga_2$, $\eta_0$, $\ep$, and $\eta$ such that every $v \in D^1(\rr)$ satisfies
\be\la{jqgj2}\ba
\| v \bar{x}^{-\eta} \|_{ L^{(2+\ep)/\eta} (\rr) }
\le C \| \sqrt{\n} v \|_{L^2(\rr)} + C \| \na v \|_{L^2(\rr)}.
\ea\ee
\end{lemma}

Let $\mathcal{H}^1(\rr)$ and $\mathcal{BMO}(\rr)$ denote the usual Hardy and $\mathcal{BMO}$ spaces.
For a function $b$, define the commutator
\be\ba\la{p4}
[b,R_iR_j](f) \triangleq bR_i\circ R_j(f) - R_i \circ R_j(bf),\quad i,j=1,2,
\ea\ee
where $R_i = (-\Delta)^{-\frac{1}{2}}\pa_i$ is the usual Riesz transform on $\rr$.

We recall the following well-known commutator estimates due to Coifman-Rochberg-Weiss \cite{CRW} and Coifman-Meyer \cite{CM}.
\begin{lemma}\la{jhzyl}
Let $b,f\in C_0^\infty(\rr)$.
Then for $p\in (1,\infty)$, there is a positive constant $C$ depending only on $p$ such that
\be\ba\la{jhz1}
\|[b,R_iR_j](f)\|_{L^p}\leq C \| b\|_{\mathcal{BMO}}\|f\|_{L^p}. \ea\ee
Moreover, for $p,q,r \in (1,\infty)$ with $\frac{1}{r} = \frac{1}{p}+ \frac{1}{q}$, there exists a positive constant $C$ depending only on $p$, $q$, and $r$ such that
\be\ba\la{jhz2}
\|\na[b,R_iR_j](f)\|_{L^r}\leq C \|\na b\|_{L^p}\|f\|_{L^q}.
\ea\ee
\end{lemma}

The following Brezis-Wainger inequality (see \cite{BW,E}) will be applied to obtain the upper bound for the density in the Cauchy problem.
\begin{lemma}\la{BWI}
For $2<q<\infty$, there exists some positive constant $C$ depending only on $q$
such that for any $v \in D^1(\rr) \cap W^{1,q}(\rr)$,
\be\la{bwi}\ba
\| v \|_{L^\infty(\rr)} \le C  \left( \| v \|_{L^q(\rr)} + \| \na v \|_{L^2(\rr)} \right) \log^{1/2}(e + \|v\|_{W^{1,q}(\rr)} ) + C.
\ea\ee
\end{lemma}

To estimate $\| \na u\|_{L^{\infty}}$ and $\| \na \n\|_{L^{q}}$, we require the following Beale-Kato-Majda type inequality; its proof can be found in \cite{K,BKM,CL}.
\begin{lemma}\la{bkm}
Let $2<q<\infty$.
Assume either that $\na u \in L^2(\rr) \cap D^{1,q}(\rr)$ or that $u \in \tilde{D}^1_+(\rr_+)$ and $\na u \in D_+^{1,q}(\rr_+)$.
Then there exists a positive constant $C$ depending only on $q$ such that
\be\ba\la{bkm1}
\|\na u\|_{L^\infty} \le C \left( \|\div u \|_{L^\infty}
+ \|\curl u\|_{L^\infty} \right) \log \left(e+ \|\na^2 u\|_{L^q} \right)+ C\|\na u\|_{L^2}+C.
\ea\ee
\end{lemma}

Finally, we show that any unit vector field defined on a simply connected domain in $\rr$ admits a polar representation.
\begin{lemma}\la{jzb}
Let $\OM \subset \rr$ be a simply connected domain.
Assume that $d=(d^1,d^2) \in H^2_{\textnormal{loc}}(\OM;\rr)$ with $|d|=1$.
Then there exists a function $\theta \in H^2_{\textnormal{loc}}(\OM)$ such that
\be\la{jzb1}\ba
d=(\cos \theta, \sin \theta).
\ea\ee
\end{lemma}
\begin{proof}
Define the vector field
\be\nonumber\ba
Q = (Q_1,Q_2),
\ea\ee
where
\be\ba\la{jzb101}
Q_1 \triangleq -d^2 \frac{\p d^1}{\p x_1} + d^1 \frac{\p d^2}{\p x_1}, \quad
Q_2 \triangleq -d^2 \frac{\p d^1}{\p x_2} + d^1 \frac{\p d^2}{\p x_2}.
\ea\ee
A direct computation gives
\be\ba\la{jzb11}
\frac{\p Q_1}{\p x_2} - \frac{\p Q_2}{\p x_1}
= 2 \left( \frac{\p d^1}{\p x_2} \frac{\p d^2}{\p x_1} - \frac{\p d^1}{\p x_1} \frac{\p d^2}{\p x_2} \right).
\ea\ee
Since $|d|=1$, we have
\be\ba\la{jzb12}
d^1\frac{\p d^1}{\p x_1} + d^2\frac{\p d^2}{\p x_1} = 0, \quad
d^1\frac{\p d^1}{\p x_2} + d^2\frac{\p d^2}{\p x_2} = 0.
\ea\ee
Using (\ref{jzb11}), (\ref{jzb12}), and the fact that $|d|=1$, we arrive at
\be\ba\nonumber
\na^\bot \cdot Q = \frac{\p Q_1}{\p x_2} - \frac{\p Q_2}{\p x_1} = 0.
\ea\ee
Since $\OM$ is simply connected, Poincar\'e's theorem (see \cite{Ma}) together with $d\in H^2_{\textnormal{loc}}(\OM;\rr)$ and (\ref{jzb101}) implies that there exists $\theta \in H^2_{\textnormal{loc}}(\OM)$ such that
\be\ba\la{jzb13}
\na \theta = Q = -d^2 \na d^1 + d^1 \na d^2.
\ea\ee

It remains to verify that $\theta$ satisfies (\ref{jzb1}).
The Sobolev embedding theorem ensures that $d$ and $\theta$ have continuous representatives.
Moreover, since $\theta$ is determined by (\ref{jzb13}) only up to an additive constant, we may choose this constant such that
\be\ba\nonumber
d(x_0)=(\cos \theta(x_0), \sin \theta(x_0)),
\ea\ee
for some $x_0 \in \OM$.
Let
\be\ba\la{jzb14}
U_1 \triangleq d^1 - \cos \theta, \quad U_2 \triangleq d^2 - \sin \theta.
\ea\ee
A straightforward calculation together with (\ref{jzb13}) yields
\be\ba\nonumber
\na U_1 = \left( 1 - d^2 \sin \theta \right) \na d^1 + d^1 \sin \theta \na d^2, \quad
\na U_2 = d^2 \cos \theta \na d^1 + \left( 1 - d^1 \cos \theta \right) \na d^2.
\ea\ee
Thus, we have
\be\ba\la{jzb15}
\frac{1}{2}\na \left( U^2_1 + U^2_2 \right) = U_1 \na U_1 + U_2 \na U_2
= \left( 1 - d^2 \sin \theta - d^1 \cos \theta \right) \left( d^1 \na d^1 + d^2 \na d^2 \right).
\ea\ee
By (\ref{jzb12}), the right-hand side of (\ref{jzb15}) vanishes identically.
Hence, $U^2_1 + U^2_2$ is constant in $\OM$.
Since $U_1(x_0)=U_2(x_0)=0$, it follows that $U_1(x)=U_2(x)=0$ for all $x \in \OM$.
Combining this with (\ref{jzb14}), we obtain (\ref{jzb1}) and complete the proof.
\end{proof}

\section{A Priori Estimates for the Cauchy Problem}
In this section, we establish an upper bound for $\n$ and derive higher-order estimates for the solution to the Cauchy problem, thereby extending the local solution globally in time.

We set
\be\la{a1}\ba
A_1^2(t) \triangleq 1 + \int \left( \mu \o^2 (t) + \frac{ G^2(t) }{2\mu+\lam(\n(t))} + |\Delta d(t)|^2 \right) dx,
\ea\ee
\be\ba\la{a2}
A_2^2(t) \triangleq \int \left( \rho |\dot{u}|^2 + |\na \Delta d|^2 + |\na d_t|^2 \right) dx,
\ea\ee
and
\be\ba\la{mdsj}
R_T \triangleq 1 + \sup_{0 \le t \le T} \| \n(t) \|_{L^\infty}.
\ea\ee

\subsection{The Case of Vacuum Far-field Density}
Throughout this subsection, the initial data $(\rho_0,u_0,d_0)$ are
assumed to satisfy (\ref{cpsol1}) and $\rho_0>0$. Since
$\int_{\mathbb{R}^2}\rho_0\,dx=\tilde M_0>0$, one can choose a positive
constant $N_0$ such that
\be\la{rho00}\ba
\int_{B_{N_0}} \n_0 dx \ge \frac{1}{2} \int_{\rr} \n_0 dx \ge \frac{\tilde{M}_0}{2}.
\ea\ee
Moreover, we suppose that $(\n,u,d)$ is the strong solution of (\ref{nlckv})--(\ref{i30}), (\ref{bjtj2}) on $\rr \times (0,T]$ with $\tilde{\n}=0$.

In addition, we set
\be\la{e1}\ba
E_1 \triangleq &   \| \sqrt{\n_0} u_0 \|_{L^2} + \| \n_0 \|_{L^\infty} + \| {\bar{x}}^a \rho_0 \|_{L^1 \cap W^{1,q}} +
\| \na u_0 \|_{L^2} \\
& + \| \bar{x}^{\frac{a}{2}} \na d_0 \|_{L^2} + \| \na^2 d_0 \|_{L^2}.
\ea\ee

First, we establish the following basic energy estimate.

\begin{lemma}\la{cp1l1}
There exist positive constants $C$ and $N_1$ both depending only on
$T$, $a$, $\ga$, $\mu$, $N_0$, $\tilde{M}_0$, $\| \sqrt{\n_0} u_0 \|_{L^2}$, $\| {\bar{x}}^a \rho_0 \|_{L^1}$,
$\| \n_0 \|_{L^\ga}$, and $\| \na d_0 \|_{L^2}$ such that
\be\la{1cp01}\ba
& \sup\limits_{0\le t\le T}\int\left(\n|u|^2+\n^\ga+\n\bar x^a
+ |\na d|^2 \right) dx \\
& + \int_0^T \int\left( \mu |\na u|^2+ \lambda(\n) (\div u)^2 + | \na^2 d |^2 \right) dxdt
\le C,
\ea\ee
and
\be\la{1cp001}\ba
\inf\limits_{0\le t\le T}\int_{B_{N_1} }\n  dx \ge \frac{1}{4} \int_{\rr} \n_0 dx\ge \frac{1}{4} \tilde{M}_0.
\ea\ee
\end{lemma}
\begin{proof}
	First, multiplying $(\ref{nlckv})_2$ and $(\ref{nlckv})_3$ by $u$ and $-(\Delta d + |\na d|^2 d)$, respectively, integrating by parts, adding the results, and using $(\ref{nlckv})_1$ and $(\ref{nlckv})_4$, we arrive at
	\be\la{bp11}\ba
	& \frac{d}{dt} \int \left( \frac{1}{2}\rho |u|^2 + \frac{P}{\ga-1} + \frac{1}{2} |\na d|^2 \right) dx \\
	& + \int \left( (\mu + \lam(\n)) (\div u)^2 + \mu |\nabla u|^2 + |\Delta d + |\na d|^2 d|^2 \right) dx = 0.
	\ea\ee
	Integrating (\ref{bp11}) over $(0,T)$ yields
	\be\la{bp12}\ba
	& \sup_{0\le t \le T} \int \left( \frac{1}{2}\rho |u|^2 + \frac{P}{\ga-1} + \frac{1}{2} |\na d|^2 \right) dx \\
	& + \int_0^T \int \left( (\mu + \lam(\n)) (\div u)^2 + \mu |\nabla u|^2 + |\Delta d + |\na d|^2 d|^2 \right) dx dt \le C.
	\ea\ee

	We next estimate the $L^2(\mathbb{R}^2 \times (0,T))$-norm of $\na^2 d$. To this end, we exploit the geometric constraint $|d|=1$ by representing $d$ in polar coordinates.
	
	By Lemma \ref{jzb}, there exists a function $\theta$ such that
	\be\la{bp13}\ba
	d=(\cos \theta, \sin \theta ).
	\ea\ee
	A direct calculation gives for any $i,j=1,2$,
	\be\la{bp103}\ba
	\p_i d = (-\sin \theta, \cos \theta) \p_i \theta
	= (-d^2,d^1) \p_i \theta,
	\ea\ee
	and
	\be\la{bp14}\ba
	\p_{ij} d = (- \sin \theta, \cos \theta)\p_{ij} \theta - (\cos \theta, \sin \theta) \p_i \theta \p_j \theta.
	\ea\ee
	Consequently, (\ref{bp103}) and the fact that $|d|=1$ imply $|\na d|=|\na \theta|$.
	Combining this with (\ref{bp13}) and (\ref{bp14}), we obtain
	\be\la{bp16}\ba
	\Delta d & = (-\sin \theta, \cos \theta) \Delta \theta - d |\na \theta|^2 \\
	& = (-\sin \theta, \cos \theta) \Delta \theta - d |\na d|^2,
	\ea\ee
	which yields
	\be\la{bp19}\ba
	|\Delta d + |\na d|^2 d|^2 = |\Delta \theta|^2.
	\ea\ee
	By (\ref{bp12}) and (\ref{bp19}), we have
	\be\la{bp110}\ba
	\int_0^T \| \Delta \theta \|^2_{L^2} dt \leq C.
	\ea\ee
	From (\ref{gn11}), (\ref{bp12}), (\ref{bp14}), and the standard elliptic estimates, we deduce that
	\be\la{bp115}\ba
	\| \na^2 d \|^2_{L^2} & \le C \left( \| \na \theta \|^4_{L^4} + \| \na^2 \theta \|^2_{L^2} \right) \\
	& \le C \left( \| \na^2 \theta \|^2_{L^2} \| \na \theta \|^2_{L^2} + \| \na^2 \theta \|^2_{L^2} \right) \\
	& \le C \| \na^2 \theta \|^2_{L^2} \le C \| \Delta \theta \|^2_{L^2},
	\ea\ee
	which together with (\ref{bp110}) leads to
	\be\la{bp116}\ba
	\int_0^T \| \na^2 d \|^2_{L^2} dt
	\leq C \int_0^T \| \Delta \theta \|^2_{L^2} dt \leq C.
	\ea\ee

Then, multiplying $(\ref{nlckv})_1$ by $\bar{x}^a$ and integrating the resulting equality over $\mathbb{R}^2$, we obtain after integration by parts that
	$$
	\begin{aligned}
		\frac{d}{d t} \int \rho \bar{x}^a d x & \leq C \int \rho|u| \bar{x}^{a-1} \log ^{1+\eta_0}\left(e+|x|^2\right) d x \\
		& \leq C\left(\int \rho \bar{x}^{2 a-2} \log ^{2\left(1+\eta_0\right)}\left(e+|x|^2\right) d x\right)^{1 / 2}\left(\int \rho |u|^2 d x\right)^{1 / 2} \\
		& \leq C\left(\int \rho \bar{x}^a d x\right)^{1 / 2},
	\end{aligned}
	$$	
	which, together with Gr\"onwall's inequality and (\ref{bp12}), gives
		\be\la{c16}\ba
		\sup _{0 \leq t \leq T} \int \rho \bar{x}^a d x \leq C.
	\ea\ee
Combining (\ref{bp12}), (\ref{bp116}), and (\ref{c16}) yields (\ref{1cp01}).
    
Finally, using (\ref{1cp01}) and adapting the argument in \cite[Lemma 3.2]{HL3}, we arrive at (\ref{1cp001}) and complete the proof of Lemma \ref{cp1l1}.
\end{proof}

\begin{lemma}\la{cp1l2}
There exists a positive constant $C$ depending only on
$T$, $a$, $\ga$, $\mu$, $N_0$, $\tilde{M}_0$, $\| \sqrt{\n_0} u_0 \|_{L^2}$, $\| {\bar{x}}^a \rho_0 \|_{L^1}$,
$\| \n_0 \|_{L^\ga}$, and $\| \bar{x}^{\frac{a}{2}} \na d_0 \|_{L^2}$ such that
\be\ba\la{1cp02}
\sup_{0\leq t\leq T} \|\nabla d \bar{x}^{\frac{a}{2}} \|_{L^{2}}^{2}
+ \int_{0}^{T} \|\nabla^{2} d\bar{x}^{\frac{a}{2}}\|_{L^{2}}^{2} dt
\leq C.
\ea\ee
Moreover, for any $2<p<\infty$, there exists a positive constant $C$ depending only on $p$,
$T$, $a$, $\ga$, $\mu$, $N_0$, $\tilde{M}_0$, $\| \sqrt{\n_0} u_0 \|_{L^2}$, $\| {\bar{x}}^a \rho_0 \|_{L^1}$,
$\| \n_0 \|_{L^\ga}$, and $\| \na d_0 \|_{H^1}$ such that
\be\la{1cp002}\ba
\sup_{0 \le t \le T} \| \na d \|_{L^p} \le C.
\ea\ee
\end{lemma}
\begin{proof}
According to (\ref{1cp01}), (\ref{1cp001}) and Lemma \ref{jqgj}, for any $\ep>0$ and $0<\eta\le 1$, we have
\be\la{jqgju}\ba
\| u \bar{x}^{-\eta} \|_{ L^{(2+\ep)/\eta} }
\le C \left( 1 + \| \na u \|_{L^2} \right).
\ea\ee
Applying $\na$ to $\eqref{nlckv}_3$ leads to
\be\la{1cp21}\ba
\na d_t-\Delta \na d=-\na (u\cdot\na d)+\na(|\na d|^2 d).
\ea\ee
Multiplying (\ref{1cp21}) by $\nabla d \bar{x}^{a}$ and integrating by parts over $\rr$ yields
\be\la{1cp22}\ba
\frac{1}{2}\frac{d}{dt} \|\nabla d \bar{x}^{\frac{a}{2}}\|_{L^{2}}^{2} 
+ \|\nabla^{2} d \bar{x}^{\frac{a}{2}}\|_{L^{2}}^{2}
\leq & C \int |\nabla d| |\nabla^{2} d| |\nabla \bar{x}^{a}| dx
+ C \int |\nabla u||\nabla d|^{2} \bar{x}^{a} dx \\
& + C \int |u||\nabla d|^{2} |\nabla \bar{x}^{a}| dx
+  C \int |\nabla d|^{4} \bar{x}^{a} dx \\
& \triangleq \sum_{i=1}^{4}J_i.
\ea\ee
We now estimate the terms $J_i$ on the right-hand side of (\ref{1cp22}).

Using Young's inequality, we have
\be\la{1cp23}\ba
J_{1}\leq C \int |\nabla d| |\nabla^{2} d| \bar{x}^{a}dx\leq \frac{1}{10} \|\nabla^{2} d \bar{x}^{\frac{a}{2}}\|_{L^{2}}^{2}
+C\|\nabla d \bar{x}^{\frac{a}{2}}\|_{L^{2}}^{2}.
\ea\ee
Furthermore, (\ref{gn11}) ensures that for any $\ep>0$,
\be\la{1cp24}\ba
\|\nabla d \bar{x}^{\frac{a}{2}}\|_{L^{4}}^{2}
& \le C \|\nabla d \bar{x}^{\frac{a}{2}}\|_{L^{2}} \|\nabla (\nabla d \bar{x}^{\frac{a}{2}})\|_{L^{2}} \\
& \le C \|\nabla d \bar{x}^{\frac{a}{2}}\|_{L^{2}}
\left( \|\nabla^{2}d \bar{x}^{\frac{a}{2}}\|_{L^{2}}
+ \|\nabla d \nabla \bar{x}^{\frac{a}{2}}\|_{L^{2}} \right) \\
& \le C \|\nabla^{2}d \bar{x}^{\frac{a}{2}}\|_{L^{2}} \|\nabla d \bar{x}^{\frac{a}{2}}\|_{L^{2}}
+ C \|\nabla d \bar{x}^{\frac{a}{2}}\|^2_{L^{2}}.
\ea\ee
This, together with H\"older's inequality, shows that
\be\la{1cp25}\ba
J_{2} \leq C \|\nabla u\|_{L^{2}} \|\nabla d \bar{x}^{\frac{a}{2}}\|_{L^{4}}^{2}
& \leq C \|\nabla^{2}d \bar{x}^{\frac{a}{2}}\|_{L^{2}} \|\nabla d \bar{x}^{\frac{a}{2}}\|_{L^{2}} \| \na u \|_{L^2}
+ C \|\nabla d \bar{x}^{\frac{a}{2}}\|^2_{L^{2}} \| \na u \|_{L^2} \\
& \leq \frac{1}{10} \|\nabla^{2} d \bar{x}^{\frac{a}{2}}\|_{L^{2}}^{2}
+ C \|\nabla d \bar{x}^{\frac{a}{2}}\|_{L^{2}}^{2} \left( 1 + \| \na u \|^2_{L^2} \right).
\ea\ee
Moreover, from (\ref{1cp01}), (\ref{jqgju}), (\ref{1cp24}) and H\"older's inequality, we deduce
\be\la{1cp26}\ba
J_{3}\leq C \int |u| |\nabla d|^2 \bar{x}^{a-\frac{3}{4}}dx
\leq & C \|\nabla d \bar{x}^{\frac{a}{2}}\|_{L^{4}} \|\nabla d \bar{x}^{\frac{a}{2}}\|_{L^{2}}\|u\bar{x}^{-\frac{3}{4}}\|_{L^{4}} \\
\leq & C\|\nabla d \bar{x}^{\frac{a}{2}}\|_{L^{4}}^{2}
+ C (\| \sqrt{\n} u \|_{L^{2}}^{2}
+ \|\nabla u\|_{L^{2}}^{2})\|\nabla d \bar{x}^{\frac{a}{2}}\|_{L^{2}}^{2} \\
\leq & \frac{1}{10} \|\nabla^{2} d \bar{x}^{\frac{a}{2}}\|_{L^{2}}^{2}
+ C \|\nabla d\bar{x}^{\frac{a}{2}}\|_{L^{2}}^{2} \left( 1 + \| \na u \|^2_{L^2} \right),
\ea\ee
and
\be\la{1cp27}\ba
J_{4} \leq C \|\nabla d\|^2_{L^{4}} \|\nabla d \bar{x}^{\frac{a}{2}}\|_{L^{4}}^{2}
& \le C \|\nabla^2 d\|_{L^{2}} \left( \|\nabla^{2}d \bar{x}^{\frac{a}{2}}\|_{L^{2}} \|\nabla d \bar{x}^{\frac{a}{2}}\|_{L^{2}}
+ \|\nabla d \bar{x}^{\frac{a}{2}}\|^2_{L^{2}} \right) \\
& \le \frac{1}{10} \|\nabla^{2} d \bar{x}^{\frac{a}{2}}\|_{L^{2}}^{2}
+ C \|\nabla d\bar{x}^{\frac{a}{2}}\|_{L^{2}}^{2} \left( 1 + \| \na^2 d \|^2_{L^2} \right).
\ea\ee
Substituting (\ref{1cp23}), (\ref{1cp25}), (\ref{1cp26}) and (\ref{1cp27}) into (\ref{1cp22}) yields
\be\la{1cp28}\ba
\frac{d}{dt} \|\nabla d \bar{x}^{\frac{a}{2}}\|_{L^{2}}^{2} 
+ \|\nabla^{2} d \bar{x}^{\frac{a}{2}}\|_{L^{2}}^{2}
\leq C \|\nabla d\bar{x}^{\frac{a}{2}}\|_{L^{2}}^{2}
\left( 1 + \| \na u \|^2_{L^2} + \| \na^2 d \|^2_{L^2} \right),
\ea\ee
which together with (\ref{1cp01}) and Gr\"onwall's inequality gives (\ref{1cp02}).

Next, multiplying (\ref{1cp21}) by $p |\na d|^{p-2} \na d$, integrating by parts over $\rr$, and using $|d|=1$, we derive
\be\la{bp22}\ba
& \frac{d}{dt} \int |\na d|^p dx
- p \int |\na d|^{p-2} \p_i d^j \p_k\p_k \p_i d^j dx \\
& \le -p \int |\na d|^{p-2} \p_i d^j \p_k \p_i d^j u^k dx
+ C \int \left( |\na d|^p |\na u| + |\na d|^{p+2} \right) dx \\
& \quad + p \int |\na d|^{p-2} \p_i d^j \p_i(|\na d|^2) d^j dx \\
& \le C \| \na d \|^{p+2}_{L^{p+2}} + C \| |\na d|^{\frac{p}{2}} \|^2_{L^4} \| \na u \|_{L^2}.
\ea\ee
Integration by parts yields
\be\la{bp23}\ba
- p \int |\na d|^{p-2} \p_i d^j \p_k\p_k \p_i d^j dx 
& = p \int \p_k( |\na d|^{p-2} \p_i d^j ) \p_k (\p_i d^j) dx \\
& = p \int |\na d|^{p-2} |\na^2 d|^2 dx
+ \frac{4(p-2)}{p} \int |\na (|\na d|^\frac{p}{2})|^2 dx.
\ea\ee

Consequently, we conclude from (\ref{bp22}), (\ref{bp23}), (\ref{gn11}), (\ref{1cp01}), and Young's inequality that
\be\la{bp025}\ba
& \frac{d}{dt} \int |\na d|^p dx + p \int |\na d|^{p-2} |\na^2 d|^2 dx \\
& \le C \| \na d \|^{p+2}_{L^{p+2}} + C \| |\na d|^{\frac{p}{2}} \|^2_{L^4} \| \na u \|_{L^2} \\
& \le C \| \na d \|^p_{L^p} \| \na^2 d \|^2_{L^2}
+ C \| |\na d|^{\frac{p}{2}} \|_{L^2} \| \na (|\na d|^{\frac{p}{2}}) \|_{L^2} \| \na u \|_{L^2} \\
& \le C \| \na d \|^p_{L^p} \| \na^2 d \|^2_{L^2} + C \| \na d \|^p_{L^p} \| \na u \|^2_{L^2}
+ \frac{p}{2} \int |\na d|^{p-2} |\na^2 d|^2 dx \\
& \le \frac{p}{2} \int |\na d|^{p-2} |\na^2 d|^2 dx 
+ C \| \na d \|^p_{L^p} ( \| \na^2 d \|^2_{L^2} + \| \na u \|^2_{L^2} ),
\ea\ee
which implies
\be\la{bp25}\ba
\frac{d}{dt} \int |\na d|^p dx + \frac{p}{2} \int |\na d|^{p-2} |\na^2 d|^2 dx
\le C \| \na d \|^p_{L^p} ( \| \na^2 d \|^2_{L^2} + \| \na u \|^2_{L^2} ).
\ea\ee

Applying Gr\"onwall's inequality to (\ref{bp25}) and using (\ref{1cp01}), we obtain (\ref{1cp002}) and complete the proof of Lemma \ref{cp1l2}.
\end{proof}

\begin{lemma}\la{cp1l3}
There exists a positive constant $C$ depending only on
$T$, $a$, $\ga$, $\mu$, $N_0$, $\tilde{M}_0$, $\beta$, and $E_1$ such that
\be\la{1cp03}\ba
\sup_{0\le t\le T}\int \left(\n+\n^{2\beta\ga+1 }\right) dx  \le C.
\ea\ee
\end{lemma}
\begin{proof}
First, using the effective viscous flux $G$ defined in (\ref{gw}), we rewrite $(\ref{nlckv})_2$ as
\be\la{1cp31}\ba
\n\dot{u}= \na G + \mu \na^\bot \o - \na d \cdot \Delta d,
\ea\ee
which yields
\be\la{1cp32}\ba
\Delta G = \div(\n \dot{u} + \na d \cdot \Delta d).
\ea\ee
Furthermore, we have
\be\la{1cp33}\ba
G + \frac{D}{Dt} \left( (- \Delta)^{-1} \div(\n u) \right)
= F_1 + F_2,
\ea\ee
where
\be\la{1cp34}\ba
F_1 \triangleq \sum_{i,j=1}^{2} [u^i,R_iR_j](\n u^j), \quad
F_2 \triangleq - (- \Delta)^{-1} \div(\na d \cdot \Delta d).
\ea\ee
From $(\ref{nlckv})_1$, (\ref{gw}), and (\ref{1cp33}), we conclude that
\be\la{1cp35}\ba
\frac{D}{Dt} (\Lambda_1(\n)-\psi) + P = - F_1 - F_2,
\ea\ee
where
\be\la{1cp36}\ba
\Lambda_1(\n) \triangleq 2 \mu \log \n + \beta^{-1} \n^\beta, \quad
\psi \triangleq (- \Delta)^{-1} \div(\n u).
\ea\ee

Next, multiplying (\ref{1cp35}) by $\n g^{2\ga-1}$ with $g \triangleq \max\{\Lambda_1(\n)-\psi,0 \}$ and integrating over $\rr$, we derive
\be\la{1cp37}\ba
\frac{d}{dt} \int \n g^{2\ga} dx
& \le C \int \n g^{2\ga-1} |F_1| dx + C \int \n g^{2\ga-1} |F_2| dx \\
& \le C \| \n g^{2\ga} \|^{1-1/(2\ga)}_{L^1} \| \n \|_{L^{2\beta\ga+1}}^{1/(2\ga)}
\left( \| F_1 \|_{L^{(2\beta\ga+1)/\beta}} + \| F_2 \|_{L^{(2\beta\ga+1)/\beta}} \right).
\ea\ee
On the one hand, it follows from (\ref{jhz1}), (\ref{jqgju}), (\ref{1cp01}) and H\"older's inequality that
\be\la{1cp38}\ba
\| F_1 \|_{L^{(2\beta\ga+1)/\beta}} & \le C \| u \|_{\mathcal{BMO}} \| \n u \|_{L^{(2\beta\ga+1)/\beta}} \\
& \le C \| \na u \|_{L^2} \| \n \bar{x}^a \|^s_{L^1} \| u \bar{x}^{-a s} \|_{L^{2/s}} \| \n \|_{L^{2\beta\ga+1}}^{1-s} \\
& \le C \left( \| \n \|_{L^{2\beta\ga+1}} + 1 \right)
\left( \| \na u \|^2_{L^2} + 1 \right),
\ea\ee
where $s=2(\beta-1)/(6 \beta \ga+1) <1/3$.

On the other hand, note that
\be\la{1cp39}\ba
\na d \cdot \Delta d = \div \left( \na d \odot \na d \right) - \frac{1}{2} \na (|\na d|^2),
\ea\ee
which together with (\ref{1cp002}) yields for any $1<p<\infty$,
\be\la{1cp310}\ba
\| F_2 \|_{L^p} \le C \| \na d \|^2_{L^{2p}} \le C.
\ea\ee
Moreover, we deduce from (\ref{1cp01}) and Young's inequality that
\be\la{1cp311}\ba
\int \n^{2 \beta \ga +1} dx
& = \int_{ \{\n \le 2\} } \n^{2 \beta \ga +1} dx + \int_{ \{\n > 2\} } \n^{2 \beta \ga +1} dx \\
& \le C \int_{ \{\n \le 2\} } \n dx + C \int_{ \{\n > 2\} } \n g^{2 \ga} dx
+ C \int_{ \{\n > 2\} } \n |\psi|^{2 \ga} dx \\
& \le C + C \int \n g^{2 \ga} dx + C \| \n \|_{L^{(2 \beta \ga+1)/( 2\beta \ga +1-\ga )}}
\| \psi \|^{2 \ga}_{ L^{2(2\beta \ga+1)} } \\
& \le C + C \int \n g^{2 \ga} dx
+ C ( 1 + \| \n \|_{L^{2 \beta \ga+1}} ) \| \n \|^\ga_{ L^{2\beta \ga+1} } \\
& \le \frac{1}{2} \int \n^{2 \beta \ga +1} dx + C + C \int \n g^{2 \ga} dx,
\ea\ee
where we have used the following estimate:
\be\la{1cp312}\ba
\| \psi \|_{L^{2(2\beta \ga+1)}} \le C \| \na \psi \|_{L^{(2\beta \ga+1)/(\beta \ga+1)}}
\le C \| \n u \|_{L^{(2\beta \ga+1)/(\beta \ga+1)}}
\le C \| \n \|^{1/2}_{ L^{2\beta \ga+1} },
\ea\ee
due to (\ref{1cp01}) and Sobolev embedding.

Combining (\ref{1cp37}), (\ref{1cp38}), (\ref{1cp310}) and (\ref{1cp311}), we arrive at
\be\la{1cp313}\ba
\frac{d}{dt} \int \n g^{2\ga} dx
& \le C \left( \| \n g^{2\ga} \|_{L^1} + 1 \right)
\left( \| \na u \|^2_{L^2} + 1 \right),
\ea\ee
which together with Gr\"onwall's inequality and (\ref{1cp01}) implies (\ref{1cp03}) and completes the proof of Lemma \ref{cp1l3}.
\end{proof}

\begin{lemma}\la{cp1l4}
For any $p>4$, there exists a positive constant $C$ depending only on
$T$, $a$, $\ga$, $\mu$, $N_0$, $\tilde{M}_0$, $\beta$, and $E_1$ such that
\be\la{1cp04}\ba
\|\n u\|_{L^p} \le  C(p) R_T^{1+\beta/4+(\beta \eta_0)/2}  (1+ \|\na u\|_{L^2})^{ 1-2/p},
\ea\ee
with $\eta_0$ as in \eqref{cpsol2} and $R_T$ given by \eqref{mdsj}.
\end{lemma}
\begin{proof}
First, we set
\be\la{1cp41}\ba
\nu \triangleq \frac{\mu^{1/2}}{2(\mu+1)} R_T^{-\beta/2} \in (0,1/2].
\ea\ee
Multiplying $(\ref{nlckv})_2$ by $|u|^\nu u$ and integrating over $\rr$, after using (\ref{1cp02}), (\ref{1cp002}) and (\ref{1cp03}), we derive
\be\la{1cp42}\ba
& \frac{1}{(2+\nu)} \frac{d}{dt} \int \n |u|^{2+\nu} dx
+ \int |u|^\nu \left(\mu |\na u|^2 + (\mu+\lam) (\div u)^2 \right) dx \\
& \le \nu \int (\mu+\lam) |\div u| |u|^\nu |\na u| dx
+ C \int \n^\ga |u|^\nu |\na u|dx + C \int |u|^\nu |\na u| |\na d|^2 dx \\
& \le \frac{1}{2} \int (\mu+\lam) |u|^\nu (\div u)^2 dx
+ \frac{\nu^2}{2} \int (\mu+\lam) |u|^\nu |\na u|^2 dx
+ C \int \n^\ga |u|^\nu |\na u| dx \\
& \quad + \frac{\mu}{8} \int |u|^{\nu} |\na u|^2 dx + C \int |u|^\nu |\na d|^4 dx \\
& \le \frac{1}{2} \int (\mu+\lam) |u|^\nu (\div u)^2 dx
+ \frac{\mu}{8(\mu+1)} \int |u|^{\nu} |\na u|^2 dx + \frac{\mu}{4} \int |u|^{\nu} |\na u|^2 dx \\
& \quad + C \int \n |u|^{2+\nu} dx + C \int \n^{(2+\nu)\ga-\nu/2} dx
+ C \| u \bar{x}^{-\frac{a}{2}} \|_{L^4} \| \na d \bar{x}^{\frac{a}{2}} \|_{L^2} \| \na d \|^3_{L^{12}} + C \| \na d \|^4_{L^4} \\
& \le \frac{1}{2} \int (\mu+\lam) |u|^\nu (\div u)^2 dx
+ \frac{\mu}{2} \int |u|^{\nu} |\na u|^2 dx + C \int \n |u|^{2+\nu} dx
+ C \left(1 + \| \na u \|^2_{L^2} \right),
\ea\ee
which together with Gr\"onwall's inequality and (\ref{1cp01}) yields
\be\la{1cp43}\ba
\sup_{0 \le t \le T} \int \n |u|^{2+\nu} dx \le C.
\ea\ee
Next, using (\ref{jqgj1}), (\ref{1cp01}), (\ref{1cp001}), (\ref{1cp41}), and (\ref{1cp43}), we obtain for $p>4$ and $r=(p-2)(2+\nu)/\nu$,
\be\nonumber\ba
\| \n u \|_{L^p} & \le \|\n u\|_{L^{2+\nu}}^{2/p} \|\n u\|_{L^r}^{1-2/p} \\
& \le C R_T^{ (1+\nu)/p} \|\n^{1/(2+\nu)} u\|_{L^{2+\nu}}^{2/p}
\left( r^{\eta_0+1/2} R_T (1+ \|\na  u\|_{L^2}) \right)^{ 1-2/p} \\
& \le C(p) R_T^{ (1+\nu)/p} \left( R_T^{1+\beta/4+(\beta\eta_0)/2} (1+ \|\na  u\|_{L^2}) \right)^{ 1-2/p} \\
& \le C(p) R_T^{1+\beta/4+(\beta\eta_0)/2} (1+ \|\na  u\|_{L^2})^{ 1-2/p},
\ea\ee
which gives (\ref{1cp04}) and completes the proof of Lemma \ref{cp1l4}.
\end{proof}

\begin{lemma}\la{cp1l5}
There exists a positive constant $C$ depending only on
$T$, $a$, $\ga$, $\mu$, $N_0$, $\tilde{M}_0$, $\beta$, and $E_1$ such that
\be\la{1cp05}\ba
&\sup_{0\le t\le T} \log A_1^2(t) + \int_0^T\frac{A_2^2(t)}{A_1^2(t)}dt
\le C R_T^{4/3}.
\ea\ee
\end{lemma}
\begin{proof}
First, we conclude from (\ref{1cp31}) that $G$ and $\o$ satisfy
\be\la{1cp52}\ba
\Delta G = \div(\n \dot{u} + \na d \cdot \Delta d), \quad \mu \Delta \o = \na^\bot \cdot (\n \dot{u} + \na d \cdot \Delta d).
\ea\ee
The standard elliptic estimates imply that for any $p \in (1,\infty)$,
\be\la{1cp53}\ba
\| \na G \|_{L^p} + \| \na \o \|_{L^p} \le C \left( \| \n \dot{u} \|_{L^p} + \| \na d \cdot \Delta d \|_{L^p} \right).
\ea\ee
Consequently, by virtue of (\ref{1cp002}), (\ref{1cp53}) and H\"older's inequality, we have
\be\la{1cp54}\ba
\| \na G \|_{L^2} + \| \na \o \|_{L^2}
& \le C \left( \| \n \dot{u} \|_{L^2} + \| \na d \cdot \Delta d \|_{L^2} \right) \\
& \le C R_T^{1/2} \| \sqrt{\n} \dot{u} \|_{L^2} + C \| \na d \|_{L^4} \| \Delta d \|_{L^4} \\
& \le C R_T^{1/2} \| \sqrt{\n} \dot{u} \|_{L^2} + C \| \Delta d \|_{L^2} + C \| \na \Delta d \|_{L^2} \\
& \le C A_1 + C R_T^{1/2} A_2.
\ea\ee

Next, multiplying (\ref{1cp31}) by $2 \dot{u}$ and integrating by parts over $\rr$, we derive
\be\la{1cp55}\ba
& \frac{d}{dt} \int \left(\mu \o^2 + \frac{G^2}{2\mu + \lam}\right)dx
+ 2 \int \n |\dot{u}|^2 dx \\
& = - \mu \int \o^2 \div u dx + 4 \int G \nabla u^1 \cdot\nabla^{\perp}u^2 dx
-2 \int G (\div u)^2 dx \\
& \quad - \int \frac{ (\beta-1)\lam - 2\mu }{(2\mu + \lam)^2} G^2 \div u dx
- 2 \beta \int \frac{ \lam P }{ (2\mu +\lam)^2 } G \div u dx \\
& \quad + 2 \ga \int \frac{P}{2\mu +\lam} G \div u dx
-2 \int \dot{u} \cdot \na d \cdot \Delta d dx
\triangleq \sum_{i=1}^7 I_i,
\ea\ee
where we have used the identities:
\be\ba\la{bp701}
\na^{\bot}\cdot \dot u= \frac{D}{Dt}\o - (\p_1 u\cdot\na) u^2
+ (\p_2 u \cdot \na) u^1
= \frac{D}{Dt}\o + \o \div u,
\ea\ee
and
\be\ba\la{bp7001}
\div \dot u&=\frac{D}{Dt}\div u +(\p_1u\cdot\na) u^1
+(\p_2u\cdot\na)u^2 \\ 
& = \frac{D}{Dt} \left( \frac{G}{2\mu+\lam} \right)
+ \frac{D}{Dt} \left( \frac{P}{2\mu+\lam} \right)
- 2\nabla u^1\cdot\nabla^{\perp}u^2 + (\div u)^2.
\ea\ee
In view of (\ref{1cp01}), (\ref{1cp03}) and (\ref{1cp54}),
and following the arguments as in \cite[Lemma 3.5]{HL3}, we arrive at
\be\la{1cp56}\ba
\sum_{i=1}^6 I_i \le \frac{1}{4} A^2_2 + C R_T^{4/3} A^2_1 (1 + \| \na u \|^2_{L^2} ).
\ea\ee

We now estimate $I_7$.
Integrating by parts over $\rr$ and using (\ref{1cp01}), (\ref{1cp02}), (\ref{1cp002}), (\ref{jqgju}) and Young's inequality yields
\be\la{1cp57}\ba
I_7 & = - 2 \int u_t \cdot \na d \cdot \Delta d dx
- 2 \int u \cdot \na u \cdot \na d \cdot \Delta d dx \\
& = 2 \frac{d}{dt} \int (\na d \odot \na d) \cdot \na u dx
- 2 \int (\na d_t \odot \na d)\cdot\na u dx \\
& \quad - 2 \int(\na d \odot \na d_t)\cdot\na u dx
- \int \div u_t |\na d|^2 dx - 2 \int u \cdot \na u \cdot \na d \cdot \Delta d dx \\
& \le \frac{d}{dt} \int \left( 2 (\na d \odot \na d) \cdot \na u - |\na d|^2 \div u \right) dx
+ C \| \nabla d_t \|_{L^2} \| \na d \|^{\frac{1}{2}}_{L^2} \| \na^3 d \|^{\frac{1}{2}}_{L^2} \| \na u \|_{L^2} \\
& \quad + C \| u \bar{x}^{-\frac{a}{4}} \|_{L^8} \| \na u \|_{L^2}
\| \na d \bar{x}^{\frac{a}{2}} \|^{\frac{1}{2}}_{L^2} \| \na d \|^{\frac{1}{2}}_{L^8}
\| \Delta d \|^{\frac{1}{8}}_{L^{2}} \| \na \Delta d \|^{\frac{7}{8}}_{L^{2}} \\
& \le \frac{d}{dt} \int \left( 2 (\na d \odot \na d) \cdot \na u - |\na d|^2 \div u \right) dx
+ \frac{1}{8} \| \nabla d_t \|^2_{L^2} + C A_2 \| \na u \|^2_{L^2} \\
& \quad + C ( 1 + \| \na u \|^2_{L^2} ) A_1^{\frac{1}{8}} A_2^{\frac{7}{8}} \\
& \le \frac{d}{dt} \int \left( 2 (\na d \odot \na d) \cdot \na u - |\na d|^2 \div u \right) dx
+ \frac{1}{4} A^2_2 + C A^2_1 (1 + \| \na u \|^2_{L^2} ),
\ea\ee
where we have used the following estimate:
\be\la{1cp057}\ba
\| \na u \|^2_{L^2} & \le C \left( \| \div u \|^2_{L^2} + \| \o \|^2_{L^2} \right)
\le C \left( \left\| \frac{G}{2\mu+\lam} \right\|^2_{L^2} + \| P \|^2_{L^2} + \| \o \|^2_{L^2} \right)
\le C A^2_1,
\ea\ee
due to (\ref{gw}) and (\ref{1cp03}).

On the other hand, multiplying (\ref{1cp21}) by $\nabla d_t-\Delta \nabla d,$ integrating the resulting equality over $\mathbb{R}^2$, and using (\ref{gn11}), (\ref{1cp01}), (\ref{1cp02}), (\ref{1cp002}), and (\ref{jqgju}), we obtain
\be\la{1cp58}\ba
& \frac{d}{dt} \| \Delta d \|^2_{L^2} + \| \na d_t \|^2_{L^2} + \| \na \Delta d \|^2_{L^2} \\
& \le C \int |\na (u \cdot \na d)|^2 + |\na (|\na d|^2 d)|^2 dx \\
& \le C \int |\na u|^2 |\na d|^2 + |u|^2 |\na^2 d|^2
+ |\na^2 d|^2 |\na d|^2 + |\na d|^6 dx \\
& \le C \| \na u \|^2_{L^2} \| \na d \|_{L^2} \| \na^3 d \|_{L^2}
+ C \| u \bar{x}^{-\frac{a}{8}} \|^2_{L^{16}}
\| \na^2 d \bar{x}^{\frac{a}{2}} \|^{\frac{1}{2}}_{L^2}
\| \na^2 d \|^{\frac{3}{2}}_{L^{\frac{12}{5}}} \\
& \quad + C \| \na^2 d \|^2_{L^4} \| \na d \|^2_{L^4} + C \| \na d \|^6_{L^6} \\
& \le C \| \na u \|^2_{L^2} \| \na^3 d \|_{L^2}
+ C (1 + \| \na u \|^2_{L^2}) \| \na^2 d \bar{x}^{\frac{a}{2}} \|^{\frac{1}{2}}_{L^2}
\| \na d\|^{\frac{3}{4}}_{L^3} \| \na^3 d\|^{\frac{3}{4}}_{L^2} \\
& \quad + C \| \na^2 d \|_{L^2} \| \na^3 d \|_{L^2} + C \\
& \le \frac{1}{8} \| \na \Delta d \|^2_{L^2} + C \| \na u \|^4_{L^2} + C \| \na^2 d \|^2_{L^2}
+ C (1 + \| \na u \|^2_{L^2})^{\frac{8}{5}} \| \na^2 d \bar{x}^{\frac{a}{2}} \|^{\frac{4}{5}}_{L^2} + C \\
& \le \frac{1}{8} \| \na \Delta d \|^2_{L^2}
+ C A^2_1 \left( 1 + \| \na u \|^2_{L^2} + \| \na^2 d \bar{x}^{\frac{a}{2}} \|^2_{L^2} \right).
\ea\ee
Substituting (\ref{1cp56}) and (\ref{1cp57}) into (\ref{1cp55}) and adding the result to (\ref{1cp58}) gives
\be\la{1cp59}\ba
\frac{d}{dt} \tilde{A}_1 + \frac{1}{4}A^2_2
\le C R_T^{4/3} A^2_1 \left( 1 + \| \na u \|^2_{L^2} + \| \na^2 d \bar{x}^{\frac{a}{2}} \|^2_{L^2} \right),
\ea\ee
where
\be\la{1cp510}\ba
\tilde{A}_1(t) \triangleq A^2_1(t) - \int \left( 2 (\na d \odot \na d) \cdot \na u - |\na d|^2 \div u \right) dx.
\ea\ee
Note that (\ref{1cp002}), (\ref{1cp057}) and Young's inequality imply
\be\la{1cp511}\ba
\left| \int \left(2 (\na d \odot \na d) \cdot \na u - |\na d|^2 \div u \right) dx \right|
\le C \| \na d \|^2_{L^4} \| \na u \|_{L^2}
\le \frac{1}{2} A^2_1 + \hat{C}_1,
\ea\ee
which together with (\ref{1cp510}) shows that
\be\la{1cp512}\ba
\frac{1}{2} A^2_1(t) \le \tilde{A}_1(t) + \hat{C}_1 \le 2 \left( A^2_1(t) + \hat{C}_1 \right).
\ea\ee

Finally, dividing (\ref{1cp59}) by $\tilde{A}_1(t) + \hat{C}_1$ and integrating over $(0,T)$,
we arrive at (\ref{1cp05}) after using (\ref{1cp01}), (\ref{1cp02}) and (\ref{1cp512}) and complete the proof of Lemma \ref{cp1l5}.
\end{proof}

\begin{lemma}\la{cp1l6}
There exists a positive constant $C$ depending only on
$T$, $a$, $\ga$, $\mu$, $N_0$, $\tilde{M}_0$, $\beta$, and $E_1$ such that
\be\ba\la{1cp06}
& \sup_{0\leq t\leq T} \left( \|\n\|_{L^\infty} + \| \na u \|_{L^2} + \| \na d \|_{H^1} \right) \\
& + \int_0^T \| \na u \|^2_{L^2} + \| \sqrt{\n} \dot{u} \|^2_{L^2} + \| \na^2 d \|^2_{H^1} + \| \na d_t \|^2_{L^2} dt
\le C.
\ea\ee
\end{lemma}
\begin{proof}
First, it follows from (\ref{gn11}), (\ref{1cp03}), (\ref{gw}), and (\ref{1cp54}) that
\be\la{1cp66}\ba
\| \na u \|_{L^4}
& \le C ( \|\div u\|_{L^4} + \|\omega\|_{L^4} ) \\
& \le C \left\| \frac{G}{2\mu+\lambda} \right\|_{L^4}
+ C \left\| \frac{P}{2\mu+\lambda} \right\|_{L^4} + C \| \o \|_{L^4} \\
& \le C \left\| G \right\|_{L^4} + C R^\ga_T + C \| \o \|_{L^4} \\
& \le C \| G \|^{1/2}_{L^2} \| \na G \|^{1/2}_{L^2}
+ C R^\ga_T + C \| \o \|^{1/2}_{L^2} \| \na \o \|^{1/2}_{L^2} \\
& \le C R^{\beta/4}_T A_1^{1/2}
\left( A_1 + R_T^{1/2} A_2 \right)^{1/2} + C R^\ga_T \\
& \le C R_T^{2\beta\ga} A_1 \left(1+\frac{A_2^2}{A_1^2} \right)^{1/4}.
\ea\ee
This, combined with (\ref{jhz1}), (\ref{jhz2}), and (\ref{1cp04}), implies that for any $p>4$,
\be\la{1cp67}\ba
\| F_1 \|_{L^\infty}
& \le C(p) \| F_1 \|_{L^p}^{1-4/p} \|\na F_1\|_{L^{4p/(p+4)}}^{4/p} \\
& \le C(p) \left( \|\na u\|_{L^2} \|\n  u\|_{L^p} \right)^{1-4/p}
\left(\|\na u\|_{L^4} \|\n  u\|_{L^p} \right)^{4/p} \\
& \le C(p) \|\na u\|_{L^2}^{1-4/p} \|\na u\|_{L^4}^{4/p} \|\n  u\|_{L^p} \\
& \le C(p) R_T^{1 +\beta/4+(\beta\eta_0)/2+(8\beta\ga)/p} A_1^{ 2-2/p}
\left(1+\frac{A_2^2}{A_1^2} \right)^{1/p} \\
& \le C(p) R_T^{1 +(\beta/4+(\beta\eta_0)/2)p/(p-1)+(9\beta\ga)/(p-1)} A_1^2
+ C \frac{A_2^2}{A_1^2}.
\ea\ee
On the other hand, using the Gagliardo-Nirenberg inequality, (\ref{1cp002}), and (\ref{1cp310}), we derive
\be\la{1cp68}\ba
\| F_2 \|_{L^\infty}
& \le C \| F_2 \|^{1/2}_{L^4} \| \na F_2 \|^{1/2}_{L^4}
\le C \| \na d \cdot \Delta d \|^{1/2}_{L^4} \\
& \le C \| \na d \|^{1/2}_{L^8} \| \Delta d \|^{1/2}_{L^8}
\le C \left( \| \Delta d \|_{L^2} + \| \na \Delta d \|_{L^2} \right)^{1/2} \\
& \le C A_1 + C A^{1/2}_1 \left( \frac{A^2_2}{A^2_1} \right)^{1/4}
\le C A^2_1 + \frac{A^2_2}{A^2_1}.
\ea\ee
Moreover, we conclude from (\ref{bwi}), (\ref{1cp01}), (\ref{1cp04}) and (\ref{1cp05}) that
\be\la{1cp69}\ba
\|\psi\|_{L^\infty}
& \le C \left( \|\psi\|_{L^{2\ga}} + \| \na \psi \|_{L^2} \right)
\log^{1/2} ( e + \|\psi \|_{W^{1,2\ga}} ) + C \\
& \le C \left( 1 + \| \n u\|_{L^2} \right) \log^{1/2}(e + \|\n u \|_{L^{2\ga}}) + C \\
& \le C R_T^{1/2} \log^{1/2} \left( R_T^{1+\beta/4+(\beta\eta_0)/2} (e+\|\na u \|_{L^2}) \right) + C \\
& \le C R_T^{1/2} \log^{1/2}(e+ A^2_1) + C R_T \\
& \le C R_T^{4/3},
\ea\ee
where in the second inequality we have used the following estimate:
\be\la{1cp610}\ba
\|\psi\|_{L^{2\ga}} \le C \| \na \psi \|_{L^{(2\ga)/(\ga+1)}}
\le C \| \n u \|_{L^{(2\ga)/(\ga+1)}}
\le C \| \n \|^{1/2}_{L^\ga} \| \sqrt{\n} u \|_{L^2} \le C,
\ea\ee
owing to (\ref{1cp01}) and Sobolev embedding.

Hence, integrating (\ref{1cp35}) over $(0,T)$ and using (\ref{1cp67}), (\ref{1cp68}) and (\ref{1cp69}), after choosing $p$ suitably large, we have
\be\la{1cp611}\ba
R_T^\beta \le C R_T^{1 + \beta/4 + \beta\eta_0},
\ea\ee
which together with $\beta>4/3$ and the definition of $\eta_0$ implies
\be\la{1cp612}\ba
\sup_{0 \le t \le T} \| \n \|_{L^\infty} \le C.
\ea\ee

Finally, combining (\ref{1cp612}), (\ref{1cp01}) and (\ref{1cp05}) yields (\ref{1cp06}) and completes the proof of Lemma \ref{cp1l6}.
\end{proof}

\begin{lemma}\la{cp1l7}
There exists a positive constant $C$ depending only on
$T$, $a$, $\ga$, $\mu$, $N_0$, $\tilde{M}_0$, $\beta$, and $E_1$ such that
\be\ba\la{1cp07}
\sup_{0\le t\le T}
t \left( \| \sqrt{\n} \dot{u} \|^2_{L^2} + \| \na d_t \|^2_{L^2} \right)
+ \int_0^{T} t \left( \| \na \dot{u} \|^2_{L^2} + \| \na^2 d_t \|^2_{L^2} \right) dt \le C.
\ea\ee
\end{lemma}
\begin{proof}
First, applying $\p_t + \div(u \cdot)$ to $(\ref{nlckv})^j_2$ gives
\be\la{1cp71}\ba
& (\n \dot u^j)_t + \div(\n u \dot u^j) - \mu \Delta \dot u^j - \p_j ((\mu+\lm) \div \dot u) \\
& = \mu \p_i ( -\p_i u \cdot \na u^j + \div u \p_i u^j) -\mu \div(\p_i u\p_i u^j) \\
& \quad -\p_j \left[ (\mu+\lambda) \pa_i u \cdot \na u^i - ( \mu+(1-\beta) \lam )(\div u)^2\right] \\
& \quad - \div (\p_j u (\mu+\lambda) \div u) + (\ga-1) \p_j (P\div u) + \text{div} (P\p_ju) \\
& \quad - \p_k ( \p_j d \cdot \p_k d )_t + \frac{1}{2} \p_j ( |\na d|^2 )_t - \div (u \p_j d \cdot \Delta d).
\ea\ee
Then, multiplying (\ref{1cp71}) by $\dot{u}$ and integrating by parts over $\rr$,
we derive after using (\ref{1cp02}), (\ref{jqgju}), (\ref{1cp06}), (\ref{1cp66}) and Young's inequality that
\be\la{1cp72}\ba
& \frac{1}{2} \frac{d}{dt} \int \n |\dot u|^2 dx
+ \mu \int |\na \dot u|^2 dx + \int (\mu+\lambda) (\div \dot u)^2 dx \\
& \le \frac{\mu}{4} \| \na \dot{u} \|^2_{L^2}
+ C \left( \|\na u\|_{L^4}^4 + \|\nabla u\|^2_{L^2}
+ \| \na d \|^2_{L^4} \| \na d_t \|^2_{L^4} + \| |u| |\na d| |\Delta d| \|^2_{L^2} \right) \\
& \le \frac{\mu}{4} \| \na \dot{u} \|^2_{L^2} + C (1 + A^2_2)
+ C \| \na d_t \|_{L^2} \| \na^2 d_t \|_{L^2} \\
& \quad + C \| u \bar{x}^{-a/4} \|^2_{L^8} \| \na d \bar{x}^{a/2} \|_{L^2} \| \na d \|_{L^8} \| \Delta d \|^2_{L^{16}} \\
& \le \frac{\mu}{4} \| \na \dot{u} \|^2_{L^2} + \frac{1}{4} \| \na^2 d_t \|^2_{L^2} + C (1 + A^2_2).
\ea\ee
On the other hand, differentiating (\ref{1cp21}) with respect to $t$ yields
\be\la{1cp73}\ba
\nabla d_{tt}-\Delta\nabla d_{t} =-\nabla (u\cdot \nabla d)_{t}+\nabla (|\nabla d|^{2}d)_{t}.
\ea\ee
Then, we multiply (\ref{1cp73}) by $\na d_t$, integrate over $\rr$ and apply
(\ref{1cp02}), (\ref{jqgju}), (\ref{1cp24}), (\ref{1cp06}), (\ref{1cp66}), and Young's inequality to obtain
\be\la{1cp74}\ba
& \frac{1}{2}\frac{d}{dt}\|\nabla d_{t}\|_{L^{2}}^{2} + \| \na^2 d_{t}\|_{L^{2}}^{2} \\
& = \int \left( u_t \cdot \na d + u \cdot \na d_t - (|\nabla d|^{2}d)_{t} \right) \cdot \Delta d_t dx \\
& = \int \left( \dot{u} \cdot \na d - u \cdot \na u \cdot \na d
+ u \cdot \na d_t - (|\nabla d|^{2}d)_{t} \right) \cdot \Delta d_t dx \\
& = -\int \left( \p_i \dot{u} \cdot \na d + \dot{u} \cdot \na \p_i d \right) \cdot \p_i d_t dx
- \int \left( u \cdot \na u \cdot \na d
- u \cdot \na d_t + (|\nabla d|^{2}d)_{t} \right) \cdot \Delta d_t dx \\
& \le C \| \na \dot{u} \|_{L^2} \| \na d \|_{L^4}\| \na d_t \|_{L^4}
+ C \| \dot{u} \bar{x}^{-\frac{a}{4}} \|_{L^8}
\| \na^2 d \bar{x}^{\frac{a}{2}} \|^{\frac{1}{2}}_{L^2}
\| \na^2 d \|^{\frac{1}{2}}_{L^4}\| \na d_t \|_{L^2} \\
& \quad + C \| u \bar{x}^{-\frac{a}{8}} \|_{L^{16}} \| \na u \|_{L^4}
\| \na d \bar{x}^{\frac{a}{2}} \|^{\frac{1}{4}}_{L^2} \| \na d \|^{\frac{3}{4}}_{L^{12}} \| \na^2 d_t \|_{L^2}
+ C \| \na u \|_{L^2} \| \na d_t \|^2_{L^4} \\
& \quad + C \left( \| \na d \|^2_{L^8} \| d_t \|_{L^4}
+ \| \na d \|_{L^4} \| \na d_t \|_{L^4}  \right) \| \Delta d_t \|_{L^2} \\
& \le \frac{1}{4} \| \na^2 d_t \|^2_{L^2} + \frac{\mu}{4} \| \na \dot{u} \|^2_{L^2}
+ C \left( 1 + \| \nabla d_t \|^2_{L^{2}} \right)
\left( 1 + A^2_2 + \| \nabla^{2}d \bar{x}^{\frac{a}{2}} \|^2_{L^{2}} \right),
\ea\ee
where we have used the following estimate:
\be\la{1cp75}\ba
\| d_t \|^2_{L^4} & \le C \left( \| u \cdot \na d \|^2_{L^4}
+ \| \na^2 d \|^2_{L^4} + \| |\na d|^2 \|^2_{L^4} \right) \\
& \le C \left( \| u \bar{x}^{-\frac{a}{8}} \|^2_{L^\infty}
\| \na d \bar{x}^{\frac{a}{8}} \|^2_{L^4} + \| \na^2 d \|^2_{H^1}
+ \| \na d \|^4_{L^8} \right) \\
& \le C \left( 1 + A^2_2 + \| \nabla^{2}d \bar{x}^{\frac{a}{2}} \|^2_{L^{2}} \right),
\ea\ee
due to (\ref{jqgju}), (\ref{1cp24}), (\ref{1cp06}), and (\ref{1cp66}).

Combining (\ref{1cp72}) and (\ref{1cp74}) leads to
\be\la{1cp76}\ba
& \frac{d}{dt} \left( \| \sqrt{\n} \dot{u} \|^2_{L^2} + \|\nabla d_{t}\|_{L^{2}}^{2} \right)
+ \mu \| \na \dot{u} \|^2_{L^2} + \| \na^2 d_{t}\|_{L^{2}}^{2} \\
& \le C \left( 1 + \| \sqrt{\n} \dot{u} \|^2_{L^2} + \| \nabla d_t \|^2_{L^{2}} \right)
\left( 1 + A^2_2 + \| \nabla^{2}d \bar{x}^{\frac{a}{2}} \|^2_{L^{2}} \right).
\ea\ee

Multiplying (\ref{1cp76}) by $t$ and applying Gr\"onwall's inequality, we arrive at
(\ref{1cp07}) after using (\ref{1cp02}) and (\ref{1cp06}).
This completes the proof of Lemma \ref{cp1l7}.
\end{proof}

\begin{lemma}\la{cp1l8}
There exists a positive constant $C$ depending only on
$T$, $a$, $\ga$, $\mu$, $N_0$, $\tilde{M}_0$, $\beta$, and $E_1$ such that
\be\ba\la{1cp08}
\sup_{0\leq t\leq T} t \left( \|\nabla^{2} d\bar{x}^{\frac{a}{2}}\|_{L^{2}}^{2} + \| \na^3 d \|^2_{L^2} \right)
+ \int_{0}^{T} t\|\nabla^{3} d \bar{x}^{\frac{a}{2}}\|_{L^{2}}^{2} dt
\leq C.
\ea\ee
\end{lemma}
\begin{proof}
First, multiplying (\ref{1cp21}) by $-\Delta \nabla d \bar{x}^{a}$ and
integrating by parts over $\rr$, we get after using (\ref{jqgju}),
(\ref{1cp02}), (\ref{1cp24}), (\ref{1cp06}) and (\ref{1cp66}) that
\be\la{1cp81}\ba
& \frac{1}{2}\frac{d}{dt} \|\nabla^{2}d\bar{x}^{\frac{a}{2}}\|_{L^{2}}^{2}
+ \|\nabla^{3} d\bar{x}^{\frac{a}{2}}\|_{L^{2}}^{2} \\
\le & C \int |\nabla d_{t}| |\nabla^{2} d| |\nabla \bar{x}^{a}| dx
+ \int \nabla(u\cdot\nabla d)\nabla \Delta d\bar{x}^{a} dx
-\int \nabla(|\nabla d|^{2}d)\nabla \Delta d\bar{x}^{a} dx \\
& + C \int |\nabla^{3}d| |\nabla^{2}d| |\nabla\bar{x}^{a}| dx \\
\le & C \int \left( |\nabla d_{t}| + | \na^3 d | \right) |\nabla^{2} d| \bar{x}^{\frac{a}{2}} dx
+ C \int |\na u| \left( |\nabla d| |\nabla^{3} d| + |\nabla^2 d|^2 \right) \bar{x}^{a} dx \\
& + C \int |u| |\nabla^{2}d|^{2} |\nabla\bar{x}^{a}| dx
+ C \int |\nabla d| |\nabla^3 d| \left( |\nabla^2 d| + |\na d|^2 \right) \bar{x}^{a} dx \\
& \le C A_2 \| \na^2 d \bar{x}^{\frac{a}{2}} \|_{L^{2}}
+ C \| \na u \|_{L^4} \| \na d \bar{x}^{\frac{a}{2}} \|_{L^4}
\| \na^3 d \bar{x}^{\frac{a}{2}} \|_{L^{2}}
+ C \| \na u \|_{L^2} \| \na^2 d \bar{x}^{\frac{a}{2}} \|^2_{L^{4}} \\
& + C \|\nabla^2 d \bar{x}^{\frac{a}{2}}\|_{L^{4}}
\|\nabla^2 d \bar{x}^{\frac{a}{2}}\|_{L^{2}}
\|u\bar{x}^{-\frac{3}{4}}\|_{L^{4}}
+ C \| \na d \bar{x}^{\frac{a}{2}} \|_{L^4} \| \na^3 d \bar{x}^{\frac{a}{2}} \|_{L^2}
\left( \| \na^2 d \|_{L^4} + \| \na d \|^2_{L^8} \right) \\
& \le \frac{1}{4} \| \na^3 d \bar{x}^{\frac{a}{2}} \|^2_{L^2}
+ C \| \na^2 d \bar{x}^{\frac{a}{2}} \|^2_{L^{4}}
+ C \left( 1 + A^2_2 + \| \nabla^{2}d \bar{x}^{\frac{a}{2}} \|^2_{L^{2}} \right) \\
& \le \frac{1}{4} \| \na^3 d \bar{x}^{\frac{a}{2}} \|^2_{L^2}
+ C \|\nabla^2 d \bar{x}^{\frac{a}{2}}\|_{L^{2}} \|\nabla (\nabla^2 d \bar{x}^{\frac{a}{2}})\|_{L^{2}}
+ C \left( 1 + A^2_2 + \| \nabla^{2}d \bar{x}^{\frac{a}{2}} \|^2_{L^{2}} \right) \\
& \le \frac{1}{2} \| \na^3 d \bar{x}^{\frac{a}{2}} \|^2_{L^2}
+ C \left( 1 + A^2_2 + \| \nabla^{2}d \bar{x}^{\frac{a}{2}} \|^2_{L^{2}} \right),
\ea\ee
which yields
\be\la{1cp82}\ba
\frac{d}{dt} \|\nabla^{2}d\bar{x}^{\frac{a}{2}}\|_{L^{2}}^{2}
+ \|\nabla^{3} d\bar{x}^{\frac{a}{2}}\|_{L^{2}}^{2}
\le C \left( 1 + A^2_2 + \| \nabla^{2}d \bar{x}^{\frac{a}{2}} \|^2_{L^{2}} \right).
\ea\ee
Hence, multiplying (\ref{1cp82}) by $t$, integrating over $(0,T)$ and using (\ref{1cp02}) and (\ref{1cp06}), we obtain
\be\ba\la{1cp83}
\sup_{0\leq t\leq T} t \|\nabla^{2} d\bar{x}^{\frac{a}{2}}\|_{L^{2}}^{2}
+ \int_{0}^{T} t\|\nabla^{3} d \bar{x}^{\frac{a}{2}}\|_{L^{2}}^{2} dt
\leq C.
\ea\ee
Furthermore, applying standard elliptic estimates to (\ref{1cp21})
and using (\ref{jqgju}), (\ref{1cp06}) and Young's inequality, we conclude
\be\la{1cp84}\ba
\|\nabla^{3}d\|_{L^{2}}^{2}
\leq & C \left( \|\nabla d_{t}\|_{L^{2}}^{2} + \||\nabla u||\nabla d|\|_{L^{2}}^{2}
+ \||u||\nabla^{2}d|\|_{L^{2}}^{2} + \||\nabla d|^{3}\|_{L^{2}}^{2}
+ \||\nabla^{2}d||\nabla d|\|_{L^{2}}^{2} \right) \\
\leq & C \left( \|\nabla d_{t}\|_{L^{2}}^{2} + \|\nabla u\|_{L^{4}}^{2} \|\nabla d\|_{L^{4}}^{2}
+ \|u\bar{x}^{-\frac{a}{4}}\|_{L^{8}}^{2} \|\nabla^{2}d\bar{x}^{\frac{a}{2}}\|_{L^{2}}\|\nabla^{2}d\|_{L^{4}} \right) \\
& + C \left( \|\nabla d\|_{L^{6}}^{6} + \| \na d \|^2_{L^4} \| \na^2 d \|^2_{L^4} \right) \\
\leq & C \left( \|\nabla d_{t}\|_{L^{2}}^{2} + A_2
+ \|\nabla^{2}d\bar{x}^{\frac{a}{2}}\|_{L^{2}}^{2}
+ \|\nabla^{2}d\|_{L^{4}}^{2} + 1 \right) \\
\leq & C \left( \|\nabla d_{t}\|_{L^{2}}^{2} + \| \sqrt{\n} \dot{u} \|_{L^2}
+ \|\nabla^{2}d\bar{x}^{\frac{a}{2}}\|_{L^{2}}^{2} + \| \na^3 d \|_{L^2} + 1 \right) \\
\leq & \frac{1}{2} \|\nabla^{3}d\|_{L^{2}}^{2}
+ C \left( \|\nabla d_{t}\|_{L^{2}}^{2} + \| \sqrt{\n} \dot{u} \|^2_{L^2}
+ \|\nabla^{2}d\bar{x}^{\frac{a}{2}}\|_{L^{2}}^{2} + 1 \right),
\ea\ee
which together with (\ref{1cp07}) and (\ref{1cp83}) gives
\be\la{1cp85}\ba
\sup_{0\leq t\leq T} t \|\nabla^{3}d\|_{L^{2}}^{2} \leq C.
\ea\ee
Combining this with (\ref{1cp83}) implies (\ref{1cp08}) and completes the proof of Lemma \ref{cp1l8}.
\end{proof}

\begin{lemma}\la{cp1l9}
There exists a positive constant $C$ depending only on
$T$, $a$, $\ga$, $\mu$, $N_0$, $\tilde{M}_0$, $\beta$, $q$, $\| \n_0 \|_{W^{1,q}}$, and $E_1$ such that
\be\la{1cp09}\ba
&\sup_{0\le t\le T} \left( \| \n \|_{W^{1,q}} + t \| \na^2 u \|^2_{L^2} \right) \\
& + \int_0^T \left( \| \na^2 u \|^2_{L^2} + \|\nabla^2 u\|^{(q+1)/q}_{L^q}
+ t \|\nabla^2 u\|_{L^q}^2 + t \| \na^4 d \|^2_{L^2} \right) dt\le C.
\ea\ee
\end{lemma}
\begin{proof}
First, define $\Phi = (\Phi^1,\Phi^2)$, where $\Phi^i \triangleq (2\mu+\lam(\n)) \p_i \n$ $(i=1,2)$. It follows from $(\ref{nlckv})_1$ that $\Phi^{i}$ satisfies
\be\la{bpg21}\ba
\p_t \Phi^i + (u \cdot \na) \Phi^i
+ (2\mu+\lam(\n)) \na \n \cdot \p_i u + \n \p_i G + \n \p_i P
+ \Phi^i \div u = 0.
\ea\ee
Multiplying (\ref{bpg21}) by $|\Phi|^{q-2} \Phi^i$, summing over $i=1,2$, and 
integrating by parts over $\mathbb{R}^2$, we obtain 
\be\la{bpg22}\ba
\frac{d}{dt} \| \Phi \|_{L^q}
& \le C ( 1 + \| \na u \|_{L^\infty} ) \| \Phi \|_{L^q}
+ C \| \na G \|_{L^q}.
\ea\ee
By (\ref{1cp002}), (\ref{1cp53}), (\ref{1cp06}), and H\"older's inequality, we have
\be\la{bpg202}\ba
\| \na G \|_{L^q} + \| \na \o \|_{L^q} & \le C \left( \| \n \dot{u}\|_{L^q}
+ \| \na d \cdot \Delta d \|_{L^q}\right) \\
& \le C \left( 1 + A_2 +\| \n \dot{u}\|_{L^q}\right),
\ea\ee
which together with the Gagliardo-Nirenberg inequality and (\ref{1cp06}) yields
\be\la{bpg23}\ba
& \| \div u \|_{L^\infty}+\| \o \|_{L^\infty} \\
& \le C \left( \| G \|_{L^\infty} + \| P \|_{L^\infty} \right)
+ \| \o \|_{L^\infty} \\
& \le C \left( 1 + \| G \|_{L^2}^{\frac{q-2}{2(q-1)}}
\| \na G \|_{L^q}^{\frac{q}{2(q-1)}}
+ \| \o \|_{L^2}^{\frac{q-2}{2(q-1)}} \| \na \o \|_{L^q}^{\frac{q}{2(q-1)}} \right) \\
& \le C \left( 1+ A_2+\| \n \dot{u} \|_{L^q} \right)^{\frac{q}{2(q-1)}}.
\ea\ee
Moreover, (\ref{dc1}), (\ref{1cp06}), (\ref{bpg202}), and (\ref{bpg23}) imply that
\be\la{bpg25}\ba
\|\na^2 u\|_{L^q}
& \le C \left( \| \nabla \div u \|_{L^q} + \|\nabla \o \|_{L^q} \right) \\
& \le C \left(1+\| \na ( (2\mu+\lam) \div u ) \|_{L^q}
+ \| \div u \|_{L^\infty} \| \na \n \|_{L^q}
+ \| \n \dot{u} \|_{L^q} \right) \\
&\le C( 1 + \| \na G \|_{L^q} + \|\na \n\|_{L^q} + \|\div u\|_{L^\infty}\|\na \n\|_{L^q}+\|\n \dot{u}\|_{L^q})\\
& \le C \left( 1 +A_2+ \| \n \dot{u} \|_{L^q}\right)
\left( e + \| \na \n \|_{L^q} \right).
\ea\ee
Combining (\ref{bpg23}), (\ref{1cp06}), and (\ref{bpg25}), and applying Lemma \ref{bkm}, we obtain 
\be\ba\la{bpg24}
\|\na u\|_{L^\infty} 
& \le C \left( \|\div u \|_{L^\infty}
+ \|\o\|_{L^\infty} \right) \log \left(e+ \|\na^2 u\|_{L^q} \right)+ C\|\na u\|_{L^2}+C \\
& \le C \left( 1 +A_2+ \| \n \dot{u} \|_{L^q} \right)
\log \left(e + \| \na \n \|_{L^q} \right).
\ea\ee
The definition of $\Phi$ and (\ref{1cp06}) give
\be\la{bpg26}\ba
2\mu \| \na \n \|_{L^q} \le \| \Phi \|_{L^q} \le C \| \na \n \|_{L^q},
\ea\ee
which along with (\ref{bpg22}) and (\ref{bpg24}) leads to
\be\la{bpg27}\ba
\frac{d}{dt} \log( e + \| \Phi \|_{L^q} )
& \le C \left( 1 +A_2+ \| \n \dot{u} \|_{L^q} \right)
\log \left(e + \| \Phi \|_{L^q} \right).
\ea\ee
Furthermore, it follows from (\ref{1cp06}), (\ref{jqgj1}), and H\"{o}lder's inequality that
\be\ba\nonumber
\| \rho \dot{u} \|_{L^q} 
& \le C\| \rho \dot{u} \|_{L^2}^{2(q-1)/(q^2-2)}
\| \n \dot{u} \|_{L^{q^2}}^{q(q-2)/(q^2-2)} \\
& \le C\| \rho \dot{u} \|_{L^2}^{2(q-1)/(q^2-2)} \left(\|\sqrt{\n} \dot{u}\|_{L^2}+\|\na \dot{u}\|_{L^2}\right)^{q(q-2)/(q^2-2)} \\
& \le  C \| \sqrt{\n} \dot{u} \|_{L^2}
+ C\| \sqrt{\n} \dot{u} \|_{L^2}^{2(q-1)/(q^2-2)}
\| \na \dot{u} \|_{L^2}^{q(q-2)/(q^2-2)},
\ea\ee
which together with (\ref{1cp07}) yields
\be\la{bpg29}\ba
&\int_0^T \left(\| \rho \dot{u} \|^{1+1 /q}_{L^q}
+ t \| \n \dot{u} \|^2_{L^q} \right) dt \le C.
\ea\ee
Consequently, from (\ref{bpg26}), (\ref{bpg27}), (\ref{bpg29}), and Gr\"onwall's inequality, we conclude that
\be\la{bpg210}\ba
\sup_{0 \le t \le T} \| \n \|_{W^{1,q}} \le C.
\ea\ee
This, together with (\ref{1cp06}), (\ref{1cp07}), (\ref{bpg25}), and (\ref{bpg29}), gives
\be\la{bpg211}\ba
\sup_{0\le t\le T} t \| \na^2 u \|^2_{L^2}
+ \int_0^T \left( \| \nabla^2 u \|^{2}_{L^2} + \| \nabla^2 u \|^{(q+1)/q}_{L^q}
+ t \|\nabla^2 u \|_{L^q}^2 \right) dt
\le C.
\ea\ee
Finally, applying the $\na^2$ operator to $(\ref{nlckv})_3$, we arrive at
\be\la{1cp92}\ba
\nabla^2 d_t + \nabla^2(u \cdot \na d) = \nabla^2 \Delta d + \nabla^2(|\nabla d|^2 d).
\ea\ee
Taking the $L^2$-inner product of (\ref{1cp92}) with $\nabla^2 \Delta d$, integrating by parts over $\rr$, and using (\ref{1cp06}), (\ref{gn11}), and H\"older's inequality, we derive
\be\la{1cp93}\ba
\| \na^4 d \|^2_{L^2}
& \le C \| \na^4 d \|_{L^2} \left( \| \na^2 d_t \|_{L^2} + \| |\na^2 u| |\na d| \|_{L^2} + \| |\na u| |\na^2 d| \|_{L^2} \right)
+ C \| \na u \|_{L^2} \| \na^3 d \|^2_{L^4} \\
& \quad + C \| \na^4 d \|_{L^2} \left( \| |\na^3 d| |\na d| \|_{L^2} + \| \na^2 d \|^2_{L^4} + \| |\na^2 d| |\na d|^2 \|_{L^2} \right) \\
& \le \frac{1}{2} \| \na^4 d \|^2_{L^2}
+ C \left( \| \na^2 d_t \|^2_{L^2} + \| \na^2 u \|^2_{L^q} + \| \na^2 u \|^2_{L^2} + \| \na^3 d \|^2_{L^2} + 1 \right),
\ea\ee
which together with (\ref{1cp06}), (\ref{1cp07}), and (\ref{bpg211}) yields
\be\la{1cp94}\ba
\int_0^T t \| \na^4 d \|^2_{L^2} dt \le C.
\ea\ee
Combining this with (\ref{bpg210}) and (\ref{bpg211}), we obtain (\ref{1cp09}), thereby completing the proof of Lemma \ref{cp1l9}.
\end{proof}

Finally, combining the above estimates with arguments similar to those in \cite[Lemma 4.3]{HL3}, we obtain the following result.

\begin{lemma}\la{cp1l10}
There exists a positive constant $C$ depending only on
$T$, $a$, $\ga$, $\mu$, $N_0$, $\tilde{M}_0$, $\beta$, $q$, $\| {\bar{x}}^a \rho_0 \|_{W^{1,q}}$,
$\| {\bar{x}}^a \rho_0 \|_{H^1}$, and $E_1$ such that
\be\ba\la{1cp010}
\sup_{0\leq t\leq T} \| \rho \bar{x}^a \|_{L^1 \cap H^{1}\cap W^{1,q}} \leq C.
\ea\ee
\end{lemma}

\subsection{The Case of Non-vacuum Far-field Density}

In this subsection, for initial data $(\n_0,u_0,d_0)$ satisfying (\ref{cp2sol1}) and $\n_0 >0$,
we assume that $(\n,u,d)$ is the strong solution to (\ref{nlckv})--(\ref{i30}), (\ref{bjtj2}) on $\rr \times (0,T]$ with $\tilde{\n}>0$.

Furthermore, we define
\be\la{e2}\ba
E_2 \triangleq & \| \n_0 - \tilde{\n} \|_{L^2 \cap L^\infty} + \| u_0 \|_{H^1}
+ \| \na^2 d_0 \|_{L^2} + \| \tilde{x}^\alpha \sqrt{\n_0} u_0 \|_{L^2} \\
& + \| \tilde{x}^{2 \alpha} K(\n_0) \|_{L^1} + \| \tilde{x}^\alpha \na d_0 \|_{L^2}.
\ea\ee

We begin with the standard energy estimate.

\begin{lemma}\la{cp2l1}
There exists a positive constant $C$ depending only on $T$, $\mu$, $\ga$, $\tilde{\n}$,
$\| K(\n_0) \|_{L^1}$, $\| \sqrt{\n_0} u_0 \|_{L^2}$, and $\| \na d_0 \|_{L^2}$ such that
\be\la{2cp01}\ba
& \sup\limits_{0\le t\le T} \int\left( K(\n) + \n |u|^2 + |\na d|^2 \right) dx \\
& + \int_0^T \int\left( \mu |\na u|^2+ \lambda(\n) (\div u)^2 + | \na^2 d |^2 \right) dxdt
\le C,
\ea\ee
and
\be\la{2cp001}\ba
\| u \|_{H^1} \le C \left( 1 + \| \na u \|_{L^2} \right), \quad t \in [0,T].
\ea\ee
\end{lemma}
\begin{proof}
First, adapting the proof of Lemma \ref{cp1l1}, we can obtain (\ref{2cp01}).

We now prove (\ref{2cp001}).
Note that
\be\la{2cp11}\ba
K(\n) = \frac{1}{\ga-1} \left[ \n^\ga - \ga \n (\tilde{\n})^{\ga-1} + (\ga-1) \tilde{\n}^\ga \right],
\ea\ee
which implies that there exists a positive constant $C$ depending only on $\tilde{\n}$ and $\ga$ such that
\be\la{2cp12}\ba
(\n - \tilde{\n})^2 \le C K(\n), \text{ if } \n < 2 \tilde{\n}, \quad
(\n - \tilde{\n})^\ga \le C K(\n), \text{ if } \n \ge 2 \tilde{\n}.
\ea\ee
This, together with (\ref{2cp01}), shows that
\be\la{2cp13}\ba
\| \n - \tilde{\n} \|_{L^2( \n < 2 \tilde{\n} )}
+ \| \n - \tilde{\n} \|_{L^\ga( \n \ge 2 \tilde{\n} )} \le C.
\ea\ee
Then, for any $v \in H^1(\rr)$, applying (\ref{gn11}) and (\ref{2cp13}), we obtain
\be\la{2cp14}\ba
\int |v|^2 dx & \le C \int \n |v|^2 dx + C \int |\n - \tilde{\n}| |v|^2 dx \\
& \le C \| \sqrt{\n} v \|^2_{L^2} + C \int_{ \left\{\n < 2 \tilde{\n} \right\} } |\n - \tilde{\n}| |v|^2 dx
+ C \int_{ \left\{\n \ge 2 \tilde{\n} \right\} } |\n - \tilde{\n}| |v|^2 dx \\
& \le C \| \sqrt{\n} v \|^2_{L^2} + C \| \n - \tilde{\n} \|_{L^2( \n < 2 \tilde{\n} )} \| v \|^2_{L^4}
+ C \| \n - \tilde{\n} \|_{L^\ga( \n \ge 2 \tilde{\n} )} \| v \|^2_{L^\frac{2\ga}{\ga-1}} \\
& \le C \| \sqrt{\n} v \|^2_{L^2} + C \| v \|_{L^2} \| \na v \|_{L^2}
+ C \| v \|^{\frac{2(\ga-1)}{\ga}}_{L^2} \| \na v \|^{\frac{2}{\ga}}_{L^2} \\
& \le \frac{1}{2} \| v \|^2_{L^2} + C \| \sqrt{\n} v \|^2_{L^2} + C \| \na v \|^2_{L^2},
\ea\ee
which yields
\be\la{2cp15}\ba
\| v \|_{L^2} \le C \| \sqrt{\n} v \|_{L^2} + C \| \na v \|_{L^2}.
\ea\ee
Combining this with (\ref{2cp01}) gives (\ref{2cp001}) and completes the proof of Lemma \ref{cp2l1}.
\end{proof}

Based on (\ref{2cp01}) and (\ref{2cp001}), following the approach in Lemma \ref{cp1l2},
we can establish the following $L^\infty(0,T;L^p)$ estimate of $\na d$.
\begin{lemma}\la{cp2l2}
For any $2<p<\infty$, there exists a positive constant $C$ depending only on 
$p$, $T$, $\mu$, $\ga$, $\tilde{\n}$,
$\| K(\n_0) \|_{L^1}$, $\| \sqrt{\n_0} u_0 \|_{L^2}$, and $\| \na d_0 \|_{H^1}$ such that
\be\la{2cp02}\ba
\sup_{0 \le t \le T} \| \na d \|_{L^p} \le C.
\ea\ee
\end{lemma}

\begin{lemma}\la{cp2l3}
There exists a positive constant $C$ depending only on $T$, $\alpha$, $\mu$, $\ga$, $\beta$, $\tilde{\n}$, and $E_2$ such that
\be\la{2cp03}\ba
\sup_{0\le t\le T} \left( \| \tilde{x}^{2 \alpha} K(\rho) \|_{L^1}
+ \| \tilde{x}^{\alpha} \sqrt{\n} u \|_{L^2} + \| \tilde{x}^\alpha \na d \|_{L^2} \right) \le C,
\ea\ee
and
\be\la{2cp003}\ba
\sup_{0\le t\le T} \int \left( |\n - \tilde{\n}|^2 + |\n - \tilde{\n}|^{4\beta\ga+1 }\right) dx \le C.
\ea\ee
\end{lemma}
\begin{proof}
First, we use the effective viscous flux $G$ defined in (\ref{gw}) to rewrite $(\ref{nlckv})_2$ as
\be\la{2cp31}\ba
\n\dot{u}= \na G + \mu \na^\bot \o - \na d \cdot \Delta d.
\ea\ee
Then, following a derivation analogous to Lemma \ref{cp1l3}, we obtain
\be\la{2cp32}\ba
\frac{D}{Dt} \left( \Lambda_2(\n) - \psi \right) + P - P(\tilde{\n}) = - F_1 - F_2,
\ea\ee
where
\be\la{2cp33}\ba
\Lambda_2(\n) \triangleq 2\mu (\log \n - \log \tilde{\n}) + \beta^{-1}(\n^\beta-\tilde{\n}^\beta),
\ea\ee
and $F_1$, $F_2$, $\psi$ are as defined in (\ref{1cp34}) and (\ref{1cp36}).

Moreover, from (\ref{1cp39}), (\ref{jhz1}), (\ref{2cp001}), and H\"older's inequality, we deduce that for any $2\leq p<4\beta \ga+1$,
\be\la{2cp34}\ba
\| F_1 \|_{L^p} & \le C \| u \|_{\mathcal{BMO}} \| \n u \|_{L^p} \\
& \le C \| \na u \|_{L^2} \left( \| (\n -\tilde{\n}) u \|_{L^p} + \| u \|_{L^p} \right) \\
& \le C \| \na u \|_{L^2} \left( \| \n -\tilde{\n} \|_{L^{4\beta\ga+1}} \| u \|_{ L^{ p(4\beta\ga+1)/(4\beta\ga+1-p) } }
+ \| u \|_{L^p} \right) \\
& \le C  \left( \| \n -\tilde{\n} \|_{L^{4\beta\ga+1}} + 1 \right) \| u \|^2_{H^1} \\
& \le C  \left( \| \n -\tilde{\n} \|_{L^{4\beta\ga+1}} + 1 \right)
\left( \| \na u \|^2_{L^2} + 1 \right),
\ea\ee
and
\be\la{2cp36}\ba
\| F_2 \|_{L^p} \le C \| \na d \|^2_{L^{2p}} \le C.
\ea\ee
In addition, the Sobolev inequality and H\"older's inequality imply
\be\la{2cp360}\ba
\| \psi(0) \|_{L^\infty} & = \| (-\Delta)^{-1} \div(\n_0 u_0) \|_{L^\infty}
\le C \| (-\Delta)^{-1} \div(\n_0 u_0) \|_{W^{1,{\frac{3}{\alpha}}}} \\
& \le C \| \na (-\Delta)^{-1} \div(\n_0 u_0) \|_{L^{\frac{6}{3+2\alpha}}}
+ C \| \na (-\Delta)^{-1} \div(\n_0 u_0) \|_{L^{\frac{3}{\alpha}}} \\
& \le C \| \n_0 u_0 \|_{L^{\frac{6}{3+2\alpha}}} + C \| \n_0 u_0 \|_{L^{\frac{3}{\alpha}}} \\
& \le C \| \n_0 \|_{L^\infty}^{\frac{1}{2}} \| \tilde{x}^\alpha \sqrt{\n_0} u_0 \|_{L^2} \| \tilde{x}^{-\alpha} \|_{ L^{\frac{3}{\alpha}} }
+ C \| \n_0 \|_{L^\infty} \| u_0 \|_{H^1} \\
& \le C.
\ea\ee
Thus, we conclude that there exists a positive constant $C_0$ depending only on $\| \n_0 \|_{L^\infty}$, $\| u_0 \|_{H^1}$, $\| \tilde{x}^\alpha \sqrt{\n_0} u_0 \|_{L^2}$, $\alpha$, $\mu$, $\tilde{\n}$, and $\beta$ such that
\be\la{2cp361}\ba
\| [\Lambda_2(\n_0)]_+ \|_{L^\infty} + \| \psi(0) \|_{L^\infty} \le C_0,
\ea\ee
with $[g]_+\triangleq\max\{g,0\}$. Denoting $f \triangleq \max\{ \Lambda_2(\n) - \psi - C_0,0 \}$, we multiply (\ref{2cp32}) by $4 \ga \n f^{4\ga-1}$,
integrate over $\rr$,
and apply (\ref{2cp34}) and (\ref{2cp36}) to derive
\be\la{2cp37}\ba
& \frac{d}{dt} \int \n f^{4\ga} dx + 4 \ga \int \n f^{4\ga-1} |P - P(\tilde{\n})| dx \\
& \le C \int_{ \{ \n < 2\tilde{\n} \} } \n f^{4\ga-1} |P - P(\tilde{\n})| dx
+ C \int \n f^{4\ga-1} (|F_1| + |F_2|) dx \\
& \le C \int_{ \{ \n < 2\tilde{\n} \} } \n^{1-1/(4\ga)} f^{4\ga-1} |\n - \tilde{\n}|^{1/(2\ga)} dx \\
& \quad + C \int (1 + |\n-\tilde{\n}|^{1/(4\ga)} ) \n^{1-1/(4\ga)} f^{4\ga-1} (|F_1| + |F_2|) dx \\
& \le C \| \n f^{4\ga} \|^{1-1/(4\ga)}_{L^1} \| \n - \tilde{\n} \|^{1/(2\ga)}_{L^2(\n<2\tilde{\n})}
+ C \| \n f^{4\ga} \|^{1-1/(4\ga)}_{L^1} \left( \| F_1 \|_{L^{4\ga}} + \| F_2 \|_{L^{4\ga}} \right) \\
& \quad + C \| \n f^{4\ga} \|^{1-1/(4\ga)}_{L^1} \| \n -\tilde{\n} \|^{1/(4\ga)}_{L^{4 \beta \ga+1}}
\left( \| F_1 \|_{L^{(4\beta\ga+1)/\beta}}
+ \| F_2 \|_{L^{(4\beta\ga+1)/\beta}} \right) \\
& \le C \left( \| \n f^{4\ga} \|_{L^1} + \| \n -\tilde{\n} \|^{4\ga+1}_{L^{4 \beta \ga+1}} + 1 \right) \left( 1 + \| \na u \|^2_{L^2} \right).
\ea\ee

On the other hand, multiplying $(\ref{nlckv})_1$, $(\ref{nlckv})_2$, and $(\ref{nlckv})_3$ by
$\tilde{x}^{2\alpha} K'(\n)$, $\tilde{x}^{2\alpha} u$,
$-\tilde{x}^{2\alpha}(\Delta d + |\na d|^2 d)$, respectively,
summing them up, and integrating by parts, we obtain
\be\la{2cp38}\ba
& \frac{d}{dt} H(t) + \int \left( \mu |\na u|^2 + (\mu+\lam)(\div u)^2
+ |\Delta d + |\na d|^2 d |^2 \right) \tilde{x}^{2\alpha} dx \\
& \le C \int \n |u|^3 |\na \tilde{x}^{2\alpha}| dx
+ C \int |\na u| |u| |\na \tilde{x}^{2\alpha}| dx
+ C \int (\mu + \lam) |\div u| |u| |\na \tilde{x}^{2\alpha}| dx \\
& \quad + C \int \left( |u| |\na d| + |\na^2 d| + |\na d|^2 \right) |\na d| |\na \tilde{x}^{2\alpha}| dx \\
& \quad + C \int |K(\n) + P - P(\tilde{\n})| |u| |\na \tilde{x}^{2\alpha}| dx \triangleq \sum_{i=1}^{5} J_i,
\ea\ee
where
\be\la{2cp39}\ba
H(t) \triangleq \frac{1}{2} \int \left( \n |u|^2 + |\na d|^2 + 2 K(\n) \right) \tilde{x}^{2\alpha} dx.
\ea\ee
By virtue of (\ref{2cp01}), (\ref{2cp001}), (\ref{2cp02}) and Young's inequality, we estimate each $J_i$ as follows:
\be\la{2cp310}\ba
|J_1| & \le C \int \n^{1/2} |u| \tilde{x}^{\alpha} \left( |u|^2 + |\n -\tilde{\n}|^{1/2} |u|^2 \right) dx \\
& \le C \| \sqrt{\n} u \tilde{x}^{\alpha} \|_{L^2}
\left( \| u \|^2_{L^4} + \|\n -\tilde{\n}\|^{1/2}_{L^{4\beta \ga+1}} \| u \|^2_{L^{(4\beta \ga+1)/(\beta \ga)}} \right) \\
& \le C \left( \| \sqrt{\n} u \tilde{x}^{\alpha} \|^2_{L^2}
+ \|\n -\tilde{\n}\|_{L^{4\beta \ga+1}} + 1 \right) \left( \| \na u \|^2_{L^2} + 1 \right),
\ea\ee
\be\la{2cp311}\ba
J_2 \le \| \na u \tilde{x}^{\alpha} \|_{L^2} \| u \|_{L^2}
\le \frac{\mu}{4} \| \na u \tilde{x}^{\alpha} \|^2_{L^2} + C \left( \| \na u \|^2_{L^2} + 1 \right),
\ea\ee
\be\la{2cp312}\ba
J_3 & \le \frac{1}{4} \int (\mu+\lam)(\div u)^2 \tilde{x}^{2\alpha} dx
+ C \int \left( |\n - \tilde{\n}|^\beta + 1 \right) |u|^2 dx \\
& \le \frac{1}{4} \int (\mu+\lam)(\div u)^2 \tilde{x}^{2\alpha} dx + C \| u \|^2_{L^2} + C \|\n -\tilde{\n}\|^\beta_{L^{4\beta \ga+1}}
\| u \|^2_{L^{ 2(4\beta \ga+1)/(4\beta \ga+1-\beta) }} \\
& \le \frac{1}{4} \int (\mu+\lam)(\div u)^2 \tilde{x}^{2\alpha} dx
+ C \left( \|\n -\tilde{\n}\|^\beta_{L^{4\beta \ga+1}} + 1 \right) \left( \| \na u \|^2_{L^2} + 1 \right),
\ea\ee
\be\la{2cp313}\ba
J_4 & \le C \| \na d \tilde{x}^{\alpha} \|_{L^2}
\left( \| u \|_{L^4} \| \na d \|_{L^4} + \| \na^2 d \|_{L^2} + \| \na d \|^2_{L^4} \right) \\
& \le C \left( \| \na d \tilde{x}^{\alpha} \|^2_{L^2} + 1 \right)
\left( \| \na u \|^2_{L^2} + \| \na^2 d \|^2_{L^2} + 1 \right),
\ea\ee
and
\be\la{2cp314}\ba
J_5 & \le C \int_{ \{\n < 2\tilde{\n}\} } |\n - \tilde{\n}| |u| \tilde{x}^{\alpha} dx
+ C \int_{ \{\n \ge 2\tilde{\n}\} } |\n - \tilde{\n}|^{\ga} |u| \tilde{x}^{\alpha} dx \\
& \le C \int_{ \{\n < 2\tilde{\n}\} } K(\n)^{1/2} |u| \tilde{x}^{\alpha} dx
+ C \int_{ \{\n \ge 2\tilde{\n}\} } K(\n)^{1/2} |\n - \tilde{\n}|^{\ga/2} |u| \tilde{x}^{\alpha} dx \\
& \le C \| K(\n) \tilde{x}^{2\alpha} \|^{1/2}_{L^1}
\left( \| u \|_{L^2} + C \|\n -\tilde{\n}\|^{\ga/2}_{L^{4\beta \ga+1}}
\| u \|_{ L^{2(4\beta \ga+1)/(4\beta \ga +1-\ga)} } \right) \\
& \le C \left( \| K(\n) \tilde{x}^{2\alpha} \|_{L^1} + \|\n -\tilde{\n}\|^{\ga}_{L^{4\beta \ga+1}} + 1 \right)
\left( \| \na u \|^2_{L^2} + 1 \right).
\ea\ee
Putting (\ref{2cp310})--(\ref{2cp314}) into (\ref{2cp38}) yields
\be\la{2cp315}\ba
& \frac{d}{dt} H(t) + \frac{1}{2} \int \left( \mu |\na u|^2 + (\mu+\lam)(\div u)^2
+ |\Delta d + |\na d|^2 d |^2 \right) \tilde{x}^{2\alpha} dx \\
& \le C \left( H(t) + \|\n -\tilde{\n} \|^{4\beta \ga+1}_{L^{4\beta \ga+1}} + 1 \right)
\left( \| \na u \|^2_{L^2} + \| \na^2 d \|^2_{L^2} + 1 \right),
\ea\ee
which along with (\ref{2cp37}) gives
\be\la{2cp316}\ba
& \frac{d}{dt} \left( \| \n f^{4\ga} \|_{L^1} + H(t) \right) + \frac{1}{2} \int \left( \mu |\na u|^2 + (\mu+\lam)(\div u)^2
+ |\Delta d + |\na d|^2 d |^2 \right) \tilde{x}^{2\alpha} dx \\
& \le C \left( \| \n f^{4\ga} \|_{L^1} + H(t) + \|\n -\tilde{\n} \|^{4\beta \ga+1}_{L^{4\beta \ga+1}} + 1 \right)
\left( \| \na u \|^2_{L^2} + \| \na^2 d \|^2_{L^2} + 1 \right).
\ea\ee

Next, we estimate $\|\n -\tilde{\n} \|_{L^{4\beta \ga+1}}$.
In view of the definition of $\psi$, we set
\be\la{2cp317}\ba
\psi = (- \Delta)^{-1} \div( \sqrt{\n} u (\sqrt{\n}-\sqrt{\tilde{\n}}) ) + \sqrt{\tilde{\n}} (- \Delta)^{-1} \div( \sqrt{\n} u )
\triangleq \psi_1 + \psi_2.
\ea\ee
Note that (\ref{2cp13}) implies
\be\la{2cp318}\ba
\int |\n -\tilde{\n}|^{4 \beta \ga+1} dx
& \le C \int_{ \{\n<2\tilde{\n} \} } |\n - \tilde{\n}|^2 dx
+ C \int_{ \{\n \ge 2\tilde{\n} \} } |\n - \tilde{\n}|^{4 \beta \ga+1} dx \\
& \le C + C \int_{ \{\n \ge 2\tilde{\n} \} } |\n - \tilde{\n}|
\left( f^{4\ga} + |\psi|^{4\ga} + C_0^{4\ga} \right) dx \\
& \le C + C \| \n f^{4\ga} \|_{L^1}
+ C \int_{ \{\n \ge 2\tilde{\n} \} } |\n - \tilde{\n}| \left( |\psi_1|^{4\ga} + |\psi_2|^{4\ga}\right) dx.
\ea\ee
On the one hand, it follows from (\ref{2cp01}), (\ref{2cp13}), and Sobolev embedding that
\be\la{2cp319}\ba
\int_{ \{\n \ge 2\tilde{\n} \} } |\n - \tilde{\n}| |\psi_1|^{4\ga} dx
& \le C \| \n -\tilde{\n} \|_{L^{4 \beta \ga+1}} \| \psi_1 \|^{4\ga}_{L^{(4\beta \ga+1)/\beta}} \\
& \le C \| \n -\tilde{\n} \|_{L^{4 \beta \ga+1}} \| \sqrt{\n} u \|^{4\ga}_{L^2}
\| \sqrt{\n} - \sqrt{\tilde{\n}} \|^{4\ga}_{L^{(4 \beta \ga+1)/\beta}} \\
& \le C \| \n -\tilde{\n} \|_{L^{4 \beta \ga+1}}
\| \n - \tilde{\n} \|^{2\ga}_{L^{(4 \beta \ga+1)/(2\beta)}} \\
& \le C \| \n -\tilde{\n} \|_{L^{4 \beta \ga+1}}
\left( \| \n - \tilde{\n} \|^{2\ga}_{L^2(\n<2\tilde{\n})}
+ \| \n - \tilde{\n} \|^{2\ga}_{L^\ga (\n \ge 2\tilde{\n})}
+ \| \n - \tilde{\n} \|^{2\ga}_{L^{4 \beta \ga+1}} \right) \\
& \le C \| \n -\tilde{\n} \|_{L^{4 \beta \ga+1}}
\left( 1 + \| \n - \tilde{\n} \|^{2\ga}_{L^{4 \beta \ga+1}} \right) \\
& \le \ep \| \n - \tilde{\n} \|^{4 \beta \ga+1}_{L^{4 \beta \ga+1}} + C(\ep).
\ea\ee
On the other hand, for $p = \frac{2\beta}{4\beta \ga+1} \in (0,1)$ and $s = \frac{1}{\alpha \ga p} \in (2,\infty)$,
we conclude from (\ref{2cp01}), (\ref{2cp13}) and Young's inequality that
\be\la{2cp320}\ba
& \int_{ \{\n \ge 2\tilde{\n} \} } |\n - \tilde{\n}| |\psi_2|^{4\ga} dx \\
& \le C \| \n -\tilde{\n} \|_{L^{\frac{s}{s-1}} (\n \ge 2\tilde{\n}) } \| \psi_2 \|^{4\ga}_{L^{4\ga s}} \\
& \le C \left( \| \n -\tilde{\n} \|_{L^{4 \beta \ga+1}}
+ \| \n -\tilde{\n} \|_{L^1 (\n \ge 2\tilde{\n})} \right)
\left( \| \sqrt{\n} u \|^{1-p}_{L^2} \| \sqrt{\n} u \tilde{x}^{\alpha} \|^{p}_{L^2}
\| \tilde{x}^{-\alpha p} \|_{L^{4\ga s}} \right)^{4\ga} \\
& \le C \left( \| \n -\tilde{\n} \|_{L^{4 \beta \ga+1}} + \| \n -\tilde{\n} \|_{L^\ga (\n \ge 2\tilde{\n})}^\ga \right)
\| \sqrt{\n} u \tilde{x}^{\alpha} \|^{4\ga p}_{L^2} \\
& \le C \left( \| \n -\tilde{\n} \|_{L^{4 \beta \ga+1}} + 1 \right)
\| \sqrt{\n} u \tilde{x}^{\alpha} \|^{4\ga p}_{L^2} \\
& \le \ep \| \n - \tilde{\n} \|^{4 \beta \ga+1}_{L^{4 \beta \ga+1}}
+ C (\ep) \| \sqrt{\n} u \tilde{x}^{\alpha} \|^2_{L^2} + C (\ep).
\ea\ee
Substituting (\ref{2cp319}) and (\ref{2cp320}) into (\ref{2cp318}) and choosing $\ep$ sufficiently small, we obtain
\be\la{2cp321}\ba
\int |\n -\tilde{\n}|^{4 \beta \ga+1} dx
& \le C \| \n f^{4\ga} \|_{L^1} + C \| \sqrt{\n} u \tilde{x}^{\alpha} \|^2_{L^2} + C.
\ea\ee
Combining this with (\ref{2cp316}) leads to
\be\la{2cp322}\ba
& \frac{d}{dt} \left( \| \n f^{4\ga} \|_{L^1} + H(t) \right) + \frac{1}{2} \int \left( \mu |\na u|^2 + (\mu+\lam)(\div u)^2
+ |\Delta d + |\na d|^2 d |^2 \right) \tilde{x}^{2\alpha} dx \\
& \le C \left( \| \n f^{4\ga} \|_{L^1} + H(t) + 1 \right)
\left( \| \na u \|^2_{L^2} + \| \na^2 d \|^2_{L^2} + 1 \right),
\ea\ee
which together with (\ref{2cp01}), (\ref{2cp361}), and Gr\"onwall's inequality implies
\be\la{2cp323}\ba
\sup_{0\le t\le T} \int \left( \n f^{4\ga} + \left( \n |u|^2 + |\na d|^2 + 2 K(\n) \right) \tilde{x}^{2\alpha} \right) dx \le C.
\ea\ee
Finally, (\ref{2cp13}) and (\ref{2cp321}) ensure
\be\la{2cp324}\ba
\int \left( |\n - \tilde{\n}|^2 + |\n - \tilde{\n}|^{4\beta\ga+1 }\right) dx
& \le C \int_{ \{\n < 2\tilde{\n}\} } |\n - \tilde{\n}|^2 dx
+ C \int |\n - \tilde{\n}|^{4\beta\ga+1 } dx \\
& \le C \| \n f^{4\ga} \|_{L^1} + C \| \sqrt{\n} u \tilde{x}^{\alpha} \|^2_{L^2} + C.
\ea\ee
This, combined with (\ref{2cp323}), yields (\ref{2cp003}) and finishes the proof of Lemma \ref{cp2l3}.
\end{proof}

\begin{lemma}\la{cp2l4}
For any $p>4$, there exists a positive constant $C$ depending only on $p$, $T$, $\alpha$, $\mu$, $\ga$, $\beta$, $\tilde{\n}$, and $E_2$ such that
\be\la{2cp04}\ba
\|\n u\|_{L^p} \le  C R_T^{1+\beta/4}  (1+ \|\na u\|_{L^2})^{ 1-2/p},
\ea\ee
with $R_T$ as in \eqref{mdsj}.
\end{lemma}
\begin{proof}
Following the procedure similar to Lemma \ref{cp1l4}, we multiply $(\ref{nlckv})_2$ by $|u|^\nu u$ with $\nu$ given in (\ref{1cp41}),
and integrate over $\rr$ to derive
\be\la{2cp41}\ba
& \frac{1}{(2+\nu)} \frac{d}{dt} \int \n |u|^{2+\nu} dx
+ \int |u|^\nu \left(\mu |\na u|^2 + (\mu+\lam) (\div u)^2 \right) dx \\
& \le \frac{1}{2} \int (\mu+\lam) |u|^\nu (\div u)^2 dx
+ \frac{\mu}{2} \int |u|^{\nu} |\na u|^2 dx \\
& \quad + C \int |P-P(\tilde{\n})| |u|^\nu |\na u| dx + C \int |u|^\nu |\na d|^4 dx.
\ea\ee
From (\ref{2cp003}), we deduce that for any $s \in [2,4]$,
\be\la{2cp42}\ba
\int |P - P(\tilde{\n})|^s dx \le C \int_{ \{ \n<2\tilde{\n} \} } |\n -\tilde{\n}|^2 dx
+ C \int_{ \{ \n \ge 2\tilde{\n} \} } |\n -\tilde{\n}|^{4\beta \ga+1} dx \le C.
\ea\ee
Combining this with (\ref{2cp001}), (\ref{2cp02}), (\ref{2cp41}), and H\"older's inequality yields
\be\la{2cp43}\ba
& \frac{d}{dt} \int \n |u|^{2+\nu} dx
+ \int |u|^\nu \left(\mu |\na u|^2 + (\mu+\lam) (\div u)^2 \right) dx \\
& \le C \int |P-P(\tilde{\n})| \left( 1 + |u| \right) |\na u| dx
+ C \int \left( 1 + |u| \right) |\na d|^4 dx \\
& \le C \| P - P(\tilde{\n}) \|_{L^2} \| \na u \|_{L^2}
+ C \| P - P(\tilde{\n}) \|_{L^4} \| u \|_{L^4} \| \na u \|_{L^2} \\
& \quad + C \| \na d \|^4_{L^4} + C \| u \|_{L^2} \| \na d \|^4_{L^8} \\
& \le C (1 + \| \na u \|^2_{L^2}).
\ea\ee
Integrating the above inequality and applying (\ref{2cp01}), we obtain
\be\la{2cp44}\ba
\sup_{0 \le t \le T} \int \n |u|^{2+\nu} dx \le C.
\ea\ee
Finally, it follows from (\ref{gn11}), (\ref{2cp001}), (\ref{2cp44}), and H\"older's inequality that for $p>4$ and $r=(p-2)(2+\nu)/\nu$,
\be\nonumber\ba
\| \n u \|_{L^p} & \le \|\n u\|_{L^{2+\nu}}^{2/p} \|\n u\|_{L^r}^{1-2/p} \\
& \le C R_T^{ (1+\nu)/p} \|\n^{1/(2+\nu)} u\|_{L^{2+\nu}}^{2/p}
\left( r^{1/2} R_T \| u \|^{2/r}_{L^2} \|\na  u\|^{1-2/r}_{L^2} \right)^{ 1-2/p} \\
& \le C(p) R_T^{ (1+\nu)/p} \left( R_T^{1+\beta/4} (1+ \|\na  u\|_{L^2}) \right)^{ 1-2/p} \\
& \le C(p) R_T^{1+\beta/4} (1+ \|\na  u\|_{L^2})^{ 1-2/p},
\ea\ee
which implies (\ref{2cp04}) and completes the proof of Lemma \ref{cp2l4}.
\end{proof}

\begin{lemma}\la{cp2l5}
There exists a positive constant $C$ depending only on $T$, $\alpha$, $\mu$, $\ga$, $\beta$, $\tilde{\n}$, and $E_2$ such that
\be\la{2cp05}\ba
&\sup\limits_{0\le t\le T} \log A_1^2(t) + \int_0^T\frac{A_2^2(t)}{A_1^2(t)} dt \le C R_T^{4/3},
\ea\ee
where $A_1$ and $A_2$ are defined in \eqref{a1} and \eqref{a2}, respectively.
\end{lemma}
\begin{proof}
First, multiplying (\ref{2cp31}) by $2 \dot{u}$ and integrating by parts over $\rr$ yields
\be\la{2cp51}\ba
& \frac{d}{dt} \int \left(\mu \o^2 + \frac{G^2}{2\mu + \lam}\right)dx
+ 2 \int \n |\dot{u}|^2 dx \\
& = - \mu \int \o^2 \div u dx + 4 \int G \nabla u^1 \cdot\nabla^{\perp}u^2 dx
-2 \int G (\div u)^2 dx \\
& \quad - \int \frac{ (\beta-1)\lam - 2\mu }{(2\mu + \lam)^2} G^2 \div u dx
- 2 \beta \int \frac{ \lam ( P-P(\tilde{\n}) ) }{ (2\mu +\lam)^2 } G \div u dx \\
& \quad + 2 \ga \int \frac{P}{2\mu +\lam} G \div u dx
-2 \int \dot{u} \cdot \na d \cdot \Delta d dx
\triangleq \sum_{i=1}^7 L_i.
\ea\ee

Next, we estimate each $L_i$.

From (\ref{2cp31}), we deduce that $G$ and $\o$ satisfy
\be\la{2cp52}\ba
\Delta G = \div(\n \dot{u} + \na d \cdot \Delta d), \quad \mu \Delta \o = \na^\bot \cdot (\n \dot{u} + \na d \cdot \Delta d).
\ea\ee
The standard elliptic estimates along with (\ref{2cp02}) show that
\be\la{2cp53}\ba
\| \na G \|_{L^2} + \| \na \o \|_{L^2}
& \le C \left( \| \n \dot{u} \|_{L^2} + \| \na d \cdot \Delta d \|_{L^2} \right) \\
& \le C R_T^{1/2} \| \sqrt{\n} \dot{u} \|_{L^2} + C \| \na d \|_{L^4} \| \Delta d \|_{L^4} \\
& \le C R_T^{1/2} \| \sqrt{\n} \dot{u} \|_{L^2} + C \| \Delta d \|_{L^2} + C \| \na \Delta d \|_{L^2} \\
& \le C A_1 + C R_T^{1/2} A_2.
\ea\ee
Combining this with (\ref{gn11}) and Young's inequality gives
\be\la{2cp54}\ba
L_1 \le C \| \o \|^2_{L^4} \| \div u \|_{L^2}
& \le C \| \o \|_{L^2} \| \na \o \|_{L^2} \| \na u \|_{L^2} \\
& \le C A_1 \left( A_1 + R_T^{1/2} A_2 \right) \| \na u \|_{L^2} \\
& \le \frac{1}{8} A^2_2 + C R_T A_1^4,
\ea\ee
where we have used the following estimate:
\be\la{2cp55}\ba
\| \na u \|_{L^2} \le C \| \div u \|_{L^2} + C \| \o \|_{L^2}
\le C A_1 + C \| P - P(\tilde{\n}) \|_{L^2} \le C A_1,
\ea\ee
due to (\ref{gw}) and (\ref{2cp42}).

To estimate $L_2$, we apply the method from \cite{P}.
Note that
\be\la{2cp56}\ba
\curl \na u^1=0, \quad \div \na^\bot u^2=0.
\ea\ee
Then, from \cite[Theorem II.1]{CLMS}, we have
\be\la{2cp57}\ba
\| \na u^1 \cdot \na^\bot u^2 \|_{\mathcal{H}^1} \le C \| \na u \|^2_{L^2}.
\ea\ee
Since $\mathcal{BMO}(\rr)$ is the dual space of $\mathcal{H}^1(\rr)$ (see \cite{FC}), this combined with (\ref{2cp55}) implies
\be\la{2cp58}\ba
L_2 & \le C \| G \|_{\mathcal{BMO}} \| \na u^1 \cdot \na^\bot u^2 \|_{\mathcal{H}^1} \\
& \le C \| \na G \|_{L^2} \| \na u \|^2_{L^2} \\
& \le C \left( A_1 + R_T^{1/2} A_2 \right) A_1^2 \\
& \le \frac{1}{8} A^2_2 + C R_T A_1^4.
\ea\ee
Moreover, it follows from (\ref{gw}), (\ref{2cp42}) and (\ref{2cp55}) that for any $\de \in (0,1)$,
\be\la{2cp59}\ba
\sum_{i=3}^{6} |L_i| & \le C \int \frac{G^2 |\div u|}{2\mu+\lam} dx
+ C \int \frac{|G|}{2\mu+\lam} \left( |P-P(\tilde{\n})| + P(\tilde{\n}) \right) |\div u| dx \\
& \le C \| \na u \|_{L^2} \left\| \frac{G^2}{2\mu+\lam} \right\|_{L^2}
+ C \| \na u \|_{L^2} \left\| \frac{G^2}{2\mu+\lam} \right\|^{1/2}_{L^2} \| P-P(\tilde{\n}) \|_{L^4} \\
& \quad + C \| \na u \|_{L^2} \left\| \frac{G}{2\mu+\lam} \right\|_{L^2} \\
& \le C \| \na u \|_{L^2} \left\| \frac{G^2}{2\mu+\lam} \right\|_{L^2}
+ C ( 1 + \| \na u \|^2_{L^2} ) \\
& \le C R_T^{ (1+\de\beta)/2 } \| \na u \|_{L^2} A_1 \left( A_2 + A_1 \right)
+ C A_1^2 \\
& \le \frac{1}{8} A^2_2 + C R_T^{ 1+\de\beta } A_1^4,
\ea\ee
where we have used the following estimate:
\be\la{2cp510}\ba
\left\|\frac{G^2}{2\mu+\lam} \right\|_{L^2}
& \le C \left\| \frac{G}{\sqrt{2\mu+\lam}} \right\|_{L^2}^{1-\de}
\| G \|_{ L^{2(1+\de)/\de}}^{1+\de} \\
& \le C (\de) A_1^{ 1-\de } \| G \|_{L^2}^{\de} \| \na G \|_{L^2} \\
& \le C(\de) A_1^{ 1-\de } R_T^{ (\de\beta)/2 } A_1^\de \left( A_1 + R_T^{1/2} A_2 \right) \\
& \le C(\de) R_T^{ (1+\de\beta)/2 } A_1 \left( A_2 + A_1 \right),
\ea\ee
due to (\ref{gn11}) and (\ref{gw}).

For $L_7$, integrating by parts over $\rr$ and using
(\ref{2cp001}), (\ref{2cp02}) and Young's inequality, we derive
\be\la{2cp511}\ba
L_7 & = - 2 \int u_t \cdot \na d \cdot \Delta d dx
- 2 \int u \cdot \na u \cdot \na d \cdot \Delta d dx \\
& \le \frac{d}{dt} \int \left( 2 (\na d \odot \na d) \cdot \na u - |\na d|^2 \div u \right) dx
+ C \int |\na d| |\na u| \left( |\na d_t|  + |u| |\na^2 d| \right) dx \\
& \le \frac{d}{dt} \int \left( 2 (\na d \odot \na d) \cdot \na u - |\na d|^2 \div u \right) dx
+ C \| \nabla d_t \|_{L^2} \| \na d \|_{L^4} \| \na u \|_{L^4} \\
& \quad + C \| u \|_{L^8} \| \na u \|_{L^4} \| \na d \|_{L^8} \| \na^2 d \|_{L^2} \\
& \le \frac{d}{dt} \int \left( 2 (\na d \odot \na d) \cdot \na u - |\na d|^2 \div u \right) dx
+ \frac{1}{8} A^2_2
+ C \| \na u \|^2_{L^4} + C A_1^4 \\
& \le \frac{d}{dt} \int \left( 2 (\na d \odot \na d) \cdot \na u - |\na d|^2 \div u \right) dx
+ \frac{1}{4} A^2_2 + C R_T^{ 1+\de\beta } A_1^4,
\ea\ee
where in the last inequality we have used the following estimate:
\be\la{2cp512}\ba
\| \na u \|^2_{L^4} & \le C \| \div u \|^2_{L^4} + C \| \o \|^2_{L^4} \\
& \le C \left\| \frac{G}{2\mu+\lam} \right\|^2_{L^4} + C \| P-P(\tilde{\n}) \|^2_{L^4} + C \| \o \|_{L^2} \| \na \o \|_{L^2} \\
& \le C(\de) R_T^{ (1+\de\beta)/2 } A_1 \left( A_2 + A_1 \right),
\ea\ee
owing to (\ref{gn11}), (\ref{2cp53}) and (\ref{2cp510}).

Furthermore, from (\ref{gn11}), (\ref{1cp21}), (\ref{2cp001}), (\ref{2cp02}), (\ref{2cp512}) and Young's inequality, we conclude that
\be\la{2cp513}\ba
& \frac{d}{dt} \| \Delta d \|^2_{L^2} + \| \na d_t \|^2_{L^2} + \| \na \Delta d \|^2_{L^2} \\
& \le C \int |\na u|^2 |\na d|^2 + |u|^2 |\na^2 d|^2
+ |\na^2 d|^2 |\na d|^2 + |\na d|^6 dx \\
& \le C \| \na u \|^2_{L^4} \| \na d \|^2_{L^4}
+ C \| u \|^2_{L^8} \| \na^2 d \|^2_{L^\frac{8}{3}}
+ C \| \na^2 d \|^2_{L^4} \| \na d \|^2_{L^4} + C \| \na d \|^6_{L^6} \\
& \le C \| \na u \|^2_{L^4}
+ C ( 1 + \| \na u \|^2_{L^2} ) \| \na d\|_{L^4} \| \na^3 d\|_{L^2}
+ C \| \na^2 d \|_{L^2} \| \na^3 d \|_{L^2} + C \\
& \le \frac{1}{8} A^2_2 + C R_T^{ 1+\de\beta } A_1^4.
\ea\ee
Substituting (\ref{2cp54}), (\ref{2cp58}), (\ref{2cp59}) and (\ref{2cp511}) into (\ref{2cp51}) and combining with (\ref{2cp513}), we obtain
\be\la{2cp514}\ba
\frac{d}{dt} \hat{A}_1 + \frac{1}{4}A^2_2 \le C R_T^{ 1+\de\beta } A_1^4,
\ea\ee
where
\be\la{2cp515}\ba
\hat{A}_1(t) \triangleq A_1^2(t) - \int \left( 2 (\na d \odot \na d) \cdot \na u - |\na d|^2 \div u \right) dx.
\ea\ee
By virtue of (\ref{2cp02}), (\ref{2cp55}), (\ref{2cp515}) and Young's inequality, we have
\be\la{2cp516}\ba
\left| \int \left(2 (\na d \odot \na d) \cdot \na u - |\na d|^2 \div u \right) dx \right|
\le C \| \na d \|^2_{L^4} \| \na u \|_{L^2}
\le \frac{1}{2} A_1^2 + \hat{C}_2,
\ea\ee
which together with (\ref{2cp515}) yields
\be\la{2cp517}\ba
\frac{1}{2} A_1^2(t) \le \hat{A}_1(t) + \hat{C}_2 \le 2 \left( A_1^2(t) + \hat{C}_2 \right).
\ea\ee

Dividing (\ref{2cp514}) by $\hat{A}_1(t) + \hat{C}_2$,
choosing $\de=1/(3\beta)$, and integrating over $(0,T)$,
we arrive at (\ref{2cp05}) after using (\ref{2cp01}), (\ref{2cp42}) and (\ref{2cp517}). This completes the proof of Lemma \ref{cp2l5}.
\end{proof}

\begin{lemma}\la{cp2l6}
There exists a positive constant $C$ depending only on $T$, $\alpha$, $\mu$, $\ga$, $\beta$, $\tilde{\n}$, and $E_2$ such that
\be\ba\la{2cp06}
& \sup_{0\leq t\leq T} \left( \|\n\|_{L^\infty} + \| u \|_{H^1} + \| \na d \|_{H^1} \right) \\
& + \int_0^T \left( \| \na u \|^2_{L^2} + \| \sqrt{\n} \dot{u} \|^2_{L^2} + \| \na^2 d \|^2_{H^1} + \| \na d_t \|^2_{L^2} \right) dt
\le C.
\ea\ee
\end{lemma}
\begin{proof}
First, for $F_1$ and $F_2$ defined in (\ref{1cp34}),
we deduce from the Gagliardo-Nirenberg inequality, (\ref{jhz1}), (\ref{jhz2}), (\ref{2cp02}), (\ref{2cp36}), (\ref{2cp04}),
and (\ref{2cp512}) that for any $p \in (4,\infty)$,
\be\la{2cp61}\ba
\| F_1 \|_{L^\infty}
& \le C(p) \| F_1 \|_{L^p}^{1-4/p} \|\na F_1\|_{L^{4p/(p+4)}}^{4/p} \\
& \le C(p) \left( \|\na u\|_{L^2} \|\n  u\|_{L^p} \right)^{1-4/p}
\left(\|\na u\|_{L^4} \|\n  u\|_{L^p} \right)^{4/p} \\
& \le C(p) \|\na u\|_{L^2}^{1-4/p} \|\na u\|_{L^4}^{4/p} \|\n  u\|_{L^p} \\
& \le C(p) R_T^{1 +\beta/4 + (1+\beta)/p} A_1^{2-2/p}
\left(1+\frac{A_2^2}{A_1^2} \right)^{1/p} \\
& \le C(p) R_T^{ 1+\beta/4 + 2(1+\beta)/(p-1) } A_1^2
+ C \frac{A_2^2}{A_1^2},
\ea\ee
and
\be\la{2cp62}\ba
\| F_2 \|_{L^\infty}
& \le C \| F_2 \|^{1/2}_{L^4} \| \na F_2 \|^{1/2}_{L^4}
\le C \| \na d \cdot \Delta d \|^{1/2}_{L^4} \\
& \le C \| \na d \|^{1/2}_{L^8} \| \Delta d \|^{1/2}_{L^8}
\le C \left( \| \Delta d \|_{L^2} + \| \na \Delta d \|_{L^2} \right)^{1/2} \\
& \le C A_1 + C A_1^{1/2} \left( \frac{A^2_2}{A_1^2} \right)^{1/4}
\le C A_1^2 + C \frac{A^2_2}{A_1^2}.
\ea\ee
On the other hand, Sobolev embedding and (\ref{2cp03}) ensure
\be\la{2cp63}\ba
\|\psi\|_{L^{4/\alpha}} & \le C \| \na \psi \|_{L^{4/(2+\alpha)}}
\le C \| \n u \|_{L^{4/(2+\alpha)}} \\
& \le C R_T^{1/2} \| \sqrt{\n} u \tilde{x}^\alpha \|_{L^2} \| \tilde{x}^{-\alpha} \|_{L^{4/\alpha}} \le C R_T^{1/2},
\ea\ee
which combined with (\ref{bwi}), (\ref{2cp01}) and (\ref{2cp001}) yields
\be\la{2cp64}\ba
\|\psi\|_{L^\infty}
& \le C \left( \|\psi\|_{L^{4/\alpha}} + \| \na \psi \|_{L^2} \right)
\log^{1/2} ( e + \|\psi \|_{W^{1,4/\alpha}} ) + C \\
& \le C \left( R_T^{1/2} + \| \n u\|_{L^2} \right) \log^{1/2}(e + R_T^{1/2} + \|\n u \|_{L^{4/\alpha}}) + C \\
& \le C R_T^{1/2} \log^{1/2} \left( R_T (e+\|\na u \|_{L^2}) \right) + C \\
& \le C R_T^{1/2} \log^{1/2}(e+ A_1^2) + C R_T \\
& \le C R_T^{4/3}.
\ea\ee
Integrating (\ref{2cp32}) over $(0,T)$ and using (\ref{2cp61}), (\ref{2cp62}) and (\ref{2cp64}),
after choosing $p$ sufficiently large, we conclude from $\beta>4/3$ that
\be\la{2cp65}\ba
\sup_{0 \le t \le T} \| \n \|_{L^\infty} \le C,
\ea\ee
which together with (\ref{2cp01}), (\ref{2cp001}) and (\ref{2cp05}) gives (\ref{2cp06}) and completes the proof of Lemma \ref{cp2l6}.
\end{proof}

Finally, the preceding estimates, combined with methods similar to those in Lemmas \ref{cp1l7} and \ref{cp1l9}, lead to the following results.
The detailed proofs are omitted.

\begin{lemma}\la{cp2l7}
There exists a positive constant $C$ depending only on $T$, $\alpha$, $\mu$, $\ga$, $\beta$, $\tilde{\n}$, $q$, $\| \na \n_0 \|_{L^2 \cap L^q }$, and $E_2$ such that
\be\ba\la{2cp07}
\sup_{0\le t\le T}
t \left( \| \sqrt{\n} \dot{u} \|^2_{L^2} + \| \na d_t \|^2_{L^2} + \| \na^3 d \|^2_{L^2} \right)
+ \int_0^{T} t \left( \| \na \dot{u} \|^2_{L^2} + \| \na^2 d_t \|^2_{L^2} \right) dt \le C,
\ea\ee
and
\be\la{2cp08}\ba
&\sup_{0\le t\le T} \left( \| \n-\tilde{\n} \|_{H^1 \cap W^{1,q}} + t \| \na^2 u \|^2_{L^2} \right) \\
& + \int_0^T \left( \| \na^2 u \|^2_{L^2} + \|\nabla^2 u\|^{(q+1)/q}_{L^q}
+ t \|\nabla^2 u\|_{L^q}^2 + t \| \na^4 d \|^2_{L^2} \right) dt \le C.
\ea\ee
\end{lemma}

\section{A Priori Estimates for the Half-space Problem}

In this section, we consider the half-space $\OM = \rr_+$ and derive the upper bound for $\n$ as well as higher-order estimates, which allow us to extend the local solution to a global one.

\subsection{The Case of Vacuum Far-field Density}
In this subsection, we assume that the initial data $(\n_0,u_0,d_0)$ satisfy (\ref{bkjsol1}) and $\n_0 >0$. Notice that $\int_{\mathbb{R}_{+}^2}\n_0 dx=\tilde{M}_1>0$; hence there exists a positive constant $\tilde{N}_{2}$ such that
\be\la{rho002}\ba
 \quad \int_{B^+_{\tilde{N}_2}} \n_0 dx \ge \frac{1}{2} \int_{\rr_+} \n_0 dx\ge \frac{1}{2} \tilde{M}_1.
\ea\ee

Moreover, let $(\n,u,d)$ be the strong solution to (\ref{nlckv})--(\ref{i30}), (\ref{bkjbjtj1}), (\ref{bkjbjtj2}) on $\rr_+ \times (0,T]$ with $\tilde{\n}=0$.
Here and throughout this subsection, the quantity $E_1$ defined in (\ref{e1}) depends only on the initial data.

First, we establish the standard energy estimate.

\begin{lemma}\la{bkj1l1}
There exists a positive constant $C$ depending only on
$T$, $a$, $\ga$, $\mu$, $\tilde{M}_1$, $\| \sqrt{\n_0} u_0 \|_{L^2}$, $\| {\bar{x}}^a \rho_0 \|_{L^1}$,
$\| \n_0 \|_{L^\ga}$, and $\| \na d_0 \|_{L^2}$ such that
\be\la{1bkj01}\ba
& \sup\limits_{0\le t\le T}\int\left(\n|u|^2+\n^\ga+\n\bar x^a
+ |\na d|^2 \right) dx \\
& + \int_0^T \int\left( \mu |\na u|^2+ \lambda(\n) (\div u)^2 + | \na^2 d |^2 \right) dxdt
\le C,
\ea\ee
and for any $\ep>0$ and $0<\eta\le1$,
\be\la{1bkj001}\ba
\| \bar{x}^{-\eta} u \|_{ L^{(2+\ep)/\eta}(\rr_+) } \le C(\ep,\eta) \left( 1 + \| \na u \|_{L^2(\rr_+)} \right).
\ea\ee
\end{lemma}
\begin{proof}
First, based on Lemma \ref{tygj} and (\ref{jzb1}), and adapting the arguments in Lemma \ref{cp1l1}, we obtain
\be\la{1bkj11}\ba
\sup\limits_{0\le t\le T} \int\left( \n |u|^2 + \n^\ga + |\na d|^2 \right) dx
+ \int_0^T \int\left( \mu |\na u|^2+ \lambda(\n) (\div u)^2 + |\na^2 d|^2 \right) dxdt
\le C.
\ea\ee

Next, multiplying $(\ref{nlckv})_1$ by ${\bar{x}}^a$ and integrating by parts over $\rr_+$, we derive after using (\ref{1bkj11}), (\ref{bkjbjtj1}), and H\"older's inequality that
\be\ba\nonumber
\frac{d}{dt} \int_{\rr_+} \n {\bar{x}}^a dx
& \le C \int_{\rr_+} \n |u| {\bar{x}}^{a-1} \log^{1+\eta_0}(e+|x|^2) dx \\
& \le C \left( \int_{\rr_+} \n {\bar{x}}^{2a-2} \log^{2(1+\eta_0)}(e+|x|^2) dx \right)^{1/2}
\left( \int_{\rr_+} \n |u|^2 dx \right)^{1/2} \\
& \le C \left( \int_{\rr_+} \n {\bar{x}}^{a} dx \right)^{1/2},
\ea\ee
which together with Gr\"onwall's inequality yields
\be\la{1bkj12}\ba
\sup_{0 \le t \le T} \int \n {\bar{x}}^{a} dx \le C.
\ea\ee
This, combined with (\ref{1bkj11}), implies (\ref{1bkj01}).

For any $N>1$, choose a smooth cutoff function $ \varphi_N $ on $\rr$ satisfying:
\be\la{1bkj13}\ba
0 \le \varphi_N \le 1, \quad \varphi_N=
\begin{cases}
1,\quad &\mathrm{ if \  } |x| \le N,\\
0,\quad &\mathrm{ if \  } |x| \ge 2N,
\end{cases}
\quad |\na \varphi_N | \le 2 N^{-1}.
\ea\ee
Denote $\varphi^+_N \triangleq \varphi_{N}|_{\rr_+}$.
Multiplying $(\ref{nlckv})_1$ by $\varphi^+_N$, integrating by parts over $\rr_+$, and applying (\ref{bkjbjtj1}), (\ref{1bkj11}), and (\ref{1bkj13}), we derive
\be\ba\nonumber
\frac{d}{dt} \int_{\rr_+} \n \varphi^+_N dx =\int_{\rr_+} \n u \cdot \na \varphi^+_N dx 
\ge -2 N^{-1}  \tilde{M}_1^{\frac{1}{2}} 
\left( \int_{\rr_+} \n |u|^2 dx \right)^{\frac{1}{2}} \ge -2 \hat{C} N^{-1},
\ea\ee
where the positive constant $\hat{C}$ depends only on $\tilde{M}_1$, $\|\sqrt{\n_0}u_0\|_{L^2}$, $\| \n_0 \|_{L^\infty}$, and $\|\nabla d_0\|_{L^2}$.

Choosing $N_2\triangleq 2(1+\tilde{N}_2+8 \tilde{M}_1^{-1} \hat{C}T)$, where $\tilde{N}_2$ is given by \eqref{rho002}, we obtain
\be\la{1bkj14}\ba
\inf_{0 \le t \le T} \int_{B^+_{N_2}} \n dx
& \ge \inf_{0 \le t \le T} \int_{\rr_+} \n \varphi^+_{{N_2}/2} dx \\
& \ge \int_{\rr_+} \n_0 \varphi^+_{{N_2}/2} dx - 4 \hat{C} N_2^{-1}T \\
& \ge \int_{\rr_+} \n_0 \varphi^+_{\tilde{N}_2} dx - 4 \hat{C} N_2^{-1}T \\
& \ge \int_{B^+_{\tilde{N}_2}} \n_0 dx - 4 \hat{C} N_2^{-1}T
\ge \frac{1}{4} \tilde{M}_1.
\ea\ee

Then, for any $v \in D_+^1(\rr_+)$, we extend it to $\rr$ by
\be\la{1bkj15}\ba
\tilde{v}(x_1,x_2) \triangleq
\begin{cases}
v(x_1,x_2), \ & x_2 \ge 0, \\
v(x_1,-x_2), \ & x_2 < 0,
\end{cases}
\ea\ee
so that $\tilde{v} \in D^1(\rr)$.
In addition, we extend $\n$ by zero outside $\mathbb{R}^2_+$ and denote the extension by $\tilde{\n}$.

By virtue of Lemma \ref{jqgj}, (\ref{1bkj11}), (\ref{1bkj12}), and (\ref{1bkj14}), one has
\be\la{1bkj16}\ba
\| \tilde{v} \bar{x}^{-\eta} \|_{ L^{(2+\ep)/\eta} (\rr) }
\le C \| \sqrt{ \tilde{\n} } \tilde{v} \|_{L^2(\rr)} + C \| \na \tilde{v} \|_{L^2(\rr)},
\ea\ee
which together with (\ref{1bkj15}) gives
\be\la{1bkj17}\ba
\| v \bar{x}^{-\eta} \|_{ L^{(2+\ep)/\eta} (\rr_+) }
\le C \| \sqrt{\n} v \|_{L^2(\rr_+)} + C \| \na v \|_{L^2(\rr_+)}.
\ea\ee
Hence, we obtain (\ref{1bkj001}) after using (\ref{1bkj11}), which completes the proof of Lemma \ref{bkj1l1}.
\end{proof}

\begin{lemma}\la{bkj1l2}
There exists a positive constant $C$ depending only on
$T$, $a$, $\ga$, $\mu$, $\tilde{M}_1$, $\| \sqrt{\n_0} u_0 \|_{L^2}$, $\| {\bar{x}}^a \rho_0 \|_{L^1}$,
$\| \n_0 \|_{L^\ga}$, and $\| \bar{x}^{\frac{a}{2}} \na d_0 \|_{L^2}$ such that
\be\ba\la{1bkj02}
\sup_{0\leq t\leq T} \|\nabla d \bar{x}^{\frac{a}{2}} \|_{L^{2}}^{2}
+ \int_{0}^{T} \|\nabla^{2} d\bar{x}^{\frac{a}{2}}\|_{L^{2}}^{2} dt
\leq C.
\ea\ee
Moreover, for any $2<p<\infty$, there exists a positive constant $C$ depending only on $p$,
$T$, $a$, $\ga$, $\mu$, $\tilde{M}_1$, $\| \sqrt{\n_0} u_0 \|_{L^2}$, $\| {\bar{x}}^a \rho_0 \|_{L^1}$,
$\| \n_0 \|_{L^\ga}$, and $\| \na d_0 \|_{H^1}$ such that
\be\la{1bkj002}\ba
\sup_{0 \le t \le T} \| \na d \|_{L^p} \le C.
\ea\ee
\end{lemma}

\begin{proof}
First, by adapting the argument used to derive (\ref{1cp02}) and using (\ref{bkjbjtj1}), (\ref{bkjbjtj2}), (\ref{1bkj01}), and (\ref{1bkj001}), we obtain (\ref{1bkj02}).
	
It remains to prove (\ref{1bkj002}).
Applying $\na$ to $(\ref{nlckv})_3$, multiplying the resulting equation by $p |\na d|^{p-2} \na d$, integrating by parts over $\rr_+$, and arguing as in (\ref{bp22}), we obtain
	\be\la{h1}\ba
	& \frac{d}{dt} \int |\na d|^p dx
	- p \int |\na d|^{p-2} \p_i d^j \p_k\p_k \p_i d^j dx \\
	& \le C \| \na d \|^{p+2}_{L^{p+2}} + C \| |\na d|^{\frac{p}{2}} \|^2_{L^4} \| \na u \|_{L^2}.
	\ea\ee
	Integrating by parts over $\mathbb{R}_{+}^2$, one has
	\be\la{h2}\ba
	& - p \int |\na d|^{p-2} \p_i d^j \p_k\p_k \p_i d^j dx \\
	& = -p \int_{\p \mathbb{R}_{+}^2} |\na d|^{p-2} \p_i d^j n_k \p_k\p_i d^j ds
	+ p \int \p_k( |\na d|^{p-2} \p_i d^j ) \p_k (\p_i d^j) dx \\
	& = p \int |\na d|^{p-2} |\na^2 d|^2 dx
	+ \frac{4(p-2)}{p} \int |\na (|\na d|^\frac{p}{2})|^2 dx,
	\ea\ee
	where in the second equality we have used the identity:
	\be\la{h3}\ba
	\p_i d^j n_k \p_k\p_i d^j 
	= \na d^j \cdot \na (n \cdot \na d^j)-\p_i d^j \p_i n_k \p_k d^j  = 0 \quad \text{ on } \p \mathbb{R}_{+}^2,
	\ea\ee
	which follows from (\ref{bkjbjds2}).

Combining (\ref{gn11}), (\ref{h1}), (\ref{h2}), and Young's inequality leads to
\be\la{h4}\ba
& \frac{d}{dt} \int |\na d|^p dx + p \int |\na d|^{p-2} |\na^2 d|^2 dx \\
& \le C \| \na d \|^{p+2}_{L^{p+2}} + C \| |\na d|^{\frac{p}{2}} \|^2_{L^4} \| \na u \|_{L^2} \\
& \le C \| \na d \|^p_{L^p} \| \na^2 d \|^2_{L^2}
+ C \| |\na d|^{\frac{p}{2}} \|_{L^2} \| \na(|\na d|^{\frac{p}{2}}) \|_{L^2} \| \na u \|_{L^2} \\
& \le C \| \na d \|^p_{L^p} \| \na^2 d \|^2_{L^2} + C \| \na d \|^p_{L^p} \| \na u \|^2_{L^2}
+ \frac{p}{2} \int |\na d|^{p-2} |\na^2 d|^2 dx \\
& \le \frac{p}{2} \int |\na d|^{p-2} |\na^2 d|^2 dx 
+ C \| \na d \|^p_{L^p} ( \| \na^2 d \|^2_{L^2} + \| \na u \|^2_{L^2} ),
\ea\ee
which gives
\be\la{h5}\ba
\frac{d}{dt} \int |\na d|^p dx + \frac{p}{2} \int |\na d|^{p-2} |\na^2 d|^2 dx
\le C \| \na d \|^p_{L^p} ( \| \na^2 d \|^2_{L^2} + \| \na u \|^2_{L^2} ).
\ea\ee
Applying Gr\"onwall's inequality to (\ref{h5}) and using (\ref{1bkj01}), we obtain (\ref{1bkj002}), which completes the proof of Lemma \ref{bkj1l2}.
\end{proof}

\begin{lemma}\la{bkj1l3}
Let $G$ be the effective viscous flux defined in \eqref{gw}.
Then we have the following decomposition:
\be\la{bkjevf02}\ba
G = \frac{\p}{\p t}\tilde{H}_1+H_2+H_3.
\ea\ee
Moreover, for any $1<p<\infty$, there exists a positive constant $C$ depending only on $p$ such that
\be\la{bkjevf03}\ba
\| \na \tilde{H}_1 \|_{L^p} \le C \| \n u \|_{L^p}, \quad
\| H_2 \|_{L^{p}} \le C \| \n u\otimes u \|_{L^p}, \quad
\| H_3 \|_{L^{p}} \le C \| |\na d|^2 \|_{L^p}.
\ea\ee
For any $2<r<\infty$, there exists a positive constant $C$ depending only on $r$ such that
\be\la{bkjevf04}\ba
\| \tilde{H}_1 \|_{L^r} \le C \| \n u \|_{L^{\frac{2r}{2+r}}}.
\ea\ee
\end{lemma}
\begin{proof}
First, for $v=(v^1,v^2) \in (L^p(\rr_+))^2$ with $1<p<\infty$, consider the Neumann problem
\be\la{bkjevf1}\ba
\begin{cases}
	\Delta{F} = \div (v) \ \   & \text{ in } \rr_+, \\
	\frac{\p F}{\p n} = v \cdot n \ \  & \text{ on } \p \rr_+.
\end{cases}
\ea\ee
It is shown in \cite{JK} that this problem admits a solution $F\in\dot W^{1,p}(\rr_+)$, unique up to an additive constant, satisfying
\be\la{bkjevf2}\ba
\| \na F \|_{L^p} \le C \| v \|_{L^p}.
\ea\ee
Then, we estimate $F$ by extending the problem to the whole space.
Define
\be\la{bkjevf3}\ba
\tilde{F}(x_1,x_2) & \triangleq
\begin{cases}
	F(x_1,x_2), \ & x_2 \ge 0, \\
	F(x_1,-x_2), \ & x_2 < 0,
\end{cases} \\
\tilde{v}^i(x_1,x_2) & \triangleq
\begin{cases}
	v^i(x_1,x_2), \quad & x_2 \ge 0, \\
	(-1)^{i+1} v^i(x_1,-x_2), \quad & x_2 < 0.
\end{cases}
\ea\ee
A direct calculation shows that $\tilde{F}$ satisfies
\be\la{bkjevf4}\ba
	\Delta{\tilde{F}} = \div (\tilde{v}) \ \   & \text{ in } \rr.
\ea\ee
If, in addition, $v\in L^{\frac{2r}{2+r}}(\mathbb{R}^2_+)$, then for any $2<r<\infty$, the whole-space solution class in (\ref{bkjevf4}) has a unique representative $\tilde F\in L^r(\mathbb{R}^2)$; we choose this representative
and set $F=\tilde F|_{\mathbb{R}^2_+}$. By the properties of Riesz potentials (see \cite{SE}),
\be\la{bkjevf5}\ba
\| \tilde{F} \|_{L^r} \le C(r) \| \tilde{v} \|_{L^{\frac{2r}{2+r}}},
\ea\ee
which together with (\ref{bkjevf3}) gives
\be\la{bkjevf6}\ba
\| F \|_{L^r} \le C(r) \| v \|_{L^{\frac{2r}{2+r}}}.
\ea\ee

Let $\tilde{H}_1$ be the $L^r$ representative of the solution
to the following Neumann problem:
\be\la{bkjevf7}\ba
\begin{cases}
	\Delta{\tilde{H}_1} = \div (\n u) \ \   & \text{ in } \rr_+, \\
	\frac{\p \tilde{H}_1}{\p n}=0 \ \  & \text{ on } \p \rr_+.
\end{cases}
\ea\ee
By the boundary condition $u \cdot n = 0$ on $\p \mathbb{R}^2_{+}$,
we deduce from (\ref{bkjevf1}), (\ref{bkjevf2}), and (\ref{bkjevf6}) that the system is solvable,
and for any $1<p<\infty$ and $2<r<\infty$,
\be\la{bkjevf8}\ba
\| \na \tilde{H}_1 \|_{L^p} \le C(p) \| \n u \|_{L^p}, \quad
\| \tilde{H}_1 \|_{L^r} \le C(r) \| \n u \|_{L^{\frac{2r}{2+r}}}.
\ea\ee
On the other hand, we set
\be\la{bkjevf9}\ba
H_1\triangleq\frac{\p}{\p t}\tilde H_1.
\ea\ee
Differentiating (\ref{bkjevf7}) with respect to $t$, we find that $H_1$ satisfies
\be\la{bkjevf10}\ba
\begin{cases}
	\Delta{H_1} = \frac{\p}{\p t} \div (\n u) \ \   & \text{ in } \rr_+, \\
	\frac{\p H_1}{\p n}=0 \ \  & \text{ on } \p \rr_+.
\end{cases}
\ea\ee

Next, let $H_2$ be the $L^p$ representative of the solution to the boundary value problem:
\be\la{bkjevf11}\ba
\begin{cases}
	\Delta{H_2} = \div \div ( \n u\otimes u ) \ \   & \text{ in } \rr_+, \\
	\frac{\p H_2}{\p n}= \div ( \n u\otimes u ) \cdot n \ \  & \text{ on } \p \rr_+.
\end{cases}
\ea\ee
To estimate $H_2$, we extend the problem to $\rr$ by defining
\be\la{bkjevf12}\ba
\tilde{H}_2(x_1,x_2) & \triangleq
\begin{cases}
	H_2(x_1,x_2), \ & x_2 \ge 0, \\
	H_2(x_1,-x_2), \ & x_2 < 0.
\end{cases}
\ea\ee
For $i,j=1,2$, set $M^{ij} \triangleq \n u^i u^j$ and
\be\la{bkjevf13}\ba
\tilde{M}^{ij}(x_1,x_2) & \triangleq
\begin{cases}
	M^{ij}(x_1,x_2), \ & x_2 \ge 0, \\
	(-1)^{i+j} M^{ij}(x_1,-x_2), \ & x_2 < 0.
\end{cases}
\ea\ee
From (\ref{bkjevf11}), (\ref{bkjevf12}), (\ref{bkjevf13}), and the boundary conditions $u \cdot n = 0$ on $\p \mathbb{R}^2_{+}$, we deduce that $\tilde{H}_2$ satisfies
\be\la{bkjevf14}\ba
	\Delta{\tilde{H}_2} = \p_i \p_j \tilde{M}^{ij} \ \   & \text{ in } \rr.
\ea\ee
Standard Calder\'on-Zygmund estimates (see \cite{SE}) imply that for any $1<p<\infty$,
\be\la{bkjevf15}\ba
\| \tilde{H}_2 \|_{L^{p}} \le C \| \tilde{M}^{ij} \|_{L^p}.
\ea\ee
This, combined with (\ref{bkjevf12}) and (\ref{bkjevf13}), gives
\be\la{bkjevf16}\ba
\| H_2 \|_{L^{p}} \le C \| \n u \otimes u \|_{L^p}.
\ea\ee
Now, let $H_3$ be the $L^p$ representative of the solution to the Neumann problem
\be\la{bkjevf160}\ba
\begin{cases}
	\Delta{H_3} = \div ( \na d \cdot \Delta d ) \ \   & \text{ in } \rr_+, \\
	\frac{\p H_3}{\p n}= n \cdot \na d \cdot \Delta d \ \  & \text{ on } \p \rr_+.
\end{cases}
\ea\ee
Since
\be\la{bkjevf161}\ba
\na d \cdot \Delta d = \div \left( \na d \odot \na d - \frac{1}{2} |\na d|^2 \mathbb{I}_2 \right),
\ea\ee
where $\mathbb{I}_2$ is the $2 \times 2$ identity matrix,
an argument similar to that for $H_2$ shows that for any $1<p<\infty$,
\be\la{bkjevf162}\ba
\| H_3 \|_{L^{p}} \le C \| |\na d|^2 \|_{L^p}.
\ea\ee
Furthermore, from $(\ref{nlckv})_2$ and (\ref{gw}), we conclude that
\be\la{bkjevf17}\ba
\rho \dot{u} = \na G + \mu \na^\bot \o - \na d \cdot \Delta d,
\ea\ee
which together with the boundary conditions $u \cdot n=0$ and $\o=0$ on $\p \mathbb{R}^2_{+}$ implies that $G$ satisfies
\be\la{bkjevf18}\ba
\begin{cases}
	\Delta G = \div \left( \rho \dot{u} + \na d \cdot \Delta d \right) & \mathrm{in}\, \,  \rr_+, \\
	\frac {\p G}{\p n}= \left( \rho \dot{u} + \na d \cdot \Delta d \right) \cdot n & \mathrm{on}\, \,  \p \rr_+.
\end{cases}
\ea\ee
Since $\rho\dot u=\partial_t(\rho u)+\div(\rho u\otimes u)$,
and $\partial_t(\rho u)\cdot n=\partial_t(\rho u\cdot n)=0$
on $\partial\mathbb{R}^2_+$, the definitions of $\tilde H_1$,
$H_2$, and $H_3$ imply that $\partial_t\tilde H_1+H_2+H_3$ satisfies
the same Neumann problem as $G$. By the uniqueness of the Neumann
problem in the integrable classes specified above, we obtain the decomposition (\ref{bkjevf02}). Moreover, (\ref{bkjevf8}), (\ref{bkjevf16}), and (\ref{bkjevf162}) yield (\ref{bkjevf03}) and (\ref{bkjevf04}). This completes the proof of Lemma \ref{bkj1l3}.

\end{proof}

\begin{lemma}\la{bkj1l4}
For any $1 \le p < \infty$, there exists a positive constant $C$ depending only on
$p$, $T$, $a$, $\ga$, $\mu$, $\tilde{M}_1$, $\beta$, and $E_1$ such that
\be\la{1bkj04}\ba
\sup_{0\le t\le T} \| \n \|_{L^p} \le C.
\ea\ee
\end{lemma}
\begin{proof}
First, from $(\ref{nlckv})_1$, (\ref{gw}), and (\ref{bkjevf02}), we deduce that
\be\la{1bkj41}\ba
\frac{D}{Dt} \left( \Lambda_1(\n) + \tilde{H}_1 \right) + P = u \cdot \na \tilde{H}_1 - H_2 - H_3,
\ea\ee
where
\be\la{1bkj42}\ba
\Lambda_1(\n) = 2\mu \log \n + \beta^{-1} \n^\beta,
\ea\ee
and $\tilde{H}_1$, $H_2$, and $H_3$ are defined in Lemma \ref{bkj1l3}.

Set $f \triangleq \max\{ \Lambda_1(\n) + \tilde{H}_1,0 \}$.
Multiplying (\ref{1bkj41}) by $2 m \n f^{2m-1}$ with $m > 1$, integrating by parts over $\rr_+$, and using $(\ref{nlckv})_1$, we derive
\be\la{1bkj43}\ba
& \frac{d}{dt} \int \n f^{2m} dx \le C \int \n f^{2m-1} ( |u \cdot \na \tilde{H}_1| + |H_2| + |H_3| ) dx \\
& \le C \int \n^{1-1/(2m)} f^{2m-1} \n^{1/(2m)} ( |u \cdot \na \tilde{H}_1| + |H_2| + |H_3| ) dx \\
& \le C \| \n f^{2m} \|^{1-1/(2m)}_{L^1}
\left( \| \n^{1/(2m)} |u| |\na \tilde{H}_1| \|_{L^{2m}}
+ \| \n^{1/(2m)} |H_2| \|_{L^{2m}} + \| \n^{1/(2m)} |H_3| \|_{L^{2m}} \right).
\ea\ee
By virtue of (\ref{1bkj01}), (\ref{1bkj001}), (\ref{bkjevf03}), and (\ref{bkjevf04}), we obtain for $s=\frac{1}{2} \frac{3(\beta-1)}{ (3+a) (2\beta m + 1)-3 } \in (0,\frac{1}{2})$ that
\be\la{1bkj44}\ba
& \| \n^{1/(2m)} |u| |\na \tilde{H}_1| \|_{L^{2m}} \\
& \le C \| (\n \bar{x}^a)^{s} \|_{L^{1/s}}
\| u \bar{x}^{-a s} \|_{L^{3/(as)}}
\| \n^{1/(2m)-s} \|_{ L^{ (2\beta m+1)/(1/(2m) - s) } }
\| \na \tilde{H}_1 \|_{ L^{2(2\beta m+1)/(\beta+1)} } \\
& \le C (1+\| \na u \|_{L^2}) \| \n \|^{1/(2m)-s}_{L^{2 \beta m+1}}
\| \n u \|_{ L^{2(2\beta m+1)/(\beta+1)} } \\
& \le C (1+\| \na u \|_{L^2}) \| \n \|^{1/(2m)-s}_{L^{2 \beta m+1}}
\| (\n \bar{x}^a)^{s} \|_{L^{1/s}}
\| u \bar{x}^{-a s} \|_{L^{3/(as)}}
\| \n^{1-s} \|_{ L^{ (2\beta m+1)/(1-s) } } \\
& \le C \left( 1 + \| \n \|^{1+1/(2m)}_{L^{2\beta m+1}} \right) \left(1 + \| \na u \|^2_{L^2} \right),
\ea\ee
and
\be\la{1bkj45}\ba
\| \n^{1/(2m)} |H_2| \|_{L^{2m}}
& \le C \| \n \|^{1/(2m)}_{L^{2\beta m+1}} \| H_2 \|_{ L^{(2\beta m+1)/\beta} } \\
& \le C \| \n \|^{1/(2m)}_{L^{2\beta m+1}} \| \n |u|^2 \|_{ L^{(2\beta m+1)/\beta} } \\
& \le C \| \n \|^{1/(2m)}_{L^{2\beta m+1}}
\| (\n \bar{x}^a)^{2s} \|_{L^{1/(2s)}}
\| |u|^2 \bar{x}^{-2 a s} \|_{L^{3/(2as)}}
\| \n^{1-2s} \|_{ L^{ (2\beta m+1)/(1-2s) } } \\
& \le C \left( 1 + \| \n \|^{1+1/(2m)}_{L^{2\beta m+1}} \right) \left(1 + \| \na u \|^2_{L^2} \right).
\ea\ee
Moreover, it follows from (\ref{1bkj01}), (\ref{bkjevf03}),
(\ref{gn11}), and Young's inequality that, for
$r=(2\beta m+1)/\beta>2$,
\be\la{1bkj46}\ba
\| \n^{1/(2m)} |H_3| \|_{L^{2m}}
& \le C \| \n \|^{1/(2m)}_{L^{2\beta m+1}} \| H_3 \|_{ L^{(2\beta m+1)/\beta} } \\
& \le C \| \n \|^{1/(2m)}_{L^{2\beta m+1}}
\| \na d \|^2_{L^{2r}} \\
& \le C \| \n \|^{1/(2m)}_{L^{2\beta m+1}}
\| \na d \|^{2/r}_{L^2}\| \na^2d \|^{2-2/r}_{L^2} \\
& \le C \| \n \|^{1/(2m)}_{L^{2\beta m+1}}
\left(1+\| \na^2 d \|^2_{L^2}\right).
\ea\ee
On the other hand, (\ref{1bkj01}) and Young's inequality show that
\be\la{1bkj47}\ba
\int \n^{2 \beta m +1} dx
& = \int_{ \{\n \le 2\} } \n^{2 \beta m +1} dx + \int_{ \{\n > 2\} } \n^{2 \beta m +1} dx \\
& \le C \int_{ \{\n \le 2\} } \n dx + C \int_{ \{\n > 2\} } \n f^{2 m} dx
+ C \int_{ \{\n > 2\} } \n |\tilde{H}_1|^{2 m} dx \\
& \le C + C \int \n f^{2 m} dx + C \| \n \|_{L^{ (2 \beta m+1)/( 2\beta m +1-m ) } }
\| \tilde{H}_1 \|^{2 m}_{ L^{2(2\beta m+1)} } \\
& \le C + C \int \n f^{2 m} dx
+ C ( 1 + \| \n \|_{ L^{2 \beta m+1} } ) \| \n \|^m_{ L^{2\beta m+1} } \\
& \le \frac{1}{2} \int \n^{2 \beta m +1} dx + C + C \int \n f^{2 m} dx,
\ea\ee
where we have used the following estimate:
\be\la{1bkj48}\ba
\| \tilde{H}_1 \|_{L^{2(2\beta m+1)}}
\le C \| \n u \|_{L^{(2\beta m+1)/(\beta m+1)}}
\le C \| \n \|^{1/2}_{ L^{2\beta m+1} },
\ea\ee
due to (\ref{1bkj01}) and (\ref{bkjevf04}).

Combining (\ref{1bkj43}), (\ref{1bkj44}), (\ref{1bkj45}), (\ref{1bkj46}), and (\ref{1bkj47}) leads to
\be\la{1bkj49}\ba
\frac{d}{dt} \int \n f^{2m} dx
& \le C \left( 1 + \| \n f^{2m} \|_{L^1} \right)
\left( 1 + \| \na u \|^2_{L^2} + \| \na^2 d \|^2_{L^2} \right).
\ea\ee
Applying Gr\"onwall's inequality and (\ref{1bkj01}) yields
\be\la{1bkj410}\ba
\sup_{0 \le t \le T} \int \n f^{2m} dx \le C,
\ea\ee
which together with (\ref{1bkj47}) and H\"older's inequality gives (\ref{1bkj04}) and completes the proof of Lemma \ref{bkj1l4}.
\end{proof}

\begin{lemma}\la{bkj1l5}
There exists a positive constant $C$ depending only on
$T$, $a$, $\ga$, $\mu$, $\tilde{M}_1$, $\beta$, and $E_1$ such that
\be\la{1bkj05}\ba
\sup_{0\le t\le T}\int \n |u|^{2+\nu} dx \le C,
\ea\ee
where
\be\la{1bkj005}\ba
\nu \triangleq R_T^{-\frac{\beta}{2}} \nu_1,
\ea\ee
for some suitably small generic constant $\nu_1 \in (0,1)$ depending only on $\mu$ and $\ga$.
\end{lemma}
\begin{proof}
First, using (\ref{dc1}) and following the proof of \cite[Lemma 3.4]{FLL}, we find that there exist positive constants $\Lambda$ and $\hat{\nu}$ such that for any $\nu \in (0,\hat{\nu})$,
\be\ba\la{1bkj51}
\int_\OM |u|^{\nu} |\na u|^2 dx \le \Lambda \int_\OM |u|^{\nu} \left( (\div u)^2 +\o^2  \right)dx.
\ea\ee
Next, we multiply $(\ref{nlckv})_2$ by $|u|^\nu u$, integrate by parts over $\rr_+$, and use (\ref{1bkj02}), (\ref{1bkj002}), (\ref{1bkj04}), and (\ref{1bkj51}) to derive
\be\ba\nonumber
& \frac{1}{(2+\nu)} \frac{d}{dt}\int \n |u|^{2+\nu} dx
+ \int |u|^\nu \left(\mu \o^2
+ (2\mu+\lam) (\div u)^2 \right) dx \\
& \le C \nu  \int  \left( (2\mu+\lam) |\div u|+\mu |\o| \right)  |u|^\nu |\na u| dx
+ C \int \n^\ga |u|^\nu |\na u| dx + C \int |u|^\nu |\na u| |\na d|^2 dx \\
& \le \frac{1}{4} \int |u|^\nu \left(\mu \o^2
+ (2\mu+\lam) (\div u)^2 \right) dx
+ C\nu_1^2 \int |u|^\nu |\na u|^2 dx + C \int \n |u|^{2+\nu} dx \\
& \quad + C \int \n^{(2+\nu)\ga-\nu/2} dx
+ \frac{\mu}{4 \Lambda} \int |u|^{\nu} |\na u|^2 dx + C \int |u|^\nu |\na d|^4 dx \\
& \le \left( \frac{1}{2} + C \nu_1^2 \right) \int |u|^\nu \left(\mu \o^2 + (2\mu+\lam) (\div u)^2 \right) dx
+ C \int \n |u|^{2+\nu} dx \\
& \quad + C + C \| u \bar{x}^{-\frac{a}{2}} \|_{L^4} \| \na d \bar{x}^{\frac{a}{2}} \|_{L^2} \| \na d \|^3_{L^{12}} + C \| \na d \|^4_{L^4} \\
& \le \left( \frac{1}{2} + C \nu_1^2 \right) \int |u|^\nu \left(\mu \o^2 + (2\mu+\lam) (\div u)^2 \right) dx
+ C \int \n |u|^{2+\nu} dx + C ( 1 + \| \na u \|^2_{L^2} ).
\ea\ee
Choosing suitably small $\nu_1$ then yields
\be\la{1bkj52}\ba 
\frac{d}{dt}\int \n |u|^{2+\nu}dx \le C \int \n |u|^{2+\nu} dx + C ( 1 + \| \na u \|^2_{L^2} ),
\ea\ee
which together with Gr\"onwall's inequality and (\ref{1bkj01}) implies (\ref{1bkj05}) and finishes the proof of Lemma \ref{bkj1l5}.
\end{proof}

\begin{lemma}\la{bkj1l6}
For any $2 < p<\infty$ and $\ep \in (0,1)$, there exists a positive constant $C$ depending only on $T$, $a$, $\mu$, $\ga$, $p$, $\ep$, $\tilde{M}_1$, $E_1$, and $\beta$ such that
\be\la{1bkj06}\ba
\| \nabla u \|_{L^{p}}
& \le C R^{\frac{1}{2}-\frac{1}{p}+\ep}_T A_1 \left( \frac{A^2_2}{A_1^2} \right)^{\frac{1}{2}-\frac{1}{p}} + C R^\ep_T A_1,
\ea\ee
where $A_1$, $A_2$, and $R_T$ are defined in \eqref{a1}, \eqref{a2}, and \eqref{mdsj}, respectively.
\end{lemma}
\begin{proof}
First, note that (\ref{bkjevf17}) and the boundary conditions (\ref{bkjbjtj1}) imply that $G$ satisfies (\ref{bkjevf18}) and $\o$ satisfies the following elliptic equation:
\be\la{1bkj62}\ba
\begin{cases}
\mu \Delta \o =\na^\bot \cdot \left( \rho \dot{u} + \na d \cdot \Delta d \right) & \mathrm{in}\, \,  \rr_+, \\
\o = 0 & \mathrm{on}\, \,  \p \rr_+.
\end{cases}
\ea\ee
Applying the standard $L^p$ estimate of elliptic equations (see \cite{JK}) together with (\ref{1bkj001}), we obtain for any $2\le p<\infty$,
\be\la{1bkj63}\ba
\| \na G \|_{L^p} + \| \na \o \|_{L^p} \le C \left( \| \n \dot{u}\|_{L^p}
+ \| \na d \cdot \Delta d \|_{L^p} \right).
\ea\ee
It thus follows from (\ref{1bkj63}), (\ref{gn11}), and (\ref{1bkj002}) that
\be\la{1bkj64}\ba
\| \na G \|_{L^2} + \| \na \o \|_{L^2}
& \le C \left( \| \rho \dot{u} \|_{L^2} + \| \na d \cdot \Delta d \|_{L^2} \right) \\
& \le C \left( R^{1/2}_T A_2 + \| \na d \|_{L^4} \| \Delta d \|_{L^4} \right) \\
& \le C R^{1/2}_T A_2 + C A_1.
\ea\ee
Moreover, (\ref{1bkj04}) and H\"older's inequality give
\be\la{1bkj66}\ba
\| G \|^2_{L^2} \le C R^\beta_T A_1^2, \quad
\left\| \frac{G}{2\mu+\lam} \right\|^2_{L^2} \le C A_1^2.
\ea\ee
In view of (\ref{gn11}), (\ref{dc1}), and (\ref{1bkj66}), we choose $0<\delta<\min\{2/p,2\varepsilon/\beta\}$ such that
\be\la{1bkj67}\ba
\| \na u \|_{L^p}
& \le C \left(\| \div u \|_{L^p} + \| \o \|_{L^p} \right) \\
& \le C \left\| \frac{G}{2\mu+\lam} \right\|_{L^p}
+ C \left\| \frac{P}{2\mu+\lam} \right\|_{L^p}
+ C \| \o \|^{\frac{2}{p}}_{L^2} \| \o \|^{1-\frac{2}{p}}_{H^1} \\
& \le C \left\| \frac{G}{2\mu+\lam} \right\|^{\frac{2}{p}-\delta}_{L^2}
\| G \|_{L^{\frac{2(1+\delta)p-4}{p\delta}}}^{-\frac{2}{p}+1+\delta}
+ C \left( A_1^{\frac{2}{p}} \| \na \o \|^{1-\frac{2}{p}}_{L^2} + A_1 \right) \\
& \le C A_1^{\frac{2}{p}-\delta} \| G \|^{\delta}_{L^2} \| \na G \|^{1-\frac{2}{p}}_{L^2}
+ C \left( A_1^{\frac{2}{p}} \| \na \o \|^{1-\frac{2}{p}}_{L^2} + A_1 \right) \\
& \le C R^{\frac{\beta\delta}{2}}_T A_1^{\frac{2}{p}}
\left( \| \na G \|_{L^2} + \| \na \o \|_{L^2} \right)^{1-\frac{2}{p}} + C A_1.
\ea\ee
This, combined with (\ref{1bkj64}) and $\beta \delta/2\le \ep$, yields (\ref{1bkj06}), which completes the proof of Lemma \ref{bkj1l6}.
\end{proof}

\begin{lemma}\la{bkj1l7}
For any $\ep \in (0,1)$, there exists a positive constant $C$ depending only on
$T$, $a$, $\mu$, $\ga$, $\ep$, $\tilde{M}_1$, $E_1$, and $\beta$ such that
\be\la{1bkj07}\ba
&\sup_{0\le t\le T} \log A_1^2(t) + \int_0^T\frac{A_2^2(t)}{ A_1^2(t)}dt
\le C R_T^{1+\ep}.
\ea\ee
\end{lemma}
\begin{proof}
Multiplying (\ref{bkjevf17}) by $2 \dot{u}$, integrating by parts over $\rr_+$, and using (\ref{bkjbjds}), we arrive at
\be\la{1bkj71}\ba
& \frac{d}{dt} \int \left(\mu \o^2 + \frac{G^2}{2\mu + \lam}\right)dx
+ 2 \int \n |\dot{u}|^2 dx \\
& = - \mu \int \o^2 \div u dx + 4 \int G \nabla u^1 \cdot\nabla^{\perp}u^2 dx
-2 \int G (\div u)^2 dx \\
& \quad - \int \frac{ (\beta-1)\lam - 2\mu }{(2\mu + \lam)^2} G^2\div u dx
-2\beta \int \frac{ \lam P }{ (2\mu +\lam)^2 } G \div u dx \\
& \quad + 2\ga \int \frac{P}{2\mu +\lam} G \div u dx
-2 \int \dot{u} \cdot \na d \cdot \Delta d dx
= \sum_{i=1}^7 I_i.
\ea\ee

We now estimate each $I_i$ as follows:

By virtue of (\ref{gn11}), (\ref{1bkj64}), and Young's inequality, we have
\be\la{1bkj72}\ba
I_1 \le C \| \o \|^2_{L^4} \| \div u \|_{L^2}
\le C \| \o \|_{L^2} \| \o \|_{H^1} A_1
\le \frac{1}{8} A^2_2 + C R_T A_1^4.
\ea\ee
It follows from (\ref{1bkj67}), (\ref{1bkj64}), and H\"older's inequality that
\be\la{1bkj73}\ba
I_2 & \le C \int |G| |\na u|^2 dx \\
& \le C \| G \|_{L^p} \| \na u \|^2_{L^{2p/(p-1)}} \\
& \le C \| G \|^{2/p}_{L^2} \| \na G \|^{1-2/p}_{L^2}
\left( R^{\ep}_T A_1^{1-\frac{1}{p}} \left( \| \na G \|_{L^2} + \| \na \o \|_{L^2} \right)^{\frac{1}{p}} + A_1 \right)^2 \\
& \le C R_T^{\frac{\beta}{p}+2\ep} A_1^2 \left( \| \na G \|_{L^2} + \| \na \o \|_{L^2} \right)
+ C R_T^{\frac{\beta}{p}} A_1^{2+\frac{2}{p}} \| \na G \|^{1-\frac{2}{p}}_{L^2} \\
& \le C R_T^{\frac{\beta}{p}+2\ep} A_1^2 \left( \| \na G \|_{L^2} + \| \na \o \|_{L^2} \right) + C R_T^{\frac{\beta}{p}} A_1^4 \\
& \le \frac{1}{8} A^2_2 + C R^{1+\ep}_T A_1^4.
\ea\ee
Using (\ref{dc1}) and (\ref{1bkj04}), we arrive at
\be\la{1bkj65}\ba
\| \na u \|_{L^2} & \le C \left( \| \div u \|_{L^2} + \| \o \|_{L^2} \right) \\
& \le C \left\| \frac{G}{2\mu+\lam} \right\|_{L^2}
+ C \left\| \frac{P}{2\mu+\lam} \right\|_{L^2} + C \| \o \|_{L^2}
\le C A_1.
\ea\ee
Combining (\ref{1bkj04}), (\ref{1bkj65}), and (\ref{1bkj66}), we obtain for any $\ep \in (0,1)$,
\be\la{1bkj74}\ba
\sum_{i=3}^{6} |I_i| & \le C \int \frac{G^2 |\div u|}{2\mu+\lam} dx
+ C \int \frac{|G|}{2\mu+\lam} P |\div u| dx \\
& \le C \| \na u \|_{L^2} \left\| \frac{G^2}{2\mu+\lam} \right\|_{L^2}
+ C \| \na u \|_{L^2} \left\| \frac{G^2}{2\mu+\lam} \right\|^{1/2}_{L^2} \| P \|_{L^4} \\
& \le C \| \na u \|_{L^2} \left\| \frac{G^2}{2\mu+\lam} \right\|_{L^2}
+ C ( 1 + \| \na u \|^2_{L^2} ) \\
& \le C R_T^{ (1+\ep\beta)/2 } \| \na u \|_{L^2} A_1 \left( A_2 + A_1 \right) + C A_1^2 \\
& \le \frac{1}{8} A^2_2 + C R_T^{ 1+\ep\beta } A_1^4,
\ea\ee
where we have used the following estimate:
\be\la{1bkj75}\ba
\left\|\frac{G^2}{2\mu+\lam} \right\|_{L^2}
& \le C \left\| \frac{G}{\sqrt{2\mu+\lam}} \right\|_{L^2}^{1-\ep}
\| G \|_{ L^{2(1+\ep)/\ep}}^{1+\ep} \\
& \le C (\ep) A_1^{ 1-\ep } \| G \|_{L^2}^{\ep} \| \na G \|_{L^2} \\
& \le C(\ep) A_1^{ 1-\ep } R_T^{ (\ep\beta)/2 } A_1^\ep \left( A_1 + R_T^{1/2} A_2 \right) \\
& \le C(\ep) R_T^{ (1+\ep\beta)/2 } A_1 \left( A_2 + A_1 \right),
\ea\ee
due to (\ref{gn11}) and (\ref{1bkj66}).

For $I_7$, proceeding similarly to the derivation of (\ref{1cp57}) and
using (\ref{1bkj01}), (\ref{1bkj001}), (\ref{1bkj02}), (\ref{1bkj002}), and (\ref{1bkj65}), we obtain
\be\la{1bkj77}\ba
I_7 \le \frac{d}{dt} \int \left( 2 (\na d \odot \na d) \cdot \na u - |\na d|^2 \div u \right) dx
+ \frac{1}{4} A^2_2 + C A_1^4.
\ea\ee

Moreover, applying $\na$ to $(\ref{nlckv})_3$, taking the inner product of the resulting equation with itself, and arguing as in (\ref{1cp58})
while using (\ref{1bkj01}), (\ref{1bkj001}), (\ref{1bkj02}), (\ref{1bkj002}), and (\ref{1bkj65}), we get
\be\la{1bkj78}\ba
\frac{d}{dt} \| \Delta d \|^2_{L^2} + \| \na d_t \|^2_{L^2} + \| \na \Delta d \|^2_{L^2}
\le \frac{1}{8} \| \na \Delta d \|^2_{L^2} + C A_1^2 \left( A_1^2 + \| \na^2 d \bar{x}^{\frac{a}{2}} \|^2_{L^2} \right).
\ea\ee

Substituting (\ref{1bkj72})--(\ref{1bkj74}) and (\ref{1bkj77}) into (\ref{1bkj71}), and adding the result to (\ref{1bkj78}), we arrive at
\be\la{1bkj79}\ba
\frac{d}{dt} \tilde{B}_1 + \frac{1}{4} A^2_2 \le C R^{1+\ep}_T A_1^2 \left( A_1^2 + \| \na^2 d \bar{x}^{\frac{a}{2}} \|^2_{L^2} \right),
\ea\ee
where
\be\la{1bkj710}\ba
\tilde{B}_1(t) \triangleq A_1^2(t) - \int \left( 2 (\na d \odot \na d) \cdot \na u - |\na d|^2 \div u \right) dx.
\ea\ee
From (\ref{1bkj002}), (\ref{1bkj65}), and Cauchy's inequality, we have
\be\la{1bkj711}\ba
\left| \int \left(2 (\na d \odot \na d) \cdot \na u - |\na d|^2 \div u \right) dx \right|
\le C \| \na d \|^2_{L^4} \| \na u \|_{L^2}
\le \frac{1}{2} A_1^2 + \hat{C}_3,
\ea\ee
which together with (\ref{1bkj710}) yields
\be\la{1bkj712}\ba
\frac{1}{2} A_1^2(t) \le \tilde{B}_1(t) + \hat{C}_3 \le 2 \left( A_1^2(t) + \hat{C}_3 \right).
\ea\ee

Dividing (\ref{1bkj79}) by $\tilde{B}_1(t) + \hat{C}_3$, integrating over $(0,T)$, and using (\ref{1bkj01}), (\ref{1bkj02}), (\ref{1bkj04}), and (\ref{1bkj712}), we obtain (\ref{1bkj07}).
This completes the proof of Lemma \ref{bkj1l7}.
\end{proof}

To derive the upper bound for the density, following the arguments used for the Cauchy problem, we need to derive pointwise estimates of the effective viscous flux.

Note that for $x,y \in \rr_+$, the Neumann Green's function for $\rr_+$ is given by
\be\la{bkjglhs}\ba
N(x,y) =- \frac{1}{2\pi} \left( \log|x-y| + \log|x-y^*| \right),
\ea\ee
where $y^* \triangleq (y_1,-y_2)$ and $\frac{\p N}{\p n_y}|_{\p \rr_+}=0$.

Using the Green's function (\ref{bkjglhs}) and the fact that $G$ satisfies the elliptic equation (\ref{bkjevf18}), we establish the required pointwise estimates of the effective viscous flux, as presented in the following lemma.

\begin{lemma}\label{bkj1l8}
For any $x \in \rr_+$, there exists a positive constant $C$ depending only on
$T$, $a$, $\mu$, $\ga$, $\tilde{M}_1$, $E_1$, and $\beta$ such that
\be\la{1bkj08}\ba
G(x,t) & \le \frac{D}{Dt} \psi (x,t)
+ C A_1^{\frac{1}{2}} A^{\frac{1}{2}}_2 + C A_1 + |J|,
\ea\ee
where
\be\la{1bkj008}\ba
\psi \triangleq \int_{\rr_+} \na_y N(x,y) \cdot \n u(y) dy,
\ea\ee
and $J$ satisfies
\be\la{1bkj0008}\ba
|J| \le C \sup_{x \in \ol{\rr_+} } \int_{\rr_+} \frac{ |u(x)-u(y)| }{ |x-y|^2 } \n |u| dy.
\ea\ee
\end{lemma}
\begin{proof}
Since $G$ satisfies (\ref{bkjevf18}), the Green's representation formula gives, for any $x \in \rr_+$,
\be\la{1bkj81}\ba
G(x,t) = & \int_{\rr_+} \na_y N(x,y) \cdot \left( \rho \dot{u} + \na d \cdot \Delta d \right) dy.
\ea\ee

For the first term on the right-hand side of (\ref{1bkj81}), integrating by parts and using the boundary condition $u \cdot n = 0$ on $\p \rr_+$, we deduce that
\be\la{1bkj82}\ba
& \int_{\rr_+} \na_y N(x,y) \cdot \n \dot{u} dy \\
& = \int_{\rr_+} \na_y N(x,y) \cdot \left( (\n u)_t + \div(\n u \otimes u ) \right) dy \\
& = \frac{\p}{\p t} \int_{\rr_+} \na_y N(x,y) \cdot \n u(y) dy
- \int_{\rr_+} \p_{y_i} \p_{y_j} N(x,y) \n u^i u^j (y) dy \\
& = ( \p_t + u \cdot \na ) \int_{\rr_+} \na_y N(x,y) \cdot \n u(y) dy - J,
\ea\ee
where
\be\la{1bkj83}\ba
J \triangleq \int_{\rr_+} \left( \p_{x_i} \p_{y_j} N(x,y) u^i(x) + \p_{y_i} \p_{y_j} N(x,y) u^i(y) \right) \n u^j (y) dy.
\ea\ee
We rewrite $J$ as
\be\la{1bkj84}\ba
J & = \int_{\rr_+} \p_{x_i} \p_{y_j} N(x,y) \left( u^i(x) - u^i(y) \right) \n u^j (y) dy \\
& + \int_{\rr_+} \left( \p_{x_i} \p_{y_j} N(x,y) + \p_{y_i} \p_{y_j} N(x,y) \right) \n u^i u^j (y) dy.
\ea\ee
For $i,j=1,2$, direct calculation shows
\be\la{1bkj85}\ba
|\p_{x_i} \p_{y_j} N(x,y)| \le C |x-y|^{-2}, \quad
\p_{x_i} \log|x-y| + \p_{y_i} \log|x-y| = 0,
\ea\ee
and
\be\la{1bkj86}\ba
& \p_{y_j} \left( \p_{x_i} \log|x-y^*| + \p_{y_i} \log|x-y^*| \right) \\
& = \p_{y_j} \left( \frac{ (x_i-y_i^*) (\p_{x_i}(x_i) - \p_{y_i}(y_i^*) ) }{|x-y^*|^2} \right) \\
& = \frac{ -\p_{y_j}(y_i^*) (\p_{x_i}(x_i) - \p_{y_i}(y_i^*) ) }{|x-y^*|^2}
+ 2 \frac{ \p_{y_j}(y_j^*) (x_i-y_i^*) (x_j-y_j^*)  (\p_{x_i}(x_i) - \p_{y_i}(y_i^*) ) }{|x-y^*|^4},
\ea\ee
which implies
\be\la{1bkj87}\ba
\left| \left( \p_{x_i} \p_{y_j} N(x,y) + \p_{y_i} \p_{y_j} N(x,y) \right) u^i(y) \right|
\le C \frac{ | u^2(y) | }{|x-y^*|^2}.
\ea\ee

Moreover, for $x=(x_1,x_2) \in \rr_+$, let $\tilde{x}=(x_1,0) \in \p \rr_+$.
Since $u \cdot n = 0$ on $\p \rr_+$, we have $u^2(\tilde{x})=0$.
Consequently,
\be\la{1bkj88}\ba
\frac{ | u^2(y) | }{|x-y^*|^2} = \frac{ | u^2(\tilde{x}) - u^2(y) | }{|x-y^*|^2}
\le \frac{ | u^2(\tilde{x}) - u^2(y) | }{|\tilde{x}-y|^2},
\ea\ee
where we have used the following fact:
\be\la{1bkj89}\ba
|\tilde{x}-y| \le |x-y^*|,
\ea\ee
due to $x_2,\ y_2 \ge 0$.

Combining (\ref{1bkj84}), (\ref{1bkj85}), (\ref{1bkj87}), and (\ref{1bkj88}) leads to
\be\la{1bkj810}\ba
|J| \le C \sup_{x \in \ol{\rr_+} } \int_{\rr_+} \frac{ |u(x)-u(y)| }{ |x-y|^2 } \n |u| dy.
\ea\ee

On the other hand, (\ref{gn11}), (\ref{tygj1}), (\ref{1bkj002}), and H\"older's inequality give
\be\la{1bkj811}\ba
& \left| \int_{\rr_+} \nabla_y N(x,y) \cdot \na d \cdot \Delta d dy \right| \\
& \le \int_{\rr_+}  |x-y|^{-1} | \na d | |\na^2 d| dy \\
& \le \int_{|x-y|<1}  |x-y|^{-1} | \na d | |\na^2 d| dy
+ \int_{|x-y| \ge 1}  |x-y|^{-1} | \na d | |\na^2 d| dy \\
& \le C \| \na d \|_{L^{12}} \| \na^2 d \|_{L^4} + C \| \na d \|_{L^{\frac{12}{5}}} \| \na^2 d \|_{L^4} \\
& \le C A_1 + C A_1^{\frac{1}{2}} A^{\frac{1}{2}}_2.
\ea\ee
The combination of (\ref{1bkj81}), (\ref{1bkj82}), (\ref{1bkj810}), and (\ref{1bkj811}) yields (\ref{1bkj08}) and completes the proof of Lemma \ref{bkj1l8}.
\end{proof}

\begin{lemma}\la{bkj1l9}
There exists a positive constant $C$ depending only on
$T$, $a$, $\mu$, $\ga$, $\tilde{M}_1$, $E_1$, and $\beta$ such that
\be\ba\la{1bkj09}
& \sup_{0\leq t\leq T} \left( \|\n\|_{L^\infty} + \| \na u \|_{L^2} + \| \na d \|_{H^1} \right) \\
& + \int_0^T \left( \| \na u \|^2_{L^2} + \| \sqrt{\n} \dot{u} \|^2_{L^2} + \| \na^2 d \|^2_{H^1} + \| \na d_t \|^2_{L^2} \right) dt
\le C.
\ea\ee
\end{lemma}
\begin{proof}
First, recall from (\ref{1bkj42}) that
\be\la{1bkj91}\ba
\frac{D}{Dt} \Lambda_1(\n) + P = -G,
\ea\ee
where $\Lambda_1(\n) = 2\mu \log \n + \beta^{-1} \n^\beta$.

For $\nu$ given in Lemma \ref{bkj1l5}, using (\ref{1bkj01}), (\ref{1bkj04}), (\ref{1bkj05}), and H\"older's inequality, one derives
\be\la{1bkj92}\ba
& \left| \int_{\rr_+} \na_y N(x,y) \cdot \n u(y) dy \right| \\
& \le C \int_{\rr_+} |x-y|^{-1} \n |u(y)| dy \\
& \le C \int_{ |x-y| \le 1 } |x-y|^{-1} \n |u(y)| dy
+ C \int_{ |x-y| > 1 } |x-y|^{-1} \n |u(y)| dy \\
& \le C \left( \int_{ |x-y| \le 1 } |x-y|^{-\frac{2+\nu}{1+\nu}} dy \right)^{\frac{1+\nu}{2+\nu}}
\left(\int_{\rr_+} \n^{2+\nu} |u|^{2+\nu} dy \right)^{\frac{1}{2+\nu}} \\
& \quad + C \left( \int_{ |x-y| > 1 } |x-y|^{-4} dy \right)^{\frac{1}{4}} \| \sqrt{\n} u \|_{L^2} \| \sqrt{\n} \|_{L^4} \\
& \le C \nu^{-\frac{1+\nu}{2+\nu}} R_T^{\frac{1+\nu}{2+\nu}} + C \\
& \le C R_T^{\frac{2+\beta}{3}},
\ea\ee
which along with (\ref{1bkj008}) shows
\be\la{1bkj93}\ba
\| \psi \|_{L^\infty} \le C R_T^{\frac{2+\beta}{3}}.
\ea\ee

For any $2<p<\infty$, the Sobolev embedding theorem (Theorem 4 of \cite[Chapter 5]{EL}) implies that for any $x, y\in \overline{\rr_+}$,
\bnn
|u(x)-u(y)|\leq C(p )\|\nabla u\|_{L^p}|x-y|^{1-\frac{2}{p}},
\enn
and thus
\be\la{1bkj94}\ba
\int_{\rr_+} \frac{|u(x )-u(y)|}{|x-y|^2}\rho|u|(y) dy
& \leq C(p) \int_{\rr_+} \frac{\|\nabla u\|_{L^p}\cdot|x-y|^{1-\frac{2}{p}}}{|x-y|^2}\rho|u|(y) dy \\
& = C(p) \|\nabla u\|_{L^p}\int_{\rr_+} |x-y|^{-(1+\frac{2}{p})} \rho|u|(y)dy.
\ea\ee

In view of (\ref{jqgj1}), (\ref{1bkj01}), (\ref{1bkj14}), and extending $u$ evenly to $\rr$, we obtain for any $1<r<\infty$,
\be\la{1bkj95}\ba
\|\n u\|_{L^r(\rr_+)}
& \le C r^{\eta_0 + \frac{1}{2}}
\left( 1 + \|\n\|_{L^\infty(\rr_+)} \right)
\left( \| \sqrt{\n} u \|_{L^2(B^{+}_{N_2})} + \|\na  u\|_{L^2(\rr_+)} \right) \\
& \le C r^{\eta_0 + \frac{1}{2}} R_T (1 + \| \na u \|_{L^2(\rr_+)}),
\ea\ee
where $N_2$ is defined in Lemma \ref{bkj1l1}.

Then, for $2<p<6$, $\tau>0$ and $\varepsilon_0\in (0,\frac{p-2}{8})$ to be chosen later, by (\ref{1bkj65}), (\ref{1bkj95}), and H\"older's inequality, we have
\be\la{1bkj96}\ba
& \int_{|x-y|<\tau}|x-y|^{-\left(1+\frac{2}{p}\right)}\rho|u|(y)dy\\
&\leq C \left(\int_{|x-y|<\tau} |x-y|^{-\left(1+\frac{2}{p}\right) (1+\varepsilon_0)} dy\right)^{\frac{1}{1+\varepsilon_0}}
\|\n u\|_{L^{\frac{1+\varepsilon_0}{\varepsilon_0}}} \\
&\leq C(p) R_T (1 + \|\na u\|_{L^2}) \varepsilon_0^{-( \frac{1}{2}+\eta_0 )}
\tau^{1-\frac{2}{p}-\frac{2\varepsilon_0}{1+\varepsilon_0}} \\
&\leq C(p) R_T A_1 \varepsilon_0^{-( \frac{1}{2}+\eta_0 )}
\tau^{1-\frac{2}{p}-\frac{2\varepsilon_0}{1+\varepsilon_0}}.
\ea\ee
Furthermore, from (\ref{1bkj05}) and H\"older's inequality, we conclude that
\be\la{1bkj97}\ba
&\int_{|x-y|>\tau}|x-y|^{-\left(1+\frac{2}{p}\right)}\rho|u|(y)dy\\
&\leq  \left(\int_{|x-y|>\tau} |x-y|^{-\left(1+\frac{2}{p}\right) (\frac{2+\nu}{1+\nu})} dy\right)^{\frac{1+\nu}{2+\nu}}
\left(\int_\Omega\rho^{2+\nu}|u|^{2+\nu}dx\right)^{\frac{1}{2+\nu}}\\
&\leq C(p) R^{ \frac{1+\nu}{2+\nu} }_T \tau^{-\frac{2}{p}+\frac{\nu}{2+\nu}}.
\ea\ee
Choose $\tau>0$ such that
\be\la{1bkj98}\ba
\tau^{-\frac{2}{p}+\frac{\nu}{2+\nu}} = A_1^{ \frac{2}{p} },
\ea\ee
and set
\be\ba\nonumber
\varepsilon_0=\frac{(p-2)\nu}{8+(6-p) \nu}\in \left(0, \frac{p-2}{8}\right),
\ea\ee
which yields
\be\la{1bkj99}\ba 
A_1 \tau^{1-\frac{2}{p}-\frac{2\varepsilon_0}{1+\varepsilon_0}} = A_1^{ \frac{2}{p} }.
\ea\ee
Combining (\ref{1bkj96}), (\ref{1bkj97}), (\ref{1bkj98}), and (\ref{1bkj99}) shows that
\be\la{1bkj910}\ba
\int_{\rr_+} |x-y|^{-(1+\frac{2}{p})} \rho|u|(y)dy
&\leq C(p) R_T \varepsilon_0^{-( \frac{1}{2}+\eta_0 )} A_1^{ \frac{2}{p} }
+ C(p) R^{ \frac{1+\nu}{2+\nu} }_T A_1^{ \frac{2}{p} } \\
& \leq C(p) R_T^{1+\frac{\beta}{4} + \frac{\beta \eta_0}{2}} A_1^{ \frac{2}{p} }.
\ea\ee
Here we used $\varepsilon_0^{-\frac{1}{2}}\le C(p)\nu^{-\frac{1}{2}}\le C(p)R_T^{\frac{\beta}{4}}$, which follows from (\ref{1bkj005}).

From (\ref{1bkj94}), (\ref{1bkj910}), and (\ref{1bkj06}), we deduce that
\be\ba\la{1bkj911}
& \int_{\rr_+} \frac{|u(x )-u(y)|}{|x-y|^2}\rho|u|(y) dy \\
& \le C(p) R_T^{1+\frac{\beta}{4} + \frac{\beta \eta_0}{2}} A_1^{ \frac{2}{p} } \|\nabla u\|_{L^p} \\
& \le C(p) R_T^{\frac{3}{2} + \frac{\beta}{4} + \frac{\beta \eta_0}{2} - \frac{1}{p} + \ep} A_1^{1+\frac{2}{p}} \left( \frac{A^2_2}{A_1^2} \right)^{\frac{1}{2}-\frac{1}{p}}
+ C(p) R_T^{1+\frac{\beta}{4} + \frac{\beta \eta_0}{2} + \ep} A_1^{ 1+\frac{2}{p} } \\
& \le C (\ep) R_T^{1+\frac{\beta}{4} + \frac{\beta \eta_0}{2} + 2\ep} A_1^{1+\frac{2}{p}} \left( \frac{A^2_2}{A_1^2} \right)^{\frac{1}{2}-\frac{1}{p}}
+ C(\ep) R_T^{1+\frac{\beta}{4} + \frac{\beta \eta_0}{2} + \ep} A_1^2,
\ea\ee
provided $2<p \le \frac{2}{1-2\ep}$.

By virtue of (\ref{1bkj911}), (\ref{1bkj01}), (\ref{1bkj04}), (\ref{1bkj07}), and H\"older's inequality, one has
\be\la{1bkj912}\ba
& \int_0^T \int_{\rr_+} \frac{|u(x )-u(y)|}{|x-y|^2}\rho|u|(y) dy dt \\
& \le C(\ep) R_T^{1+\frac{\beta}{4} + \frac{\beta \eta_0}{2} + 2\ep}
\left( \int_0^T A_1^2 dt \right)^{\frac{1}{2}+\frac{1}{p}}
\left( \int_0^T \frac{A^2_2}{A_1^2} dt \right)^{\frac{1}{2}-\frac{1}{p}}
+ C(\ep) R_T^{1+\frac{\beta}{4} + \frac{\beta \eta_0}{2} +\ep} \\
& \le C(\ep) R_T^{1+\frac{\beta}{4} + \frac{\beta \eta_0}{2} + 3\ep}.
\ea\ee
Moreover, by (\ref{1bkj01}), (\ref{1bkj04}), and (\ref{1bkj07}), we obtain
\be\la{1bkj913}\ba
\int_0^T \left( A_1^{\frac{1}{2}} A^{\frac{1}{2}}_2 + A_1 \right) dt
\le C \left( \int_0^T A_1^2 dt \right)^{\frac{3}{4}}
\left( \int_0^T \frac{A^2_2}{A_1^2} dt \right)^{\frac{1}{4}} + C
\le C R^{\frac{1}{4}+\ep}_T.
\ea\ee
Integrating (\ref{1bkj91}) over $(0,T)$ and applying (\ref{1bkj08}), (\ref{1bkj93}), (\ref{1bkj912}), and (\ref{1bkj913}), we arrive at
\be\la{1bkj914}\ba
R_T^\beta \le C R_T^{ \max\{ \frac{2+\beta}{3}, 1+\frac{\beta}{4} + \frac{\beta \eta_0}{2} + 3 \ep \} },
\ea\ee
which along with $\beta>\frac{4}{3}$ gives
\be\la{1bkj915}\ba
\sup_{0\le t \le T} \| \n \|_{L^\infty} \le C.
\ea\ee
This, combined with (\ref{1bkj01}) and (\ref{1bkj07}), implies (\ref{1bkj09}), thereby completing the proof of Lemma \ref{bkj1l9}.
\end{proof}

\begin{lemma}\la{bkj1l10}
There exists a positive constant $C$ depending only on
$T$, $a$, $\mu$, $\ga$, $\tilde{M}_1$, $E_1$, and $\beta$ such that
\be\ba\la{1bkj010}
\sup_{0\le t\le T}
t \left( \| \sqrt{\n} \dot{u} \|^2_{L^2} + \| \na d_t \|^2_{L^2} \right)
+ \int_0^{T} t \left( \| \na \dot{u} \|^2_{L^2} + \| \na^2 d_t \|^2_{L^2} \right) dt \le C.
\ea\ee
\end{lemma}
\begin{proof}
The proof follows the ideas in \cite{CL,H1,FLL}.
We apply the operator $ \dot{u}^j[\frac{\pa}{\pa t}+\div(u\cdot)]$ to
$(\ref{bkjevf17})^j$, sum with respect to $j$,
and integrate by parts over $\rr_+$ to obtain
\be\la{1bkj101}\ba
\frac{d}{dt}\left(\frac{1}{2}\int\rho|\dot{u}|^2dx \right)
&=\int \bigg( {\dot{u}}\cdot \nabla G_t + {\dot{u}}^j\mathrm {div}(u \partial_j G ) \bigg) dx \\
&\quad + \mu \int \bigg( {\dot{u}} \cdot \na^\bot \o_t
+ {\dot{u}}^j\partial _k( u^k ( \na^\bot \o )^j ) \bigg) dx \\
& \quad - \int \bigg( {\dot{u}} \cdot (\na d \cdot \Delta d)_t
+ {\dot{u}}^j\partial _k( u^k \p_j d \cdot \Delta d ) \bigg) dx \\
& \triangleq I_1+I_2+I_3.
\ea\ee

Integrating by parts and using (\ref{bkjbjds}), (\ref{1bkj04}), (\ref{1bkj06}), (\ref{1bkj09}), and Young's inequality, we obtain
\be\la{1bkj102}\ba
I_1 & = \int_{\partial \rr_+} G_t ( \dot{u} \cdot n) ds
- \int \div \dot{u} \left( \dot{G} - u \cdot \na G \right) dx
- \int u \cdot \na \dot{u}^j \p_j G dx \\
& = - \int \div \dot{u} \dot{G} dx
+ \int \left( u \cdot \na G \div \dot{u} - u \cdot \na \dot{u}^j \p_j G \right) dx \\
& = - \int (2\mu+\lam) (\div \dot{u})^2 dx
+ \int \n \lam'(\n)(\div u)^2 \div \dot{u} dx -\ga \int P \div u \div \dot{u} dx \\
& \quad + \int (2\mu+\lam) \p_i u^j \p_j u^i \div \dot{u} dx
+ \int \left( - \div u  \div \dot{u} G + \p_j u \cdot \na \dot{u}^j G \right) dx \\
& \le - 2 \mu \| \div \dot{u} \|^2_{L^2}
+ C \| \na \dot{u} \|_{L^2} \| \na u \|_{L^4} \left( \| \na u \|_{L^4} + \| P \|_{L^4} + \| G \|_{L^4} \right) \\
& \le - 2 \mu \| \div \dot{u} \|^2_{L^2}
+ \varepsilon \Vert \nabla \dot{u} \Vert_{L^2}^2 + C(\ep) ( 1 + A^2_2 ),
\ea\ee
where in the third equality we have used the following fact:
\be\ba\nonumber
\dot{G} & = G_t + u \cdot \na G \\
& = \lam_t \div u + (2\mu+\lam) \div u_t
+ u \cdot \na ( (2\mu+\lam) \div u ) - P_t - u \cdot \na P \\
& = (\lam_t + u \cdot \na \lam ) \div u
+ (2\mu+\lam) \div \dot{u}
- (2\mu+\lam) \div ( u \cdot \na u ) \\
& \quad + (2\mu+\lam) u \cdot \na \div u + \ga P \div u
 \\
& = -\n \lam'(\n)(\div u)^2 + (2\mu+\lam) \div \dot{u}
- (2\mu+\lam) \p_i u^j \p_j u^i
+ \ga P \div u,
\ea\ee
due to $(\ref{nlckv})_1$.

For $I_2$, integrating over $\rr_+$ and using the boundary condition (\ref{bkjbjtj1}) and Young's inequality, we derive
\be\la{1bkj103}\ba
I_2&=\mu \int \left({\dot{u}}\cdot \nabla ^\bot \omega _t
+{\dot{u}}^j \p_k \left( u^k (\na^\bot \o)^j \right) \right) dx \\
&=-\mu \int \curl \dot{u} \o_t dx-\mu \int u \cdot \na \dot{u} \cdot \na^\bot \o dx \\
&=-\mu \int (\curl \dot{u})^2 dx
+ \mu \int \curl \dot{u} \curl (u \cdot \na u) dx
+ \mu \int \left( u \cdot \na \curl \dot{u} + \na^\bot_j u \cdot \na \dot{u}^j \right) \o dx \\
& \le -\mu \int (\curl \dot{u})^2 dx
+ \mu \int \curl \dot{u} (\na^\bot_j u^i \p_i u^j + u \cdot \na \o) dx \\
& \quad + \mu \int \left( u \cdot \na \curl \dot{u} + \na^\bot_j u \cdot \na \dot{u}^j \right) \o dx \\
& = -\mu \int (\curl \dot{u})^2 dx
+ \mu \int \curl \dot{u} (\na^\bot_j u^i \p_i u^j - \o \div u ) dx
+ \mu \int \na^\bot_j u \cdot \na \dot{u}^j \o dx \\
& \le -\mu \| \curl \dot{u}\|^2_{L^2} + \ep \| \na \dot{u} \|^2_{L^2} + C(\ep) ( 1 + A^2_2 ).
\ea\ee
By virtue of (\ref{1cp39}), (\ref{bkjbjds}), (\ref{1bkj001}), (\ref{1bkj02}), (\ref{1bkj002}), and Young's inequality, it holds that
\be\la{1bkj104}\ba
I_3 & = \int \left( \p_k \dot{u}^j (\p_j d \cdot \p_k d)_t - \frac{1}{2} \div \dot{u} (|\na d|^2)_t + \p_k \dot{u}^j u^k \p_j d \cdot \Delta d \right) dx \\
& \le C \| \na \dot{u} \|_{L^2} \left( \| \na d \|_{L^4} \| \na d_t \|_{L^4} + \| |u| |\na d| |\Delta d| \|_{L^2} \right) \\
& \le \ep \| \na \dot{u} \|^2_{L^2} + C(\ep) \| \na d_t \|^2_{L^4}
+ C(\ep) \| |u| |\na d| |\Delta d| \|^2_{L^2} \\
& \le \ep \| \na \dot{u} \|^2_{L^2} + C(\ep) \| \na d_t \|_{L^2} \| \na^2 d_t \|_{L^2} + C(\ep) \| u \bar{x}^{-\frac{a}{4}} \|^2_{L^8} \| \na d \bar{x}^{\frac{a}{2}} \|_{L^2} \| \na d \|_{L^8} \| \Delta d \|^2_{L^{16}} \\
& \le \ep \| \na \dot{u} \|^2_{L^2} + \frac{1}{4} \| \na^2 d_t \|^2_{L^2} + C(\ep) (1 + A^2_2).
\ea\ee
Substituting (\ref{1bkj102}), (\ref{1bkj103}), and (\ref{1bkj104}) into (\ref{1bkj101}) leads to
\be\la{1bkj105}\ba
& \frac{d}{dt}\left(\frac{1}{2}\int\rho|\dot{u}|^2dx \right)
+ 2 \mu \| \div \dot{u} \|^2_{L^2} + \mu \| \curl \dot{u}\|^2_{L^2} \\
& \le 3 \varepsilon \Vert \nabla {\dot{u}}\Vert _{L^2}^2
+ \frac{1}{4} \| \na^2 d_t \|^2_{L^2} + C(\ep) (1 + A^2_2).
\ea\ee

On the other hand, differentiating $(\ref{nlckv})_3$ with respect to $x$ and $t$, multiplying the resulting equation by $\na d_t$, and following a procedure similar to the derivation of (\ref{1cp74}), we obtain by using (\ref{1bkj16}), (\ref{1bkj02}), (\ref{1bkj002}), and (\ref{1bkj09}) that
\be\la{1bkj106}\ba
& \frac{1}{2}\frac{d}{dt}\|\nabla d_{t}\|_{L^{2}}^{2} + \frac{3}{4} \| \na^2 d_{t}\|_{L^{2}}^{2} \\
& \le \ep \| \na \dot{u} \|^2_{L^2}
+ C(\ep) \left( 1 + \| \nabla d_t \|^2_{L^{2}} \right)
\left( 1 + A^2_2 + \| \nabla^{2}d \bar{x}^{\frac{a}{2}} \|^2_{L^{2}} \right),
\ea\ee
Adding (\ref{1bkj106}) to (\ref{1bkj105}) yields
\be\la{1bkj107}\ba
& \frac{d}{dt} \left( \| \sqrt{\n} \dot{u} \|^2_{L^2} + \| \na d_t \|^2_{L^2} \right)
+ 2\mu \left( \| \div \dot{u} \|^2_{L^2} + \| \curl \dot{u}\|^2_{L^2} \right) + \| \na^2 d_t \|^2_{L^2} \\
& \le 8 \varepsilon \Vert \nabla {\dot{u}}\Vert _{L^2}^2
+ C(\ep) \left( 1 + \| \sqrt{\n} \dot{u} \|^2_{L^2} + \| \nabla d_t \|^2_{L^{2}} \right)
\left( 1 + A^2_2 + \| \nabla^{2}d \bar{x}^{\frac{a}{2}} \|^2_{L^{2}} \right).
\ea\ee
Note that the boundary condition (\ref{bkjbjds}) together with (\ref{dc1}) implies
\be\la{1bkj108}\ba
\| \na \dot{u} \|_{L^2} \le C \left( \| \div \dot{u} \|_{L^2} + \| \curl \dot{u} \|_{L^2} \right).
\ea\ee
Combining this with (\ref{1bkj107}) and choosing $\ep$ sufficiently small, we arrive at
\be\la{1bkj109}\ba
& \frac{d}{dt} \left( \| \sqrt{\n} \dot{u} \|^2_{L^2} + \| \na d_t \|^2_{L^2} \right)
+ \mu \left( \| \div \dot{u} \|^2_{L^2} + \| \curl \dot{u}\|^2_{L^2} \right) + \| \na^2 d_t \|^2_{L^2} \\
& \le C \left( 1 + \| \sqrt{\n} \dot{u} \|^2_{L^2} + \| \nabla d_t \|^2_{L^{2}} \right)
\left( 1 + A^2_2 + \| \nabla^{2}d \bar{x}^{\frac{a}{2}} \|^2_{L^{2}} \right).
\ea\ee

Finally, multiplying (\ref{1bkj109}) by $t$, applying Gr\"onwall's inequality and (\ref{1bkj108}), we get (\ref{1bkj010}).
This completes the proof of Lemma \ref{bkj1l10}.
\end{proof}

Based on (\ref{1bkj02}), (\ref{1bkj09}), and (\ref{1bkj010}), and adapting arguments similar to those in Lemmas \ref{cp1l8} and \ref{cp1l9}, we can obtain the following estimates.
\begin{lemma}\la{bkj1l11}
There exists a positive constant $C$ depending only on
$T$, $a$, $\mu$, $\ga$, $q$, $\tilde{M}_1$, $E_1$,  $\beta$, and $\| \n_0 \|_{W^{1,q}}$ such that
\be\ba\la{1bkj011}
\sup_{0\leq t\leq T} t \left( \|\nabla^{2} d\bar{x}^{\frac{a}{2}}\|_{L^{2}}^{2} + \| \na^3 d \|^2_{L^2} \right)
+ \int_{0}^{T} t\|\nabla^{3} d \bar{x}^{\frac{a}{2}}\|_{L^{2}}^{2} dt
\leq C,
\ea\ee
and
\be\la{1bkj0011}\ba
&\sup_{0\le t\le T} \left( \| \n \|_{W^{1,q}} + t \| \na^2 u \|^2_{L^2} \right) \\
& + \int_0^T \left( \| \na^2 u \|^2_{L^2} + \|\nabla^2 u\|^{(q+1)/q}_{L^q}
+ t \|\nabla^2 u\|_{L^q}^2 + t \| \na^4 d \|^2_{L^2} \right) dt\le C.
\ea\ee
\end{lemma}

\begin{lemma}\la{bkj1l12}
There exists a positive constant $C$ depending only on
$T$, $a$, $\mu$, $\ga$, $q$, $\tilde{M}_1$, $E_1$,  $\beta$, $\| {\bar{x}}^a \rho_0 \|_{H^1}$, and $\| {\bar{x}}^a \rho_0 \|_{W^{1,q}}$ such that
\be\ba\la{1bkj012}
\sup_{0\leq t\leq T} \| \rho \bar{x}^a \|_{L^1 \cap H^{1}\cap W^{1,q}} \leq C.
\ea\ee
\end{lemma}
\begin{proof}
First, by the Gagliardo-Nirenberg inequality, (\ref{1bkj001}), and (\ref{1bkj09}), we have for any $\ve\in(0,1]$ that
\be\la{1bkj121}\ba
\| u \bar{x}^{-\ep} \|_{L^\infty}
& \le C(\ep) \| u \bar{x}^{-\ep} \|_{L^{q/\ep}} + C \| \na( u \bar{x}^{-\ep} ) \|_{L^q} \\
& \le C(\ep) + C(\ep) \| \na u \|_{L^q} + C \| u \bar{x}^{-1} \|_{L^q} \\
& \le C(\ep) + C(\ep) \| \na u \|_{L^q}.
\ea\ee

Set $ v\triangleq\n\bar x^a$.
From $(\ref{nlckv})_1$ we deduce that $v$ satisfies
\bnn\ba
v_t+u\cdot\na v-a vu\cdot\na \log \bar x+v\div u=0,
\ea\enn
which together with (\ref{1bkj121}) and H\"older's inequality shows that for any $p\in [2,q]$,
\be\la{1bkj122}\ba 
(\|\na v\|_{L^p} )_t
& \le C(\|\na u\|_{L^\infty}+\|u\cdot \na \log \bar x\|_{L^\infty}) \|\na v\|_{L^p} \\
&\quad +C\|v\|_{L^\infty}\left( \||\na u||\na\log \bar x|\|_{L^p}+\||  u||\na^2\log \bar x|\|_{L^p}+\| \na^2 u \|_{L^p}\right)\\
& \le C(1 +\|\na u\|_{W^{1,q}})  \|\na v\|_{L^p} \\
& \quad + C \|v\|_{L^\infty} \left( \|\na u\|_{L^p} 
+ \| u \bar x^{-2/5} \|_{L^{4p}} \|\bar x^{-3/2}\|_{L^{4p/3}} + \|\na^2 u\|_{L^p}\right) \\
& \le C(1 +\|\na^2u\|_{L^p}+\|\na u\|_{W^{1,q}})(1+ \|\na v\|_{L^p}+\|\na v\|_{L^q}). \ea\ee
Choosing $p=q$ in (\ref{1bkj122}) and applying (\ref{1bkj0011}) and Gr\"onwall's inequality, we arrive at
\be\la{1bkj123}\ba
\sup\limits_{0\le t\le T}\|\na (\n \bar x^a)\|_{L^q} \le C.
\ea\ee

Furthermore, taking $p=2$ in (\ref{1bkj122}), we obtain after using Gr\"onwall's inequality, (\ref{1bkj0011}), and (\ref{1bkj123}) that
\bnn
\sup\limits_{0\le t\le T}\|\na(\n \bar x^a)\|_{L^2 } \le C.
\enn
Combining this with (\ref{1bkj123}) and (\ref{1bkj01}) gives (\ref{1bkj012}), thus completing the proof of Lemma \ref{bkj1l12}.
\end{proof}

\subsection{The Case of Non-vacuum Far-field Density}

In this subsection, we suppose that the initial data $(\n_0,u_0,d_0)$ satisfy (\ref{bkj2sol1}) and $\n_0 >0$.
Let $(\n,u,d)$ be the corresponding strong solution to (\ref{nlckv})--(\ref{i30}), (\ref{bkjbjtj1}), (\ref{bkjbjtj2}) on $\rr_+ \times (0,T]$ with $\tilde{\n}>0$, and let $E_2$ be the quantity defined in (\ref{e2}), which depends only on the initial data.

We now state the standard energy estimate.

\begin{lemma}\la{bkj2l1}
There exists a positive constant $C$ depending only on $T$, $\mu$, $\ga$, $\tilde{\n}$,
$\| K(\n_0) \|_{L^1}$, $\| \sqrt{\n_0} u_0 \|_{L^2}$, and $\| \na d_0 \|_{L^2}$ such that
\be\la{2bkj01}\ba
& \sup\limits_{0\le t\le T} \int\left( K(\n) + \n |u|^2 + |\na d|^2 \right) dx \\
& + \int_0^T \int\left( \mu |\na u|^2+ \lambda(\n) (\div u)^2 + | \na^2 d |^2 \right) dxdt
\le C,
\ea\ee
and
\be\la{2bkj001}\ba
\| u \|_{H^1} \le C \left( 1 + \| \na u \|_{L^2} \right), \quad t \in [0,T].
\ea\ee
\end{lemma}
\begin{proof}
First, following the argument of Lemma \ref{cp2l1}, we can obtain (\ref{2bkj01}) and
\be\la{2bkj11}\ba
(\n - \tilde{\n})^2 \le C K(\n), \text{ if } \n < 2 \tilde{\n}, \quad
(\n - \tilde{\n})^\ga \le C K(\n), \text{ if } \n \ge 2 \tilde{\n},
\ea\ee
where the constant $C$ depends on $\ga$ and $\tilde{\n}$.

Combining (\ref{2bkj11}) with (\ref{2bkj01}) yields
\be\la{2bkj12}\ba
\| \n - \tilde{\n} \|_{L^2( \n < 2 \tilde{\n} )}
+ \| \n - \tilde{\n} \|_{L^\ga( \n \ge 2 \tilde{\n} )} \le C.
\ea\ee
Next, for any $v \in H^1(\rr_+)$, from (\ref{gn11}), (\ref{2bkj12}), and H\"older's inequality, we deduce that
\be\la{2bkj13}\ba
\int |v|^2 dx & \le C \int \n |v|^2 dx + C \int |\n - \tilde{\n}| |v|^2 dx \\
& \le C \| \sqrt{\n} v \|^2_{L^2} + C \int_{ \left\{\n < 2 \tilde{\n} \right\} } |\n - \tilde{\n}| |v|^2 dx
+ C \int_{ \left\{\n \ge 2 \tilde{\n} \right\} } |\n - \tilde{\n}| |v|^2 dx \\
& \le C \| \sqrt{\n} v \|^2_{L^2} + C \| \n - \tilde{\n} \|_{L^2( \n < 2 \tilde{\n} )} \| v \|^2_{L^4}
+ C \| \n - \tilde{\n} \|_{L^\ga( \n \ge 2 \tilde{\n} )} \| v \|^2_{L^\frac{2\ga}{\ga-1}} \\
& \le C \| \sqrt{\n} v \|^2_{L^2} + C \| v \|_{L^2} \| \na v \|_{L^2}
+ C \| v \|^{\frac{2(\ga-1)}{\ga}}_{L^2} \| \na v \|^{\frac{2}{\ga}}_{L^2} \\
& \le \frac{1}{2} \| v \|^2_{L^2} + C \| \sqrt{\n} v \|^2_{L^2} + C \| \na v \|^2_{L^2},
\ea\ee
which shows
\be\la{2bkj14}\ba
\| v \|_{L^2} \le C \| \sqrt{\n} v \|_{L^2} + C \| \na v \|_{L^2}.
\ea\ee
Taking $v=u$ in (\ref{2bkj14}) and using (\ref{2bkj01}), we obtain (\ref{2bkj001}) and finish the proof of Lemma \ref{bkj2l1}.
\end{proof}

Adapting the approach of Lemma \ref{bkj1l2} and applying (\ref{2bkj01}) and (\ref{2bkj001}), we derive the following estimates.

\begin{lemma}\la{bkj2l2}
For any $2<p<\infty$, there exists a positive constant $C$ depending only on 
$p$, $T$, $\mu$, $\ga$, $\tilde{\n}$,
$\| K(\n_0) \|_{L^1}$, $\| \sqrt{\n_0} u_0 \|_{L^2}$, and $\| \na d_0 \|_{H^1}$ such that
\be\la{2bkj02}\ba
\sup_{0 \le t \le T} \| \na d \|_{L^p} \le C.
\ea\ee
\end{lemma}

\begin{lemma}\la{bkj2l3}
There exists a positive constant $C$ depending only on $T$, $\alpha$, $\mu$, $\ga$, $\beta$, $\tilde{\n}$, and $E_2$ such that
\be\la{2bkj03}\ba
\sup_{0\le t\le T} \left( \| \tilde{x}^{2 \alpha} K(\rho) \|_{L^1}
+ \| \tilde{x}^{\alpha} \sqrt{\n} u \|_{L^2} + \| \tilde{x}^\alpha \na d \|_{L^2} \right) \le C.
\ea\ee
Moreover, for any $2 \le p <\infty$, there exists a positive constant $C$ depending only on $p$, $T$, $\alpha$, $\mu$, $\ga$, $\beta$, $\tilde{\n}$, and $E_2$ such that
\be\la{2bkj003}\ba
\sup_{0\le t\le T} \| \n - \tilde{\n} \|_{L^p} \le C.
\ea\ee
\end{lemma}
\begin{proof}
Let $G$ be the effective viscous flux defined in (\ref{gw}), and let $\tilde{H}_1$, $H_2$, and $H_3$ be as in Lemma \ref{bkj1l3}.
Arguing as in Lemma \ref{bkj1l3}, we have the decomposition
\be\la{2bkj31}\ba
G = \frac{\p}{\p t}\tilde{H}_1+H_2+H_3,
\ea\ee
and $\tilde{H}_1$, $H_2$, and $H_3$ satisfy (\ref{bkjevf03}) and (\ref{bkjevf04}).

Furthermore, we deduce from $(\ref{nlckv})_1$, (\ref{bkjevf17}), (\ref{bkjevf02}) that
\be\la{2bkj32}\ba
\frac{D}{Dt} \left( \Lambda_2(\n) + \tilde{H}_1 \right) + P - P(\tilde{\n}) = u \cdot \na \tilde{H}_1 - H_2 - H_3,
\ea\ee
where
\be\la{2bkj33}\ba
\Lambda_2(\n) = 2\mu (\log \n - \log \tilde{\n}) + \beta^{-1}(\n^\beta-\tilde{\n}^\beta).
\ea\ee
By (\ref{bkjevf03}), (\ref{bkjevf04}), and H\"older's inequality, we have
\be\la{2bkj330}\ba
\| \tilde{H}_1(0) \|_{L^\infty} & \le C \| \tilde{H}_1(0) \|_{W^{1,{\frac{3}{\alpha}}}} \\
& \le C \| \n_0 u_0 \|_{L^{\frac{6}{3+2\alpha}}} + C \| \n_0 u_0 \|_{L^{\frac{3}{\alpha}}} \\
& \le C \| \n_0 \|_{L^\infty}^{\frac{1}{2}} \| \tilde{x}^\alpha \sqrt{\n_0} u_0 \|_{L^2} \| \tilde{x}^{-\alpha} \|_{ L^{\frac{3}{\alpha}} }
+ C \| \n_0 \|_{L^\infty} \| u_0 \|_{H^1} \\
& \le C,
\ea\ee
which together with (\ref{2bkj33}) implies that there exists a positive constant $\tilde{C}_0$ depending only on $\| \n_0 \|_{L^\infty}$, $\| u_0 \|_{H^1}$, $\| \tilde{x}^\alpha \sqrt{\n_0} u_0 \|_{L^2}$, $\alpha$, $\mu$, $\tilde{\n}$, and $\beta$ such that
\be\la{2bkj331}\ba
\|[\Lambda_2(\n_0)]_+\|_{L^\infty} + \| \tilde{H}_1(0) \|_{L^\infty} \le \tilde{C}_0.
\ea\ee
Defining $f \triangleq \max\{ \Lambda_2(\n) + \tilde{H}_1 - \tilde{C}_0,0 \}$, multiplying (\ref{2bkj32}) by $4 m \n f^{4m-1}$ with $m \ge \ga$, and integrating over $\rr_+$, we derive after using (\ref{2bkj12}) that
\be\la{2bkj34}\ba
& \frac{d}{dt} \int \n f^{4m} dx + 4 m \int \n f^{4m-1} |P - P(\tilde{\n})| dx \\
& \le C \int_{ \{ \n < 2\tilde{\n} \} } \n f^{4m-1} |P - P(\tilde{\n})| dx
+ C \int \n f^{4m-1} ( |u \cdot \na \tilde{H}_1| + |H_2| + |H_3| ) dx \\
& \le C \int_{ \{ \n < 2\tilde{\n} \} } \n^{1-1/(4m)} f^{4m-1} |\n - \tilde{\n}|^{1/(2m)} dx \\
& \quad + C \int (1 + |\n-\tilde{\n}|^{1/(4m)} ) \n^{1-1/(4m)} f^{4m-1} ( |u \cdot \na \tilde{H}_1| + |H_2| + |H_3| ) dx \\
& \le C \| \n f^{4m} \|^{1-1/(4m)}_{L^1} \| \n - \tilde{\n} \|^{1/(2m)}_{L^2(\n<2\tilde{\n})}
+ C \| \n f^{4m} \|^{1-1/(4m)}_{L^1} \left( \| u \cdot \na \tilde{H}_1 \|_{L^{4m}} + \| H_2 \|_{L^{4m}} \right) \\
& \quad + C \| \n f^{4m} \|^{1-1/(4m)}_{L^1} \| H_3 \|_{L^{4m}}
+ C \| \n f^{4m} \|^{1-1/(4m)}_{L^1} \| \n -\tilde{\n} \|^{1/(4m)}_{L^{4\beta m+1}} \| u \cdot \na \tilde{H}_1 \|_{L^{(4\beta m+1)/\beta}} \\
& \quad + C \| \n f^{4m} \|^{1-1/(4m)}_{L^1} \| \n -\tilde{\n} \|^{1/(4m)}_{L^{4\beta m+1}}
\left( \| H_2 \|_{L^{(4\beta m+1)/\beta}} + \| H_3 \|_{L^{(4\beta m+1)/\beta}} \right).
\ea\ee

From (\ref{bkjevf03}), (\ref{bkjevf04}), (\ref{2bkj01}), (\ref{2bkj001}), and Young's inequality, we obtain for any $2<s<4\beta m+1$,
\be\la{2bkj35}\ba
\| u \cdot \na \tilde{H}_1 \|_{L^s}
& \le \| \na \tilde{H}_1 \|_{L^{(4\beta m+1+s)/2}}
\| u \|_{L^{s(4\beta m+1+s)/(4\beta m+1-s)}} \\
& \le C \| \n u \|_{L^{(4\beta m+1+s)/2}} \| u \|_{H^1} \\
& \le C \| (\n-\tilde{\n}) u \|_{L^{(4\beta m+1+s)/2}} \| u \|_{H^1}
+ C \| u \|_{L^{(4\beta m+1+s)/2}} \| u \|_{H^1} \\
& \le C \| \n -\tilde{\n} \|_{L^{4 \beta m+1}}
\| u \|_{L^{(4\beta m+1+s)(4\beta m+1)/(4\beta m+1-s)}} \| u \|_{H^1} + C \| u \|^2_{H^1} \\
& \le C \left( 1 + \| \n -\tilde{\n} \|_{L^{4 \beta m+1}} \right) \left( 1 + \| \na u \|^2_{L^2} \right),
\ea\ee
\be\la{2bkj36}\ba
\| H_2 \|_{L^s} \le C \| \n |u|^2 \|_{L^s}
& \le C \| (\n-\tilde{\n}) |u|^2 \|_{L^s} + C \| |u|^2 \|_{L^s} \\
& \le C \| \n-\tilde{\n} \|_{L^{4\beta m+1}}
\| |u|^2 \|_{L^{s(4\beta m+1)/(4\beta m+1-s)}} + C \| u \|^2_{H^1} \\
& \le C \left( 1 + \| \n -\tilde{\n} \|_{L^{4 \beta m+1}} \right) \left( 1 + \| \na u \|^2_{L^2} \right),
\ea\ee
and
\be\la{2bkj37}\ba
\| H_3 \|_{L^s} \le C \left(1+\| \na^2 d \|^2_{L^2}\right).
\ea\ee
Combining (\ref{2bkj34}), (\ref{2bkj35}), (\ref{2bkj36}), and (\ref{2bkj37}) yields
\be\la{2bkj38}\ba
& \frac{d}{dt} \int \n f^{4m} dx + 4 m \int \n f^{4m-1} |P - P(\tilde{\n})| dx \\
& \le C \left( \| \n f^{4m} \|_{L^1} + \| \n -\tilde{\n} \|^{4m+1}_{L^{4 \beta m+1}} + 1 \right) \left( 1 + \| \na u \|^2_{L^2} + \| \na^2 d \|^2_{L^2} \right).
\ea\ee

Next, we multiply $(\ref{nlckv})_1$, $(\ref{nlckv})_2$, and $(\ref{nlckv})_3$ by
$\tilde{x}^{2\alpha} K'(\n)$, $\tilde{x}^{2\alpha} u$,
$-\tilde{x}^{2\alpha}(\Delta d + |\na d|^2 d)$, respectively,
sum the results, and integrate by parts to derive
\be\la{2bkj39}\ba
& \frac{d}{dt} D(t) + \int \left( \mu \o^2 + (2\mu+\lam)(\div u)^2
+ |\Delta d+|\na d|^2d|^2 \right) \tilde{x}^{2\alpha} dx \\
& \le C \int \n |u|^3 |\na \tilde{x}^{2\alpha}| dx
+ C \int |\o| |u| |\na \tilde{x}^{2\alpha}| dx
+ C \int (2 \mu + \lam) |\div u| |u| |\na \tilde{x}^{2\alpha}| dx \\
& \quad + C \int \left( |u| |\na d| + |\na^2 d| + |\na d|^2 \right) |\na d| |\na \tilde{x}^{2\alpha}| dx \\
& \quad + C \int |K(\n) + P - P(\tilde{\n})| |u| |\na \tilde{x}^{2\alpha}| dx \triangleq \sum_{i=1}^{5} I_i,
\ea\ee
where
\be\la{2bkj310}\ba
D(t) \triangleq \frac{1}{2} \int \left( \n |u|^2 + |\na d|^2 + 2 K(\n) \right) \tilde{x}^{2\alpha} dx.
\ea\ee
Now, we will use (\ref{2bkj01}), (\ref{2bkj001}), and Young's inequality to estimate each term on the right-hand side of (\ref{2bkj39}) as follows:
\be\la{2bkj311}\ba
|I_1| & \le C \int \n^{1/2} |u| \tilde{x}^{\alpha} \left( |u|^2 + |\n -\tilde{\n}|^{1/2} |u|^2 \right) dx \\
& \le C \| \sqrt{\n} u \tilde{x}^{\alpha} \|_{L^2}
\left( \| u \|^2_{L^4} + \|\n -\tilde{\n}\|^{1/2}_{L^{4\beta m+1}} \| u \|^2_{L^{(4\beta m+1)/(\beta m)}} \right) \\
& \le C \left( \| \sqrt{\n} u \tilde{x}^{\alpha} \|^2_{L^2}
+ \|\n -\tilde{\n}\|_{L^{4\beta m+1}} + 1 \right) \left( \| \na u \|^2_{L^2} + 1 \right),
\ea\ee
\be\la{2bkj312}\ba
I_2 \le C \| \o \tilde{x}^{\alpha} \|_{L^2} \| u \|_{L^2}
\le \frac{\mu}{4} \| \o \tilde{x}^{\alpha} \|^2_{L^2} + C \left( \| \na u \|^2_{L^2} + 1 \right),
\ea\ee
\be\la{2bkj313}\ba
I_3 & \le \frac{1}{4} \int (2\mu+\lam)(\div u)^2 \tilde{x}^{2\alpha} dx
+ C \int \left( |\n - \tilde{\n}|^\beta + 1 \right) |u|^2 dx \\
& \le \frac{1}{4} \int (2\mu+\lam)(\div u)^2 \tilde{x}^{2\alpha} dx + C \| u \|^2_{L^2} + C \|\n -\tilde{\n}\|^\beta_{L^{4\beta m+1}}
\| u \|^2_{L^{ 2(4\beta m+1)/(4\beta m+1-\beta) }} \\
& \le \frac{1}{4} \int (2\mu+\lam)(\div u)^2 \tilde{x}^{2\alpha} dx
+ C \left( \|\n -\tilde{\n}\|^\beta_{L^{4\beta m+1}} + 1 \right) \left( \| \na u \|^2_{L^2} + 1 \right),
\ea\ee
\be\la{2bkj314}\ba
I_4 & \le C \| \na d \tilde{x}^{\alpha} \|_{L^2}
\left( \| u \|_{L^4} \| \na d \|_{L^4} + \| \na^2 d \|_{L^2} + \| \na d \|^2_{L^4} \right) \\
& \le C \left( \| \na d \tilde{x}^{\alpha} \|^2_{L^2} + 1 \right)
\left( \| \na u \|^2_{L^2} + \| \na^2 d \|^2_{L^2} + 1 \right),
\ea\ee
and
\be\la{2bkj315}\ba
I_5 & \le C \int_{ \{\n < 2\tilde{\n}\} } |\n - \tilde{\n}| |u| \tilde{x}^{\alpha} dx
+ C \int_{ \{\n \ge 2\tilde{\n}\} } |\n - \tilde{\n}|^{\ga} |u| \tilde{x}^{\alpha} dx \\
& \le C \int_{ \{\n < 2\tilde{\n}\} } K(\n)^{1/2} |u| \tilde{x}^{\alpha} dx
+ C \int_{ \{\n \ge 2\tilde{\n}\} } K(\n)^{1/2} |\n - \tilde{\n}|^{m/2} |u| \tilde{x}^{\alpha} dx \\
& \le C \| K(\n) \tilde{x}^{2\alpha} \|^{1/2}_{L^1}
\left( \| u \|_{L^2} + C \|\n -\tilde{\n}\|^{m/2}_{L^{4\beta m+1}}
\| u \|_{ L^{2(4\beta m+1)/(4\beta m +1-m)} } \right) \\
& \le C \left( \| K(\n) \tilde{x}^{2\alpha} \|_{L^1} + \|\n -\tilde{\n}\|^{m}_{L^{4\beta m+1}} + 1 \right)
\left( \| \na u \|^2_{L^2} + 1 \right).
\ea\ee
Substituting (\ref{2bkj311})--(\ref{2bkj315}) into (\ref{2bkj39}) leads to
\be\la{2bkj316}\ba
& \frac{d}{dt} D(t) + \frac{1}{2} \int \left( \mu \o^2 + (2\mu+\lam)(\div u)^2
+ |\Delta d+|\na d|^2d|^2 \right) \tilde{x}^{2\alpha} dx \\
& \le C \left( D(t) + \|\n -\tilde{\n} \|^{4\beta m+1}_{L^{4\beta m+1}} + 1 \right)
\left( \| \na u \|^2_{L^2} + \| \na^2 d \|^2_{L^2} + 1 \right).
\ea\ee
Adding (\ref{2bkj316}) to (\ref{2bkj38}) yields
\be\la{2bkj317}\ba
& \frac{d}{dt} \left( \| \n f^{4m} \|_{L^1} + D(t) \right) + \frac{1}{2} \int \left( \mu \o^2 + (2\mu+\lam)(\div u)^2
+ |\Delta d+|\na d|^2d|^2 \right) \tilde{x}^{2\alpha} dx \\
& \le C \left( \| \n f^{4m} \|_{L^1} + D(t) + \|\n -\tilde{\n} \|^{4\beta m+1}_{L^{4\beta m+1}} + 1 \right)
\left( \| \na u \|^2_{L^2} + \| \na^2 d \|^2_{L^2} + 1 \right).
\ea\ee

To estimate $\|\n -\tilde{\n} \|_{L^{4\beta m+1}}$,
decompose $\tilde{H}_1$ as $\tilde{H}_1 = \tilde{H}_{11} + \tilde{H}_{12}$, where $\tilde{H}_{11}$ and $\tilde{H}_{12}$ solve
\be\la{2bkj318}\ba
\begin{cases}
\Delta{\tilde{H}_{11}} = \div (\sqrt{\n} u (\sqrt{\n}-\sqrt{\tilde{\n}}) ) \ \   & \text{ in } \rr_+, \\
\frac{\p \tilde{H}_{11}}{\p n}=0 \ \  & \text{ on } \p \rr_+, \\
\tilde{H}_{11} \to 0 & \text{ as } \ |x| \to \infty,
\end{cases}
\ea\ee
and
\be\la{2bkj319}\ba
\begin{cases}
\Delta{\tilde{H}_{12}} = \sqrt{\tilde{\n}} \div (\sqrt{\n} u) \ \   & \text{ in } \rr_+, \\
\frac{\p \tilde{H}_{12}}{\p n}=0 \ \  & \text{ on } \p \rr_+, \\
\tilde{H}_{12} \to 0 & \text{ as } \ |x| \to \infty.
\end{cases}
\ea\ee
Moreover, (\ref{bkjevf1}) and (\ref{bkjevf6}) imply that for any $2<r<\infty$,
\be\la{2bkj320}\ba
\| \tilde{H}_{11} \|_{L^r} \le C \| \sqrt{\n} u (\sqrt{\n}-\sqrt{\tilde{\n}}) \|_{L^{\frac{2r}{2+r}}}, \quad
\| \tilde{H}_{12} \|_{L^r} \le C \| \sqrt{\n} u \|_{L^{\frac{2r}{2+r}}}.
\ea\ee
By virtue of (\ref{2bkj12}) and the definition of $f$, we have
\be\la{2bkj321}\ba
\int |\n -\tilde{\n}|^{4\beta m+1} dx
& \le C \int_{ \{\n<2\tilde{\n} \} } |\n - \tilde{\n}|^2 dx
+ C \int_{ \{\n \ge 2\tilde{\n} \} } |\n - \tilde{\n}|^{4\beta m+1} dx \\
& \le C + C \int_{ \{\n \ge 2\tilde{\n} \} } |\n - \tilde{\n}|
\left( f^{4m} + |\tilde{H}_1|^{4m} + \tilde{C}_0^{4m} \right) dx \\
& \le C + C \| \n f^{4m} \|_{L^1}
+ C \int_{ \{\n \ge 2\tilde{\n} \} } |\n - \tilde{\n}| \left( |\tilde{H}_{11}|^{4m} + |\tilde{H}_{12}|^{4m}\right) dx.
\ea\ee
In view of (\ref{2bkj01}), (\ref{2bkj12}), and (\ref{2bkj320}), it holds that
\be\la{2bkj322}\ba
\int_{ \{\n \ge 2\tilde{\n} \} } |\n - \tilde{\n}| |\tilde{H}_{11}|^{4m} dx
& \le C \| \n -\tilde{\n} \|_{L^{4\beta m+1}} \| \tilde{H}_{11} \|^{4m}_{L^{(4\beta m+1)/\beta}} \\
& \le C \| \n -\tilde{\n} \|_{L^{4\beta m+1}} \| \sqrt{\n} u (\sqrt{\n}-\sqrt{\tilde{\n}}) \|^{4m}_{L^{2(4\beta m+1)/(4\beta m+1+2\beta)}} \\
& \le C \| \n -\tilde{\n} \|_{L^{4\beta m+1}} \| \sqrt{\n} u \|^{4m}_{L^2}
\| \sqrt{\n} - \sqrt{\tilde{\n}} \|^{4m}_{L^{(4\beta m+1)/\beta}} \\
& \le C \| \n -\tilde{\n} \|_{L^{4\beta m+1}}
\| \n - \tilde{\n} \|^{2m}_{L^{(4\beta m+1)/(2\beta)}} \\
& \le C \| \n -\tilde{\n} \|_{L^{4\beta m+1}}
\left( \| \n - \tilde{\n} \|^{2m}_{L^2(\n<2\tilde{\n})}
+ \| \n - \tilde{\n} \|^{2m}_{L^\ga(\n \ge 2\tilde{\n})}
+ \| \n - \tilde{\n} \|^{2m}_{L^{4\beta m+1}} \right) \\
& \le C \| \n -\tilde{\n} \|_{L^{4\beta m+1}}
\left( 1 + \| \n - \tilde{\n} \|^{2m}_{L^{4\beta m+1}} \right) \\
& \le \ep \| \n - \tilde{\n} \|^{4\beta m+1}_{L^{4\beta m+1}} + C(\ep).
\ea\ee
For $p = \frac{2\beta}{4\beta m+1} \in (0,1)$ and $s = \frac{1}{\alpha m p} \in (2,\infty)$,
we use (\ref{2bkj01}), (\ref{2bkj12}), and Young's inequality to derive
\be\la{2bkj323}\ba
& \int_{ \{\n \ge 2\tilde{\n} \} } |\n - \tilde{\n}| |\tilde{H}_{12}|^{4m} dx \\
& \le C \| \n -\tilde{\n} \|_{L^{\frac{s}{s-1}} (\n \ge 2\tilde{\n}) } \| \tilde{H}_{12} \|^{4m}_{L^{4m s}} \\
& \le C \left( \| \n -\tilde{\n} \|_{L^{4\beta m+1} (\n \ge 2\tilde{\n})}
+ \| \n -\tilde{\n} \|_{L^1 (\n \ge 2\tilde{\n})} \right)
\left( \| \sqrt{\n} u \|^{1-p}_{L^2} \| \sqrt{\n} u \tilde{x}^{\alpha} \|^{p}_{L^2}
\| \tilde{x}^{-\alpha p} \|_{L^{4m s}} \right)^{4m} \\
& \le C \left( \| \n -\tilde{\n} \|_{L^{4\beta m+1}}
+ \| \n -\tilde{\n} \|^\ga_{L^\ga (\n \ge 2\tilde{\n})} \right)
\| \sqrt{\n} u \tilde{x}^{\alpha} \|^{4m p}_{L^2} \\
& \le C \left( \| \n -\tilde{\n} \|_{L^{4\beta m+1}} + 1 \right)
\| \sqrt{\n} u \tilde{x}^{\alpha} \|^{4m p}_{L^2} \\
& \le \ep \| \n - \tilde{\n} \|^{4\beta m+1}_{L^{4\beta m+1}}
+ C (\ep) \| \sqrt{\n} u \tilde{x}^{\alpha} \|^2_{L^2} + C (\ep).
\ea\ee
Combining (\ref{2bkj321}), (\ref{2bkj322}), and (\ref{2bkj323}), we arrive at
\be\la{2bkj324}\ba
\int |\n -\tilde{\n}|^{4\beta m+1} dx
& \le C \| \n f^{4m} \|_{L^1} + C \| \sqrt{\n} u \tilde{x}^{\alpha} \|^2_{L^2} + C.
\ea\ee
Putting (\ref{2bkj324}) into (\ref{2bkj317}) yields
\be\la{2bkj325}\ba
& \frac{d}{dt} \left( \| \n f^{4m} \|_{L^1} + D(t) \right) + \frac{1}{2} \int \left( \mu \o^2 + (2\mu+\lam)(\div u)^2 + |\Delta d + |\na d|^2 d |^2 \right) \tilde{x}^{2\alpha} dx \\
& \le C \left( \| \n f^{4m} \|_{L^1} + D(t) + 1 \right)
\left( \| \na u \|^2_{L^2} + \| \na^2 d \|^2_{L^2} + 1 \right).
\ea\ee
Applying Gr\"onwall's inequality and (\ref{2bkj01}), we get
\be\la{2bkj326}\ba
\sup_{0\le t\le T} \int \left( \n f^{4m} + \left( \n |u|^2 + |\na d|^2 + 2 K(\n) \right) \tilde{x}^{2\alpha} \right) dx \le C,
\ea\ee
which leads to (\ref{2bkj03}).

Finally, from (\ref{2bkj12}) and (\ref{2bkj324}), we deduce that
\be\la{2bkj327}\ba
\int \left( |\n - \tilde{\n}|^2 + |\n - \tilde{\n}|^{4\beta m+1 }\right) dx
& \le C \int_{ \{\n < 2\tilde{\n}\} } |\n - \tilde{\n}|^2 dx
+ C \int |\n - \tilde{\n}|^{4\beta m+1 } dx \\
& \le C \| \n f^{4m} \|_{L^1} + C \| \sqrt{\n} u \tilde{x}^{\alpha} \|^2_{L^2} + C,
\ea\ee
which together with (\ref{2bkj326}) gives
\be\nonumber\ba
\sup_{0\le t\le T} \int \left( |\n - \tilde{\n}|^2 + |\n - \tilde{\n}|^{4 \beta m + 1 }\right) dx \le C.
\ea\ee
This, combined with H\"older's inequality, implies (\ref{2bkj003}) and completes the proof of Lemma \ref{bkj2l3}.
\end{proof}

\begin{lemma}\la{bkj2l4}
For any $2 < p<\infty$ and $\ep \in (0,1)$, there exists a positive constant $C$ depending only on $p$, $\ep$, $T$, $\alpha$, $\mu$, $\ga$, $\beta$, $\tilde{\n}$, and $E_2$ such that
\be\la{2bkj04}\ba
\| \nabla u \|_{L^{p}}
& \le C R^{\frac{1}{2}-\frac{1}{p}+\ep}_T A_1 \left( \frac{A^2_2}{A_1^2} \right)^{\frac{1}{2}-\frac{1}{p}} + C R^\ep_T A_1,
\ea\ee
where $A_1$, $A_2$, and $R_T$ are defined in \eqref{a1}, \eqref{a2}, and \eqref{mdsj}, respectively.
\end{lemma}
\begin{proof}
First, from the boundary conditions (\ref{bkjbjtj1}), we deduce that $G$ and $\o$ satisfy the following elliptic systems:
\be\la{2bkj41}\ba
\begin{cases}
\Delta G = \div \left( \rho \dot{u} + \na d \cdot \Delta d \right) & \mathrm{in}\, \,  \rr_+, \\
\frac {\p G}{\p n}= \left( \rho \dot{u} + \na d \cdot \Delta d \right) \cdot n & \mathrm{on}\, \,  \p \rr_+,
\end{cases}
\ea\ee
and
\be\la{2bkj42}\ba
\begin{cases}
\mu \Delta \o =\na^\bot \cdot \left( \rho \dot{u} + \na d \cdot \Delta d \right) & \mathrm{in}\, \,  \rr_+, \\
\o = 0 &\mathrm{on}\, \,  \p \rr_+.
\end{cases}
\ea\ee

By the standard $L^p$ estimate for elliptic equations (see \cite{JK}) and (\ref{2bkj001}), we obtain for any $2\le p<\infty$,
\be\la{2bkj43}\ba
\| \na G \|_{L^p} + \| \na \o \|_{L^p} \le C \left( \| \n \dot{u}\|_{L^p}
+ \| \na d \cdot \Delta d \|_{L^p} \right).
\ea\ee
Note that (\ref{2bkj003}) implies, for any $s \in [2,4]$,
\be\la{2bkj44}\ba
\int |P - P(\tilde{\n})|^s dx \le C \int_{ \{ \n<2\tilde{\n} \} } |\n -\tilde{\n}|^2 dx
+ C \int_{ \{ \n \ge 2\tilde{\n} \} } |\n -\tilde{\n}|^{4\ga} dx \le C,
\ea\ee
which together with (\ref{dc1}) yields
\be\la{2bkj45}\ba
\| \na u \|_{L^2} & \le C \left( \| \div u \|_{L^2} + \| \o \|_{L^2} \right) \\
& \le C \left\| \frac{G}{2\mu+\lam} \right\|_{L^2}
+ C \left\| \frac{P-P(\tilde{\n})}{2\mu+\lam} \right\|_{L^2} + C \| \o \|_{L^2}
\le C A_1.
\ea\ee
Combining (\ref{2bkj43}), (\ref{2bkj45}), and (\ref{2bkj02}) leads to
\be\la{2bkj46}\ba
\| \na G \|_{L^2} + \| \na \o \|_{L^2}
& \le C \left( \| \rho \dot{u} \|_{L^2} + \| \na d \cdot \Delta d \|_{L^2} \right) \\
& \le C \left( R^{1/2}_T A_2 + \| \na d \|_{L^4} \| \Delta d \|_{L^4}
+ A_1 \right) \\
& \le C R^{1/2}_T A_2 + C A_1.
\ea\ee
Moreover, a direct computation gives
\be\la{2bkj47}\ba
\| G \|^2_{L^2} \le C R^\beta_T A_1^2, \quad
\left\| \frac{G}{2\mu+\lam} \right\|^2_{L^2} \le C A_1^2.
\ea\ee
For the given $p$ and $\ep$, choose
$0<\delta<\min\{2/p,2\ep/\beta\}$. By virtue of (\ref{gn11}),
(\ref{dc1}), and (\ref{2bkj47}), and following the same interpolation
argument as in (\ref{1bkj67}), we obtain
\be\la{2bkj48}\ba
\| \na u \|_{L^p}
& \le C R^{\frac{\beta \delta}{2}}_T A_1^{\frac{2}{p}}
\left( \| \na G \|_{L^2} + \| \na \o \|_{L^2} \right)^{1-\frac{2}{p}} + C A_1 \\
& \le C R_T^{\frac{1}{2}-\frac{1}{p}+\frac{\beta\delta}{2}} A_1
\left(\frac{A_2^2}{A_1^2}\right)^{\frac{1}{2}-\frac{1}{p}}
+ C R_T^{\frac{\beta\delta}{2}}A_1 \\
& \le C R_T^{\frac{1}{2}-\frac{1}{p}+\ep} A_1
\left(\frac{A_2^2}{A_1^2}\right)^{\frac{1}{2}-\frac{1}{p}}
+ C R_T^\ep A_1,
\ea\ee
where in the last two inequalities we used (\ref{2bkj46}), and $\beta\delta/2<\ep$. This proves (\ref{2bkj04})
and completes the proof of Lemma \ref{bkj2l4}.
\end{proof}

Building on (\ref{2bkj01}), (\ref{2bkj001}), (\ref{2bkj02}), (\ref{2bkj04}), and (\ref{2bkj44}), and arguing as in Lemma \ref{bkj1l7}, we obtain the following estimate.

\begin{lemma}\la{bkj2l5}
For any $\ep \in (0,1)$, there exists a positive constant $C$ depending only on $\ep$, $T$, $\alpha$, $\mu$, $\ga$, $\beta$, $\tilde{\n}$, and $E_2$ such that
\be\la{2bkj05}\ba
\sup_{0 \le t \le T} \log A_1^2 + \int_0^T \frac{A^2_2}{A_1^2} dt
\le C R^{1+\ep}_T.
\ea\ee
\end{lemma}

\begin{lemma}\la{bkj2l6}
There exists a positive constant $C$ depending only on $T$, $\alpha$, $\mu$, $\ga$, $\beta$, $\tilde{\n}$, and $E_2$ such that
\be\ba\la{2bkj06}
& \sup_{0\leq t\leq T} \left( \|\n\|_{L^\infty} + \| u \|_{H^1} + \| \na d \|_{H^1} \right) \\
& + \int_0^T \left( \| \na u \|^2_{L^2} + \| \sqrt{\n} \dot{u} \|^2_{L^2} + \| \na^2 d \|^2_{H^1} + \| \na d_t \|^2_{L^2} \right) dt
\le C.
\ea\ee
\end{lemma}
\begin{proof}
First, using (\ref{gw}), we rewrite $(\ref{nlckv})_1$ as
\be\la{2bkj61}\ba
\frac{D}{Dt} \Lambda_2(\n) + P - P(\tilde{\n}) = -G,
\ea\ee
where $\Lambda_2(\n) = 2\mu (\log \n - \log \tilde{\n}) + \beta^{-1} (\n^\beta - \tilde{\n}^\beta)$.

Using the Green's function (\ref{bkjglhs}) and the fact that $G$ satisfies (\ref{2bkj41}), we arrive at
\be\la{2bkj62}\ba
G(x,t) = & \int_{\rr_+} \na_y N(x,y) \cdot \left( \rho \dot{u} + \na d \cdot \Delta d \right) dy.
\ea\ee
As in (\ref{1bkj82}), we have
\be\la{2bkj63}\ba
\int_{\rr_+} \na_y N(x,y) \cdot \n \dot{u} dy
= ( \p_t + u \cdot \na ) \int_{\rr_+} \na_y N(x,y) \cdot \n u(y) dy - J,
\ea\ee
where $J$ is defined in (\ref{1bkj83}) and satisfies
\be\la{2bkj64}\ba
|J| \le C \sup_{x \in \ol{\rr_+} } \int_{\rr_+} \frac{ |u(x)-u(y)| }{ |x-y|^2 } \n |u| dy.
\ea\ee

Moreover, using (\ref{2bkj01}), (\ref{2bkj001}), (\ref{2bkj02}), and (\ref{2bkj44}), and adapting the argument of Lemma \ref{bkj1l5}, there exists a $\nu_2 \in (0,1)$ such that
\be\la{2bkj65}\ba
\sup_{0\le t\le T}\int \n |u|^{2+\nu} dx \le C,
\ea\ee
with $\nu \triangleq R_T^{-\frac{\beta}{2}} \nu_2$.

It thus follows from (\ref{2bkj65}), (\ref{2bkj03}), and H\"older's inequality that
\be\la{2bkj66}\ba
& \left| \int_{\rr_+} \na_y N(x,y) \cdot \n u(y) dy \right| \\
& \le C \int_{\rr_+} |x-y|^{-1} \n |u(y)| dy \\
& \le C \int_{ |x-y| \le 1 } |x-y|^{-1} \n |u(y)| dy
+ C \int_{ |x-y| > 1 } |x-y|^{-1} \n |u(y)| dy \\
& \le C \left( \int_{ |x-y| \le 1 } |x-y|^{-\frac{2+\nu}{1+\nu}} dy \right)^{\frac{1+\nu}{2+\nu}}
\left(\int_{\rr_+} \n^{2+\nu} |u|^{2+\nu} dy \right)^{\frac{1}{2+\nu}} \\
& \quad + C R^{\frac{1}{2}}_T \left( \int_{ |x-y| > 1 } |x-y|^{-\frac{4}{2-\alpha}} dy \right)^{\frac{2-\alpha}{4}} \| \sqrt{\n} u \tilde{x}^\alpha \|_{L^2} \| \tilde{x}^{-\alpha} \|_{L^\frac{4}{\alpha}} \\
& \le C \nu^{-\frac{1+\nu}{2+\nu}} R_T^{\frac{1+\nu}{2+\nu}} + C R^{\frac{1}{2}}_T \\
& \le C R_T^{\frac{2+\beta}{3}}.
\ea\ee

Arguing as in (\ref{1bkj94}), (\ref{1bkj96}), (\ref{1bkj97}), and using (\ref{2bkj001}) and (\ref{2bkj04}), one obtains
\be\ba\la{2bkj67}
& \int_{\rr_+} \frac{|u(x)-u(y)|}{|x-y|^2}\rho|u|(y) dy \\
& \le C(p) R_T^{1+\frac{\beta}{4}} A_1^{ \frac{2}{p} } \|\nabla u\|_{L^p} \\
& \le C(p) R_T^{\frac{3}{2} + \frac{\beta}{4} - \frac{1}{p} + \ep} A_1^{1+\frac{2}{p}} \left( \frac{A^2_2}{A_1^2} \right)^{\frac{1}{2}-\frac{1}{p}}
+ C(p) R_T^{1+\frac{\beta}{4} + \ep} A_1^{ 1+\frac{2}{p} } \\
& \le C (\ep) R_T^{1+\frac{\beta}{4} + 2\ep} A_1^{1+\frac{2}{p}} \left( \frac{A^2_2}{A_1^2} \right)^{\frac{1}{2}-\frac{1}{p}}
+ C(\ep) R_T^{1+\frac{\beta}{4} + \ep} A_1^2,
\ea\ee
which together with (\ref{2bkj64}) yields
\be\la{2bkj68}\ba
\int_0^T |J(t)| dt
& \le C(\ep) R_T^{1+\frac{\beta}{4} + 2\ep}
\left( \int_0^T A_1^2 dt \right)^{\frac{1}{2}+\frac{1}{p}}
\left( \int_0^T \frac{A^2_2}{A_1^2} dt \right)^{\frac{1}{2}-\frac{1}{p}}
+ C(\ep) R_T^{1+\frac{\beta}{4} + \ep} \\
& \le C(\ep) R_T^{1+\frac{\beta}{4} + 3 \ep}.
\ea\ee
By a calculation similar to that in (\ref{1bkj811}), one gets
\be\la{2bkj69}\ba
\left| \int_{\rr_+} \nabla_y N(x,y) \cdot \na d \cdot \Delta d dy \right|
& \le C \left( \| \na d \|_{L^{12}} \| \na^2 d \|_{L^4} + \| \na d \|_{L^{\frac{12}{5}}} \| \na^2 d \|_{L^4} \right) \\
& \le C A_1 + C A_1^{\frac{1}{2}} A^{\frac{1}{2}}_2.
\ea\ee
In view of (\ref{2bkj05}) and H\"older's inequality, it holds that
\be\la{2bkj611}\ba
\int_0^T \left( A_1^{\frac{1}{2}} A^{\frac{1}{2}}_2 + A_1 \right) dt
\le C \left( \int_0^T A_1^2 dt \right)^{\frac{3}{4}}
\left( \int_0^T \frac{A^2_2}{A_1^2} dt \right)^{\frac{1}{4}} + C
\le C R^{\frac{1}{4}+\ep}_T.
\ea\ee
Thus, integrating (\ref{2bkj61}) over $(0,T)$, and using (\ref{2bkj62}), (\ref{2bkj63}), (\ref{2bkj66}), (\ref{2bkj68}), (\ref{2bkj69}), and (\ref{2bkj611}), we arrive at
\be\nonumber\ba
R_T^\beta \le C R_T^{ \max\{ \frac{2+\beta}{3}, 1+\frac{\beta}{4} + 3 \ep \} }.
\ea\ee
Since $\beta>\frac{4}{3}$, this implies
\be\la{2bkj612}\ba
\sup_{0\le t \le T} \| \n \|_{L^\infty} \le C.
\ea\ee
Combining this with (\ref{2bkj01}), (\ref{2bkj001}), (\ref{2bkj45}), and (\ref{2bkj05}) yields (\ref{2bkj06}), which completes the proof of Lemma \ref{bkj2l6}.
\end{proof}

Finally, from (\ref{2bkj001}) and (\ref{2bkj06}), we can establish the following higher-order estimates.
The proof is standard and is therefore omitted.

\begin{lemma}\la{bkj2l7}
There exists a positive constant $C$ depending only on $T$, $\alpha$, $\mu$, $\ga$, $\beta$, $\tilde{\n}$, $q$, $\| \na \n_0 \|_{L^2 \cap L^q }$, and $E_2$ such that
\be\ba\la{2bkj07}
\sup_{0\le t\le T}
t \left( \| \sqrt{\n} \dot{u} \|^2_{L^2} + \| \na d_t \|^2_{L^2} + \| \na^3 d \|^2_{L^2} \right)
+ \int_0^{T} t \left( \| \na \dot{u} \|^2_{L^2} + \| \na^2 d_t \|^2_{L^2} \right) dt \le C,
\ea\ee
and
\be\la{2bkj007}\ba
&\sup_{0\le t\le T} \left( \| \n-\tilde{\n} \|_{H^1 \cap W^{1,q}} + t \| \na^2 u \|^2_{L^2} \right) \\
& + \int_0^T \left( \| \na^2 u \|^2_{L^2} + \|\nabla^2 u\|^{(q+1)/q}_{L^q}
+ t \|\nabla^2 u\|_{L^q}^2 + t \| \na^4 d \|^2_{L^2} \right) dt \le C.
\ea\ee
\end{lemma}

\section{Proofs of Theorems \ref{thcp1}--\ref{thbkj2}}
With the a priori estimates established in Sections 3 and 4 at hand, we are now ready to prove Theorems \ref{thcp1}--\ref{thbkj2}.

\noindent\textbf{Proof of Theorem \ref{thcp1}}.
Let $(\n_0,u_0,d_0)$ be the initial data satisfying (\ref{cpsol1}).
By the definition of $\tilde M_0$ in (\ref{cpsol1}), we have
\be\la{pfth1}\ba
\int_{\rr} \n_0 dx = \tilde M_0>0.
\ea\ee
Hence, there exists a positive constant $\hat{N}_0$ such that
\be\la{pfth2}\ba
\int_{B_{\hat{N}_0}} \n_0 dx
\ge \frac{3}{4} \int_{\rr} \n_0 dx
= \frac{3}{4}\tilde M_0.
\ea\ee
For sufficiently small $\ep>0$, we construct the approximation
$\n_0^{\ep} = \hat{\n}_0^{\ep} + \ep e^{-|x|^2}$, where
$0 \le \hat{\n}_0^{\ep} \in C_0^\infty(\rr)$ is chosen such that
\be\la{pfth3}\ba
\frac{1}{2}\tilde M_0
\le \int_{B_{\hat{N}_0}} \hat{\n}_0^{\ep} dx
\le \int_{\rr} \hat{\n}_0^{\ep} dx
\le \frac{3}{2}\tilde M_0,
\ea\ee
and
\be\la{pfth4}\ba
\lim_{\ep \to 0} \| \bar{x}^a \hat{\n}_0^{\ep} - \bar{x}^a \n_0 \|_{L^1 \cap H^1 \cap W^{1,q}} = 0.
\ea\ee
Then, we consider the unique smooth solution $u_0^\ep$ of the following elliptic equation:
\be\la{pfth5}\ba
\begin{cases}
- \Delta u_0^\ep + \n_0^\ep u_0^\ep = - \Delta(u_0 * j_\ep) + \sqrt{\n_0^\ep} \left( (\sqrt{\n_0} u_0) * j_\ep \right) \quad \text{ in } \rr, \\
u_0^\ep \to 0 \quad \text{ as } |x| \to \infty,
\end{cases}
\ea\ee
where $j_\ep$ is the standard mollifying kernel of width $\ep$.
Then standard arguments yield
\be\la{pfth6}\ba
\lim_{\ep \to 0} \left( \left\| \sqrt{\n_0^\ep} u_0^\ep - \sqrt{\n_0} u_0 \right\|_{L^2} + \| \na u_0^\ep - \na u_0 \|_{L^2} \right) = 0.
\ea\ee
Moreover, since $d_0 \in H^2_{\textnormal{loc}}(\rr)$ and $|d_0| = 1$, we deduce from Lemma \ref{jzb} that there exists $\theta_0$ such that
\be\la{pfth7}\ba
d_0(x) = (\cos\theta_0(x),\sin\theta_0(x)).
\ea\ee
Let
\be\la{pfth8}\ba
\theta_0^\ep \triangleq \theta_0 * j_\ep, \quad d_0^\ep \triangleq (\cos\theta_0^\ep,\sin\theta_0^\ep).
\ea\ee
A straightforward calculation gives
\be\la{pfth9}\ba
\lim_{\ep \to 0} \left( \| \na d_0^\ep - \na d_0 \|_{H^1} + \left\| \bar{x}^{\frac{a}{2}} \na d_0^\ep - \bar{x}^{\frac{a}{2}} \na d_0 \right\|_{L^2} \right) = 0.
\ea\ee

It follows from (\ref{pfth1})--(\ref{pfth9}) that $(\n_0^\ep,u_0^\ep,d_0^\ep)$ satisfies (\ref{cpsol1}) and $\n_0^\ep > 0$.
Hence, the local existence result in Lemma \ref{lct} shows that the problem (\ref{nlckv})--(\ref{bjtj2}) with $\tilde{\n}=0$ and the initial data $(\n_0^\ep,u_0^\ep,d_0^\ep)$ has a unique local strong solution $(\n^\ep,u^\ep,d^\ep)$ on $\rr \times (0,T^\ep]$ for some $T^\ep>0$.
Furthermore, the a priori estimates established in Lemmas \ref{cp1l6}--\ref{cp1l10} imply that this solution $(\n^\ep,u^\ep,d^\ep)$ can be extended to $\rr \times (0,T]$ for any $T>0$ and satisfies all the estimates in Lemmas \ref{cp1l6}--\ref{cp1l10} with $C$ independent of $\ep$.
Letting $\ep \to 0$ and using standard compactness arguments as in \cite{HL2,L2,P,VK}, we conclude that the problem (\ref{nlckv})--(\ref{bjtj2}) with $\tilde{\n}=0$ admits a global strong solution $(\n,u,d)$ satisfying the properties listed in Theorem \ref{thcp1}.
Moreover, the uniqueness argument for strong solutions satisfying (\ref{cpsol3}) and (\ref{cpsol4}) is similar to that in \cite{LZLL}.
This completes the proof of Theorem \ref{thcp1}.

The proofs of Theorems \ref{thcp2}--\ref{thbkj2} are similar to the proof of Theorem \ref{thcp1}, and thus we omit the details.

\bigskip

\noindent\textbf{Data availability.} No data was used for the research described in the article.

\bigskip

\noindent\textbf{Conflict of interest.} The authors declare that they have no conflict of interest.

\end{document}